\documentclass[a4paper,12pt, reqno]{amsart}
\usepackage[ french, british]{babel}
\usepackage{a4wide}
\usepackage{amsthm}
\usepackage{amssymb}
\usepackage{longtable}

\newtheorem{thm}{Theorem}[section]
\newtheorem{cor}[thm]{Corollary}
\newtheorem{prop}[thm]{Proposition}
\newtheorem{lem}[thm]{Lemma}

\theoremstyle{definition}

\theoremstyle{remark}

\makeatletter
\let\c@equation\c@thm
\makeatother
\numberwithin{equation}{section}

\begin{document}
	
\author{S.Senthamarai Kannan and Pinakinath Saha}
	
\address{Chennai Mathematical Institute, Plot H1, SIPCOT IT Park, 
		Siruseri, Kelambakkam,  603103, India.}
	
\email{kannan@cmi.ac.in.}
	
\address{Chennai Mathematical Institute, Plot H1, SIPCOT IT Park, 
		Siruseri, Kelambakkam, 603103, India.}
\email{pinakinath@cmi.ac.in.}
\title{rigidity of bott-samelson-demazure-hansen variety for $F_{4}$ and $G_{2}$}

\begin{abstract} Let $G$ be a simple algebraic \texttt{}group of adjoint type over $\mathbb{C},$ whose root system is of type $F_{4}.$  Let $T$ be a maximal torus of $G$ and $B$ be a Borel subgroup of $G$ containing $T.$ Let $w$ be an element of Weyl group $W$ and $X(w)$ be the Schubert variety in the flag variety $G/B$ corresponding to $w.$ Let $Z(w, \underline{i})$ be the Bott-Samelson-Demazure-Hansen variety (the desingularization of $X(w)$) corresponding to a reduced expression $\underline{i}$ of $w.$
	 		
In this article, we study the cohomology modules of the tangent bundle on $Z(w_{0}, \underline{i}),$ where $w_{0}$ is the longest element of the Weyl group $W.$ We describe all the reduced expressions of $w_{0}$ in terms of a Coxeter element such that $Z(w_{0}, \underline{i})$ is rigid (see Theorem \ref{theorem 8.1}). Further, if $G$ is of type $G_{2},$ there is no reduced expression $\underline{i}$ of $w_{0}$ for which $Z(w_{0}, \underline{i})$ is rigid (see Theorem \ref{th 8.2}).
\end{abstract}	
%\subjclass[2010]{11A05 (primary); 11R04 (secondary)}
%\keywords{Euclidean rings, number fields, class number, non-Wieferich primes, primitive roots}

\maketitle

\section{introduction}
Let $G$ be a simple algebraic group of adjoint type over  the field $\mathbb{C}$ of 
complex numbers. We fix a maximal torus $T$ of $G$ and 
let $W = N_G(T)/T$ denote the Weyl group of $G$ with respect to $T$.
We denote by $R$ the set of roots of $G$ with respect to $T$ and by $R^{+}\subset R$ a set of positive roots. Let $B^{+}$ be the Borel 
subgroup of $G$ containing $T$ with respect to $R^{+}$. 
Let $w_0$ denote the longest element of the Weyl group $W$. Let $B$ be the 
Borel subgroup of $G$ opposite to $B^+$ determined by $T$, i.e. $B=n_{w_0}B^+n_{w_0}^{-1}$, where $n_{w_0}$ is a representative of $w_0$ in $N_G(T)$.
Note that the 
roots of 
$B$ is the set $R^{-} :=-R^{+}$ of negative roots. We use the notation $\beta<0$ for 
$\beta \in R^{-}$.
Let $S = \{\alpha_1,\ldots,\alpha_n\}$ 
denote the set of all simple roots in $R^{+}$, where $n$ is the rank of $G$. For simplicity of notation, the simple reflection $s_{\alpha_{i}}$ corresponding to a simple root $\alpha_{i}$ is denoted by $s_{i}$. For $w \in W$, let $X(w):=\overline{BwB/B}$ denote the Schubert variety in
$G/B$ corresponding to $w$. Given a reduced expression 
$w=s_{{i_1}}s_{{i_2}}\cdots s_{{i_r}}$ of $w$, with 
the corresponding tuple $\underline i:=(i_1,\ldots,i_r)$, we denote by 
$Z(w,{\underline i})$ the desingularization  of the Schubert variety $X(w)$, 
which is now known as Bott-Samelson-Demazure-Hansen variety. This was 
first introduced by Bott and Samelson in a differential geometric and 
topological context (see \cite{BS}). Demazure in  \cite{De} and 
Hansen in \cite{Han} independently adapted the construction in 
algebro-geometric situation, which explains the reason for the name. 
For the sake of simplicity, we will denote any Bott-Samelson-Demazure-Hansen 
variety by a BSDH-variety. 

The construction of the BSDH-variety $Z(w,{\underline i})$ depends on the choice of the reduced expression $\underline i$ of $w$.  In [5],  the automorphism groups of these varieties were studied. There, the following vanishing results of the tangent bundle $T_{Z(w, \underline i)}$ on $Z(w, \underline i)$ were proved (see \cite[Section 3]{CKP}):\\
(1) $H^j(Z(w, \underline i), T_{Z(w, \underline i)})=0$ for all $j\geq 2$.\\
(2) If $G$ is simply laced, then $H^j(Z(w, \underline i), T_{Z(w, \underline i)})=0$ for all
$j\geq 1$.

As a consequence, it follows that the BSDH-varieties are rigid for simply laced groups and their deformations are  unobstructed in general (see [5, Section 3] ). 
The above vanishing result is independent of the choice of the reduced expression $\underline i$ of 
$w$. While computing the first cohomology module $H^1(Z(w, \underline i),  T_{Z(w, \underline i)})$ for non simply laced group, we observed that this cohomology module very much depend on the choice of a reduced expression $\underline i$ of $w$.

It is a natural question to ask  that for which reduced expressions $\underline i$ of $w$, 
the cohomology module $H^1(Z(w, \underline i), T_{Z(w, \underline i)})$  does vanish ?
In \cite{CK} a partial answer is given to this question for $w=w_0$ when $G=PSp(2n, \mathbb C).$ 
In \cite{KP} a partial answer is given to this question for $w=w_0$ when $G=PSO(2n+1, \mathbb C).$ 
In this article, we give partial answers to this question for $w=w_0$ when $G$ is of type $F_{4},G_{2}.$     

Recall that a Coxeter element is an element of the Weyl group 
having a reduced expression of the form $s_{i_{1}}s_{i_{2}} \cdots s_{i_{n}}$
such that $i_{j}\neq i_{l}$ whenever $j\neq l$ (see \cite[p.56, Section 4.4]{Hum3}).
Note that for any Coxeter element $c\in W$, the Weyl group corresponding to the root system of type $F_{4}$ (respectively, $G_{2}$) there is a decreasing sequence of integers $4\geq a_1> a_2 > \ldots > a_k=1$
(respectively, $2\geq a_1>\ldots>a_{k}=1$) such that $c=\prod\limits_{j=1}^{k}[a_j, a_{j-1}-1]$, where $a_0:=5$ (respectively, $a_{0}:=3$), $[i, j]:=s_{i}s_{i+1}\cdots s_{j}$ for $i\leq j$.

In this paper we prove the following theorems.
\begin{thm}\label{th 1.1} Assume that $G$ is of type $F_{4}.$ Then,
$H^j(Z(w_{0},\underline{i}), T_{(w_{0}, \underline{i})})=0$ for all $j\ge 1$ if and only if $a_{1}\neq3$ or $a_{2}\neq 2.$	
\end{thm} 
\begin{thm}\label{th 1.2} Assume that $G$ is of type $G_{2}.$ Then,
	$H^1(Z(w_{0},\underline{i_{r}}), T_{(w_{0}, \underline{i_{r}})})\neq0$ for $r=1,2.$ 
\end{thm}

By the above results, we conclude that if $G$ is of type $F_{4}$ (respectively, $G_{2}$) and $\underline i=(\underline i^1, \underline i^2,\underline i^3, \underline i^4, \underline i^5, \underline i^6)$ (respectively, $\underline i=(\underline i^1, \underline i^2)$) is a reduced expression of $w_0$ as above, then the BSDH-variety $Z(w_0, \underline i)$ is rigid (respectively, non rigid).

The organization of the paper is as follows:
In Section 2, we recall some preliminaries on BSDH-varieties. We deal with $G$ which is of type $F_{4},$ in the later sections 3, 4, 5, 6 and 7. In Section 3, we prove $H^1(w, \alpha_{j})=0$ for $j=1,2$ and $w\in W.$
In Section 4 (respectively, Section 5) we compute the weight spaces of $H^0$ (respectively, $H^1$) of the relative tangent bundle of BSDH-varieties
associated to some elements of the Weyl group. 
In Section 6, we prove surjectivity results of some maps from cohomology module of tangent bundle on BSDH variety to cohomology module of relative tangent bundle on BSDH variety.
In Section 7, we prove Theorem \ref{th 1.1} using the results from the previous sections. In Section 8, we prove Theorem \ref{th 1.2}.

\section{preliminaries}
In this section, we set up some notation and preliminaries. We refer to \cite{Bri}, \cite{Hum1}, \cite{Hum2}, \cite{Jan} for preliminaries in algebraic groups and Lie algebras. 

Let $G$ be a simple algebraic group of adjoint type over $\mathbb{C}$ and $T$ be a maximal torus of
$G$.  Let $W=N_{G}(T)/T$ denote the Weyl group of $G$ with respect to $T$ and we denote 
the set of roots of $G$ with respect to $T$ by $R$. Let $B^{+}$ be a  Borel subgroup of $G$ 
containing $T$. Let $B$ be the Borel subgroup of $G$ opposite to $B^{+}$ determined by $T$. 
That is, $B=n_{0}B^{+}n_{0}^{-1}$, where $n_{0}$ is a representative in $N_{G}(T)$ of the longest     element $w_{0}$ of $W$. Let  $R^{+}\subset R$ be 
the set of positive roots of $G$ with respect to the Borel subgroup $B^{+}$. Note that the set of 
roots of $B$ is equal to the set $R^{-} :=-R^{+}$ of negative roots.

Let $S = \{\alpha_1,\ldots,\alpha_n\}$ denote the set of simple roots in
$R^{+}.$ For $\beta \in R^{+},$ we also use the notation $\beta > 0$.  
The simple reflection in  $W$ corresponding to $\alpha_i$ is denoted
by $s_{i}$.

Let $\mathfrak{g}$ be the Lie algebra of $G$. 
Let $\mathfrak{h}\subset \mathfrak{g}$ be the Lie algebra of $T$ and  $\mathfrak{b}\subset \mathfrak{g}$ be the Lie algebra of $B$. Let $X(T)$ denote the group of all characters of $T$. 
We have $X(T)\otimes \mathbb{R}=Hom_{\mathbb{R}}(\mathfrak{h}_{\mathbb{R}}, \mathbb{R})$, the dual of the real form of $\mathfrak{h}$. The positive definite 
$W$-invariant form on $Hom_{\mathbb{R}}(\mathfrak{h}_{\mathbb{R}}, \mathbb{R})$ 
induced by the Killing form of $\mathfrak{g}$ is denoted by $(~,~)$. 
We use the notation $\left< ~,~ \right>$ to
denote $\langle \mu, \alpha \rangle  = \frac{2(\mu,
	\alpha)}{(\alpha,\alpha)}$,  for every  $\mu\in X(T)\otimes \mathbb{R}$ and $\alpha\in R$.  
We denote by $X(T)^+$ the set of dominant characters of 
$T$ with respect to $B^{+}$. Let $\rho$ denote the half sum of all 
positive roots of $G$ with respect to $T$ and $B^{+}.$
For any simple root $\alpha$, we denote the fundamental weight
corresponding to $\alpha$  by $\omega_{\alpha}.$  For $1\leq i \leq n,$ let $h(\alpha_{i})\in \mathfrak{h}$ be the fundamental coweight corresponding to $\alpha_{i}.$ That is ; $\alpha_{i}(h(\alpha_{j}))=\delta_{ij},$ where $\delta_{ij}$ is Kronecker delta.

For a simple root $\alpha \in S,$ we denote by $n_{\alpha},$ a representative of $s_{\alpha}$ in $N_{G}(T),$ and  $P_{\alpha}$ the minimal parabolic subgroup of $G$ containing $B$ and $n_{\alpha}.$ We recall that the BSDH-variety corresponds to a reduced expression $\underline{i}$ of $w=s_{i_{1}}s_{i_{2}}\cdots s_{i_{r}}$ defined by
$$Z(w,\underline{i})=\frac{P_{\alpha_{i_{1}}}\times P_{\alpha_{i_{2}}}\times\cdots \times P_{\alpha_{i_{r}}}}{B\times B\times \cdots \times B}$$

where the action of $B\times B\times \cdots \times B$ on $P_{\alpha_{i_{1}}}\times P_{\alpha_{i_{2}}}\times\cdots\times P_{\alpha_{i_{r}}}$ is given by

$(p_{1}, p_{2}, \dots ,p_{r})(b_{1}, b_{2},\dots, b_{r})=(p_{1}\cdot b_{1}, b_{1}^{-1} \cdot  p_{2}\cdot b_{2},\dots, b_{r-1}^{-1}\cdot p_{r}\cdot b_{r}),$ $p_{j}\in P_{\alpha_{i_{j}}}$, $b_{j}\in B$ for $1\le j\le r,$ and $\underline{i}=(i_{1},i_{2},\dots,i_{r})$ (see \cite[Definition 1, p.73]{De}, \cite[Definition 2.2.1, p.64]{Bri}).

We note that for each reduced expression $\underline{i}$ of $w,$ $Z(w,\underline{i})$ is a smooth projective variety. We denote by $\phi_{w},$ the natural birational surjective morphism from $Z(w, \underline{i})$ to $X(w).$

Let $f_{r}:Z(w, \underline{i})\longrightarrow Z(ws_{i_{r}}, \underline{i'})$ denote the map induced by the projection

$P_{\alpha_{i_{1}}}\times P_{\alpha_{i_{2}}}\times\cdots\times P_{\alpha_{i_{r}}}\longrightarrow P_{\alpha_{i_{1}}}\times P_{\alpha_{i_{2}}}\times \cdots \times P_{\alpha_{i_{r-1}}},$ where $\underline{i'}=(i_{1},i_{2},\dots, i_{r-1}).$ Then we observe that $f_{r}$ is a $P_{\alpha_{i_{r}}}/B\simeq \mathbb{P}^1$-fibration.

For a $B$-module $V,$ let $\mathcal{L}(w, V)$ denote the restriction of the associted homogeneous vector bundle on $G/B$ to $X(w).$ By abuse of notation, we denote the pull back of $\mathcal{L}(w, V)$ via $\phi_{w}$ to $Z(w, \underline{i})$ also by  $\mathcal{L}(w, V),$ when there is no confusion. Since for any $B$-module $V$ the vector bundle $\mathcal{L}(w, V)$ on $Z(w, \underline{i})$ is the pull back of the homogeneous vector bundle from $X(w),$ we conclude that the cohomology modules

$$H^j(Z(w,\underline{i}), \mathcal{L}(w, V))\simeq H^j(X(w), \mathcal{L}(w, V))$$

for all $j\ge 0$ (see \cite[Theorem 3.3.4(b)]{Bri}), are independent of choice of reduced expression $\underline{i}.$ Hence we denote $H^j(Z(w, \underline{i}), \mathcal{L}(w, V))$ by $H^j(w, V).$ In particular, if $\lambda$ is character of $B,$ then we denote the cohomology modules $H^j(Z(w,\underline{i}), \mathcal{L}_{\lambda})$ by $H^j(w, \lambda).$

We recall the following short exact sequence of $B$-modules from \cite{CKP}, we call it $SES.$

If $l(w)=l(s_{\gamma}w)+1,$ $\gamma \in S,$ then we have 
\begin{enumerate}
\item [(1)] $H^0(w, V)\simeq H^0(s_{\gamma}, H^0(s_{\gamma}w, V)).$
	
\item[(2)] $0\rightarrow H^1(s_{\gamma}, H^0(s_{\gamma}w, V))\rightarrow H^1(w, V)\rightarrow  H^0(s_{\gamma}, H^1(s_{\gamma}w, V))\rightarrow0.$
\end{enumerate}

Let $\alpha$ be a simple root and $\lambda\in X(T)$ be such that $\langle \lambda, \alpha\rangle\ge0.$ Let $\mathbb{C}_{\lambda}$ denote one dimensional $B$-module associated to $\lambda.$ Here, we recall the following result due to Demazure \cite[p.271]{Dem} on short exact sequence of $B$-modules:

\begin{lem} \label{lemma 1.1}
	Let $\alpha$ be a simple root and $\lambda\in X(T)$ be such that $\langle \lambda , \alpha \rangle \ge 0.$ Let $ev:H^0(s_{\alpha}, \lambda)\longrightarrow \mathbb{C}_{\lambda}$ be the evaluation map. Then we have 
\begin{enumerate}
\item[(1)] If $\langle \lambda, \alpha \rangle =0,$ then $H^0(s_{\alpha}, \lambda)\simeq \mathbb{C}_{\lambda}.$

\item[(2)] If $\langle \lambda , \alpha \rangle \ge 1,$ then $\mathbb{C}_{s_{\alpha}(\lambda)}\hookrightarrow H^0(s_{\alpha}, \lambda) $, and there is a short exact sequene of $B$-modules:

$$0\rightarrow H^0(s_{\alpha}, \lambda-\alpha)\longrightarrow H^0(s_{\alpha}, \lambda)/\mathbb{C}_{s_{\alpha}(\lambda)}\longrightarrow \mathbb{C}_{\lambda}\rightarrow 0.$$ Further more, $H^{0}(s_{\alpha}, \lambda- \alpha)=0$ when $\langle\lambda , \alpha \rangle=1.$

\item[(3)] Let $n=\langle \lambda ,\alpha \rangle.$ As a $B$-module, $H^0(s_{\alpha}, \lambda)$ has a composition series 

$$0\subseteq V_{n}\subseteq V_{n-1}\subseteq \dots \subseteq V_{0}=H^0(s_{\alpha},\lambda)$$
		
such that  $V_{i}/V_{i+1}\simeq \mathbb{C}_{\lambda - i\alpha}$ for $i=0,1,\dots,n-1$ and $V_{n}=\mathbb{C}_{s_{\alpha}(\lambda)}.$
		
\end{enumerate}
\end{lem}

We define the dot action by $w\cdot \lambda= w(\lambda + \rho)-\rho,$ where $\rho$ is the half sum of positive roots. As a consequence of exact sequences of Lemma \ref{lemma 1.1}, we can prove the following.

Let $w\in W$, $\alpha$ be a simple root, and set $v=ws_{\alpha}$.
\begin{lem} \label{lemma 1.2}
	If $l(w) = l(v)+1$, then we have
	\begin{enumerate}
		
		\item  If $\langle \lambda , \alpha \rangle \geq 0$, then 
		$H^{j}(w , \lambda) = H^{j}(v, H^0({s_\alpha, \lambda}) )$ for all $j\geq 0$.
		
		\item  If $\langle \lambda ,\alpha \rangle \geq 0$, then $H^{j}(w , \lambda ) = H^{j+1}(w , s_{\alpha}\cdot \lambda)$ for all $j\geq 0$.
		
		\item If $\langle \lambda , \alpha \rangle  \leq -2$, then $H^{j+1}(w , \lambda ) = H^{j}(w ,s_{\alpha}\cdot \lambda)$ for all $j\geq 0$. 
		
		\item If $\langle \lambda , \alpha \rangle  = -1$, then $H^{j}( w ,\lambda)$ vanishes for every $j\geq 0$.
		
	\end{enumerate}
	
\end{lem}

The following consequence of Lemma \ref {lemma 1.2} will be used to compute the cohomology modules in this paper.
Now onwards we will denote the Levi subgroup of $P_{\alpha}$($\alpha \in S$) containing $T$ by $L_{\alpha}$ and the subgroup $L_{\alpha}\cap B$ by $B_{\alpha}.$ Let $\pi: \tilde{G}\longrightarrow G$ be the universal cover. Let $\tilde{L}_{\alpha}$ (respectively, $\tilde{B_{\alpha}}$) be the inverse image of $L_{\alpha}$ (respectively, $B_{\alpha}$).
\begin{lem}\label{lemma1.3}
	Let $V$ be an irreducible  $L_{\alpha}$-module. Let $\lambda$
	be a character of $B_{\alpha}$. Then we have 
	\begin{enumerate}
		\item As $L_{\alpha}$-modules, $H^j(L_{\alpha}/B_{\alpha}, V \otimes \mathbb C_{\lambda})\simeq V \otimes
		H^j(L_{\alpha}/B_{\alpha}, \mathbb C_{\lambda}).$
		\item If
		$\langle \lambda , \alpha \rangle \geq 0$, then 
		$H^{0}(L_{\alpha}/B_{\alpha} , V\otimes \mathbb{C}_{\lambda})$ 
		is isomorphic as an $L_{\alpha}$-module to the tensor product of $V$ and 
		$H^{0}(L_{\alpha}/B_{\alpha} , \mathbb{C}_{\lambda})$. Further, we have 
		$H^{j}(L_{\alpha}/B_{\alpha} , V\otimes \mathbb{C}_{\lambda}) =0$ 
		for every $j\geq 1$.
		
		\item If
		$\langle \lambda , \alpha \rangle  \leq -2$, then 
		$H^{0}(L_{\alpha}/B_{\alpha} , V\otimes \mathbb{C}_{\lambda})=0$, 
		and $H^{1}(L_{\alpha}/B_{\alpha} , V\otimes \mathbb{C}_{\lambda})$
		is isomorphic to the tensor product of  $V$ and $H^{0}(L_{\alpha}/B_{\alpha} , 
		\mathbb{C}_{s_{\alpha}\cdot\lambda})$. 
		\item If $\langle \lambda , \alpha \rangle  = -1$, then 
		$H^{j}( L_{\alpha}/B_{\alpha} , V\otimes \mathbb{C}_{\lambda}) =0$ 
		for every $j\geq 0$.
	\end{enumerate}
\end{lem}

\begin{proof} Proof (1):
	By \cite[Proposition 4.8, p.53, I]{Jan} and \cite[Proposition 5.12, p.77, I]{Jan}, 
	for all $j\geq 0$, we have the following isomorphism of  $L_{\alpha}$-modules:
	$$H^j(L_{\alpha}/B_{\alpha}, V \otimes \mathbb C_{\lambda})\simeq V \otimes
	H^j(L_{\alpha}/B_{\alpha}, \mathbb C_{\lambda}).$$ 
	
	Proof of (2), (3) and (4) follows from Lemma \ref{lemma 1.2}  by taking $w=s_{\alpha}$ and 
	the fact that $L_{\alpha}/B_{\alpha} \simeq P_{\alpha}/B$.

\end{proof}

Recall the structure of indecomposable 
$B_{\alpha}$-modules and $\widetilde {B}_{\alpha}$-modules (see \cite[Corollary 9.1, p.130]{BKS}).

\begin{lem}\label{lemma 1.4}
	\begin{enumerate}
		\item
		Any finite dimensional indecomposable $\widetilde{B}_{\alpha}$-module $V$ is isomorphic to 
		$V^{\prime}\otimes \mathbb{C}_{\lambda}$ for some irreducible representation
		$V^{\prime}$ of $\widetilde{L}_{\alpha}$ and for some character $\lambda$ of $\widetilde{B}_{\alpha}$.
		\item
		Any finite dimensional indecomposable $B_{\alpha}$-module $V$ is isomorphic to 
		$V^{\prime}\otimes \mathbb{C}_{\lambda}$ for some irreducible representation
		$V^{\prime}$ of $\widetilde{L}_{\alpha}$ and for some character $\lambda$ of $\widetilde{B}_{\alpha}$.
	\end{enumerate}
\end{lem}

\section{reduced expressions}

Now onwards we will assume that $G$ is of type $F_{4}.$ 
Note that longest element $w_{0}$ of the Weyl group $W$ of $G$ is equal to $-identity.$ We recall the following Proposition from \cite[Proposition 1.3, p.858]{YZ}. We use the notation as in \cite{YZ}.

\begin{prop}\label{Z}
	Let $c\in W$ be a Coxeter element, let $\omega_{i}$ be the fundamental weight corresponding to the simple root $\alpha_{i}.$ Then there exists a least positive integer $h(i, c)$ such that
	$c^{h(i,c)}(\omega_{i})=w_{0}(\omega_{i}).$
\end{prop}

%For $k=1$ we have $c=s_{1}s_{2}s_{3}s_{4},$ and the reduced expression obtained from this is $w_{0}=c^6=(s_{1}s_{2}s_{3}s_{4})^6.$

%$For $k=2$ we have $4\ge a_{1}>a_{2}=1.$

Now we can deduce the following:
\begin{lem}
 Let $c \in W$ be a Coxeter element. Then, we have
\begin{enumerate}
	\item [(1)] $w_{0}=c^{6}.$

\item[(2)] For any sequence $\underline{i}=(\underline{i}^{1}, \underline{i}^{2},\cdots, \underline{i}^{6})$ of reduced expressions of $c;$ the sequence $\underline{i}=(\underline{i}^{1}, \underline{i}^{2},\cdots, \underline{i}^{6})$ is a reduced expression of $w_{0}.$ 
\end{enumerate}

\end{lem}

\begin{proof}
Proof of (1): Let $\eta : S\longrightarrow S$ be the involution of $S$ defined by $i\rightarrow i^*,$ where $i^*$ is given by $\omega_{i^*} = -w_{0}(\omega_{i}).$ Since $G$ is of type $F_{4},$  $w_{0} = -identity$, and hence $\omega_{i^*} =\omega_{i}$ for every $i.$ Therefore, we have $i=i^*$ for every $i.$
Let $h$ be the Coxeter number. By \cite[Proposition 1.7]{YZ}, we have $h(i,c)+ h(i^*,c)= h.$ Since $h= 2|R^+|/4$ (see \cite[Proposition 3.18]{HUM}) and $i=i^*,$ we have $h(i,c)= h/2=6,$ as $|R^+|=24.$ 
By Proposition \ref{Z}, we have $c^6(\omega_{i})= -\omega_{i}$ for all $1 \le i \le 4.$ Since $\{\omega_{i} :
	1 \le i \le 4\}$ forms an  $\mathbb{R}$-basis of $X(T)\otimes \mathbb{R},$ it follows that $c^6 = -identity.$ Hence, we have $w_{0}=c^6.$
The assertion (2) follows from the fact that $l(c)= 4$ and $l(w_{0})=|R^+|=24.$ (see \cite[p.66, Table 1]{Hum1}).
\end{proof}

\begin{lem}\label{lem 3.3}
	Let $v\in W$  and $\alpha \in S$. Then $H^{1}(s_{j},H^0(v,\alpha))=0$ for $j=1,2.$
	
\end{lem}
\begin{proof}
If $H^{1}(s_{j},H^0(v,\alpha))_{\mu}\neq0,$ then there exists an indecomposable $\tilde{L}_{\alpha_{j}}$-summand $V$ of $H^0(v,\alpha)$ such that $H^1(s_{j}, V)_{\mu}\neq 0.$ By Lemma 2.4, we have $V\simeq V'\otimes \mathbb{C}_{\lambda} $ for some character $\lambda$ of $\tilde{B}_{\alpha_{j}}$ and for some irreducible $\tilde{L}_{\alpha_{j}}$-module $V'.$
Since $H^1(s_{j}, V)_{\mu}\neq 0,$ from Lemma \ref {lemma1.3}(3) we  have $\langle \lambda, \alpha_{j} \rangle\le -2$. If $\alpha$ is a short root, then $H^1(w, \alpha)=0$ for all $w\in W$ (see \cite[Corollary 5.6, p.778]{Ka}). Hence we may assume that $\alpha$ is a long root. Then there exists $w\in W$ such that $w(\alpha)=\alpha_{0}$. Thus $H^0(v, \alpha)\subseteq H^0(vw, \alpha_{0}).$ Again, since $\alpha_{0}$ is highest long root,  $H^0(w_{0}, \alpha_{0})=\mathfrak{g}\longrightarrow H^0(vw, \alpha_{0})$ is surjective. Let $\mu'$ be the lowest weight of $V.$ Then by the above argument $\mu'$ is a root. Therefore we have $\mu'=\mu_{1}+\lambda,$ where $\mu_{1}$ is the lowest weight of $V'.$ Hence, we have 
$\langle \mu', \alpha_{j} \rangle\le -2$. Since $\alpha_{j}$ is a long root and $\mu'$ is a root, we have $\langle \mu', \alpha_{j} \rangle =-1 , 0, 1.$  This is a contradiction. Thus we have $H^{1}(s_{j},H^0(v,\alpha))_{\mu}=0.$ 
\end{proof}

\section{cohomology modules $H^0(w, \alpha_{i})$}
Let $w_{r}=(s_{1}s_{2}s_{3}s_{4})^{r}s_{1}s_{2}$ for $1\le r\le 5.$ In this section we compute various cohomology modules $H^0(w, \alpha_{i})$ for some elements $w\in W$ and $i=2,3.$
\begin{lem}\label{lem 4.1}
\item[(1)] $H^0(w_{3}, \alpha_{2})=0.$
\item[(2)] $H^0(w_{r}, \alpha_{2})=0$ for $r=4,5.$
\end{lem}
\begin{proof}
 We have $w_{3}=[1, 4]^{3} 12.$
By using SES we have 

$H^0(s_{1}s_{2}, \alpha_{2})=\mathbb{C}h(\alpha_{2})\oplus \mathbb{C}_{-\alpha_{2}}\oplus \mathbb{C}_{-(\alpha_{1} + \alpha_{2})}.$

Since $\langle \alpha_{2}, \alpha_{4}\rangle =0,$ by using SES we have
\begin{center}
$H^0(s_{4}s_{1}s_{2}, \alpha_{2})=H^0(s_{1}s_{2}, \alpha_{2}).$
\end{center}

Since $\langle -\alpha_{2}, \alpha_{3}\rangle =2$ and $\langle -(\alpha_{1}+\alpha_{2}), \alpha_{3}\rangle =2,$ then by using SES we have 

$H^0(s_{3}s_{4}s_{1}s_{2}, \alpha_{2})$=$\mathbb{C}h(\alpha_{2})\oplus (\mathbb{C}_{-\alpha_{2}}\oplus \mathbb{C}_{-(\alpha_{2} + \alpha_{3})} \oplus \mathbb{C}_{-(\alpha_{2} + 2\alpha_{3})}) \oplus (\mathbb{C}_{-(\alpha_{1} + \alpha_{2})}\oplus\mathbb{C}_{-(\alpha_{1}+ \alpha_{2} +\alpha_{3})} \oplus\mathbb{C}_{-(\alpha_{1} + \alpha_{2} + 2\alpha_{3})} ).\hspace{12cm}(4.1.1)$

Since $\mathbb{C}h(\alpha_{2})\oplus \mathbb{C}_{-\alpha_{2}}$ is indecomposable two dimensional $\tilde{B}_{\alpha_{2}}$-module, by Lemma \ref{lemma 1.4} we have $\mathbb{C}h(\alpha_{2})\oplus \mathbb{C}_{-\alpha_{2}}=V \otimes \mathbb{C}_{-\omega_{2}},$ where $V$ is the standard two dimensional irreducible $\tilde{L}_{\alpha_{2}}$-module.

Thus by Lemma \ref{lemma1.3}(4), we have $H^0(\tilde{L}_{\alpha_{2}}/\tilde{B}_{\alpha_{2}}, \mathbb{C}h(\alpha_{2})\oplus \mathbb{C}_{-\alpha_{2}})=0.$

Since $\langle -(\alpha_{2}+ \alpha_{3}), \alpha_{2}\rangle=-1,$ $\langle -(\alpha_{1}+\alpha_{2}), \alpha_{2}\rangle =-1,$ by Lemma \ref {lemma 1.2}(4) we have $H^0(\tilde{L}_{\alpha_{2}}/\tilde{B}_{\alpha_{2}}, \mathbb{C}_{-(\alpha_{2}+\alpha_{3})})=0$ and $H^0(\tilde{L}_{\alpha_{2}}/\tilde{B}_{\alpha_{2}}, \mathbb{C}_{-(\alpha_{1}+\alpha_{2})})=0.$

Since $\langle -(\alpha_{2}+ 2\alpha_{3}), \alpha_{2}\rangle=0,$ $\langle -(\alpha_{1}+\alpha_{2}+\alpha_{3}), \alpha_{2}\rangle =0,$ by Lemma \ref{lemma1.3}(2) we have 
\begin{center}
$H^0(\tilde{L}_{\alpha_{2}}/\tilde{B}_{\alpha_{2}}, \mathbb{C}_{-(\alpha_{2}+2\alpha_{3})})=\mathbb{C}_{-(\alpha_{2} + 2\alpha_{3})}$
\end{center}
and
\begin{center}
$H^0(\tilde{L}_{\alpha_{2}}/\tilde{B}_{\alpha_{2}}, \mathbb{C}_{-(\alpha_{1}+\alpha_{2}+\alpha_{3})})=\mathbb{C}_{-(\alpha_{1}+\alpha_{2}+\alpha_{3})}.$ 
\end{center}
Since $\langle-(\alpha_{1}+\alpha_{2}+2\alpha_{3}), \alpha_{2} \rangle =1,$ we have 
\begin{center}
$H^0(\tilde{L}_{\alpha_{2}}/\tilde{B}_{\alpha_{2}}, \mathbb{C}_{-(\alpha_{1}+\alpha_{2}+2\alpha_{3})})=\mathbb{C}_{-(\alpha_{1}+\alpha_{2}+2\alpha_{3})}\oplus \mathbb{C}_{-(\alpha_{1}+2\alpha_{2}+2\alpha_{3})}.$ 
\end{center}
Thus we have
\begin{center} $H^0(s_{2}s_{3}s_{4}s_{1}s_{2},\alpha_{2})=\mathbb{C}_{-(\alpha_{2} + 2\alpha_{3})}\oplus \mathbb{C}_{-(\alpha_{1}+\alpha_{2}+\alpha_{3})}\oplus \mathbb{C}_{-(\alpha_{1}+\alpha_{2}+2\alpha_{3})}\oplus \mathbb{C}_{-(\alpha_{1}+2\alpha_{2}+2\alpha_{3})}.\hspace{11cm}(4.1.2)$
\end{center}

Since $\langle -(\alpha_{1}+\alpha_{2}+\alpha_{3}), \alpha_{1}\rangle =-1,$ $\langle -(\alpha_{1}+\alpha_{2}+2\alpha_{3}), \alpha_{1}\rangle =-1,$ by using Lemma \ref {lemma1.3}(4) we have 
\begin{center}
$H^0(\tilde{L}_{\alpha_{1}}/\tilde{B}_{\alpha_{1}}, \mathbb{C}_{-(\alpha_{1}+\alpha_{2}+\alpha_{3})})=0,$
\end{center}
and 
\begin{center}
$H^0(\tilde{L}_{\alpha_{1}}/\tilde{B}_{\alpha_{1}}, \mathbb{C}_{-(\alpha_{1}+\alpha_{2}+2\alpha_{3})})=0.$
\end{center}
Since $\langle -(\alpha_{1}+2\alpha_{2}+2\alpha_{3}), \alpha_{1}\rangle=0,$ by using Lemma \ref{lemma1.3}(2) we have 
\begin{center}
$H^0(\tilde{L}_{\alpha_{1}}/\tilde{B}_{\alpha_{1}}, \mathbb{C}_{-(\alpha_{1}+2\alpha_{2}+2\alpha_{3})})=\mathbb{C}_{-(\alpha_{1}+2\alpha_{2}+2\alpha_{3})}.$
\end{center}

Since $\langle -(\alpha_{2}+2\alpha_{3}), \alpha_{1}\rangle=1,$ by using Lemma \ref{lemma1.3}(2) we have 
\begin{center}
$H^0(\tilde{L}_{\alpha_{1}}/\tilde{B}_{\alpha_{1}}, \mathbb{C}_{-(\alpha_{2}+2\alpha_{3})})=\mathbb{C}_{-(\alpha_{2}+2\alpha_{3})} \oplus \mathbb{C}_{-(\alpha_{1}+\alpha_{2}+2\alpha_{3})}.$
\end{center}
Therefore we have
\begin{center}
$H^0(w_{1},\alpha_{2})$=$\mathbb{C}_{-(\alpha_{1}+2\alpha_{2}+2\alpha_{3})}\oplus \mathbb{C}_{-(\alpha_{2}+2\alpha_{3})} \oplus \mathbb{C}_{-(\alpha_{1}+\alpha_{2}+2\alpha_{3})}.\hspace{12cm}(4.1.3)$
\end{center}

 By using SES we have 
\begin{center}
	$H^0(w_{3}, \alpha_{2})=H^0([1,4]^2, H^0(w_{1}, \alpha_{2})).$
\end{center}

Note that the computations of the module $H^0([1,4]^2, H^0(w_{1}, \alpha_{2}))$ is independent of the choice of a reduced expression of $[1, 4]^2.$ 
We consider the reduced expression $s_{2}s_{1}s_{2}s_{3}s_{2}s_{3}s_{4}s_{3},$ of $[1, 4]^2$ to compute $H^0([1,4]^2, H^0(w_{1}, \alpha_{2})).$

Since  $\langle -(\alpha_{2}+2\alpha_{3}), \alpha_{3}\rangle=-2,$ $\langle -(\alpha_{1}+\alpha_{2}+2\alpha_{3}), \alpha_{3}\rangle=-2,$  by using Lemma \ref {lemma1.3}(3) we have 
\begin{center}
$H^0(\tilde{L}_{\alpha_{3}}/\tilde{B}_{\alpha_{3}}, \mathbb{C}_{-(\alpha_{2}+2\alpha_{3})})=0$ 
\end{center}
and
\begin{center}
$H^0(\tilde{L}_{\alpha_{3}}/\tilde{B}_{\alpha_{3}}, \mathbb{C}_{-(\alpha_{1}+\alpha_{2}+2\alpha_{3})})=0.$ 
\end{center}

Since $\langle -(\alpha_{1}+2\alpha_{2}+2\alpha_{3}), \alpha_{3}\rangle=0,$  by using Lemma \ref {lemma1.3}(2) we have 
\begin{center}
	$H^0(\tilde{L}_{\alpha_{3}}/\tilde{B}_{\alpha_{3}}, \mathbb{C}_{-(\alpha_{1}+2\alpha_{2}+2\alpha_{3})})=\mathbb{C}_{-(\alpha_{1}+2\alpha_{2}+2\alpha_{3})}.$ 
\end{center}
Thus from the above discussion we have 

\begin{center}
$H^0(s_{3}w_{1}, \alpha_{2})=\mathbb{C}_{-(\alpha_{1}+2\alpha_{2}+2\alpha_{3})}.$
\end{center}

Since $\langle -(\alpha_{1}+ 2\alpha_{2}+2\alpha_{3}), \alpha_{4}\rangle=2,$ by using SES and Lemma \ref{lemma1.3}(2) we have

\begin{center}
	$H^0(s_{4}s_{3}w_{1}, \alpha_{2})=\mathbb{C}_{-(\alpha_{1}+2\alpha_{2}+2\alpha_{3})}\oplus \mathbb{C}_{-(\alpha_{1}+2\alpha_{2}+2\alpha_{3}+\alpha_{4})} \oplus \mathbb{C}_{-(\alpha_{1}+2\alpha_{2}+2\alpha_{3}+2\alpha_{4})}.$
	
\end{center}
Since $\langle-(\alpha_{1}+2\alpha_{2}+2\alpha_{3}), \alpha_{3} \rangle=0,$ $\langle-(\alpha_{1}+2\alpha_{2}+2\alpha_{3}+\alpha_{4}), \alpha_{3} \rangle=1,$ $\langle-(\alpha_{1}+2\alpha_{2}+2\alpha_{3}+2\alpha_{4}), \alpha_{3} \rangle=2,$ by using Lemma \ref{lemma1.3}(2) we have 

\begin{center}
	$H^0(s_{3}s_{4}s_{3}w_{1}, \alpha_{2})=\mathbb{C}_{-(\alpha_{1}+2\alpha_{2}+2\alpha_{3})}\oplus \mathbb{C}_{-(\alpha_{1}+2\alpha_{2}+2\alpha_{3}+\alpha_{4})} \oplus \mathbb{C}_{-(\alpha_{1}+2\alpha_{2}+3\alpha_{3}+\alpha_{4})}\oplus \mathbb{C}_{-(\alpha_{1}+2\alpha_{2}+2\alpha_{3}+2\alpha_{4})}\oplus \mathbb{C}_{-(\alpha_{1}+2\alpha_{2}+3\alpha_{3}+2\alpha_{4})}\oplus \mathbb{C}_{-(\alpha_{1}+2\alpha_{2}+4\alpha_{3}+2\alpha_{4})}.$
\end{center}
Since $\langle-(\alpha_{1}+2\alpha_{2}+2\alpha_{3}), \alpha_{2} \rangle =-1,$ $\langle-(\alpha_{1}+2\alpha_{2}+2\alpha_{3}+\alpha_{4}), \alpha_{2} \rangle =-1,$ $\langle-(\alpha_{1}+2\alpha_{2}+2\alpha_{3}+2\alpha_{4}), \alpha_{2} \rangle =-1,$ 
$\langle-(\alpha_{1}+2\alpha_{2}+3\alpha_{3}+\alpha_{4}), \alpha_{2} \rangle =0,$
$\langle-(\alpha_{1}+2\alpha_{2}+3\alpha_{3}+2\alpha_{4}), \alpha_{2} \rangle =0,$ 
$\langle-(\alpha_{1}+2\alpha_{2}+4\alpha_{3}+2\alpha_{4}), \alpha_{2} \rangle =1,$ by using Lemma \ref{lemma1.3}(2), Lemma \ref{lemma1.3}(4) we have 

\begin{center}
	$H^0(s_{2}s_{3}s_{4}s_{3}w_{1}, \alpha_{2})$=$ \mathbb{C}_{-(\alpha_{1}+2\alpha_{2}+3\alpha_{3}+\alpha_{4})}\oplus \mathbb{C}_{-(\alpha_{1}+2\alpha_{2}+3\alpha_{3}+2\alpha_{4})}\oplus \mathbb{C}_{-(\alpha_{1}+2\alpha_{2}+4\alpha_{3}+2\alpha_{4})}\oplus \mathbb{C}_{-(\alpha_{1}+3\alpha_{2}+4\alpha_{3}+2\alpha_{4})}.$
\end{center}

 Since $\mathbb{C}_{-(\alpha_{1}+2\alpha_{2}+3\alpha_{3}+2\alpha_{4})}\oplus \mathbb{C}_{-(\alpha_{1}+2\alpha_{2}+4\alpha_{3}+2\alpha_{4})}$
is two dimensional indecomposable $\tilde{B}_{\alpha_{3}}$-module, thus by Lemma \ref{lemma 1.4}(1) we have

$\mathbb{C}_{-(\alpha_{1}+2\alpha_{2}+3\alpha_{3}+2\alpha_{4})}\oplus \mathbb{C}_{-(\alpha_{1}+2\alpha_{2}+4\alpha_{3}+2\alpha_{4})}=V \otimes \mathbb{C}_{-\omega_{3}},$
where $V$ is the standard two dimensional irreducible $\tilde{L}_{\alpha_{3}}$-module.

Thus by Lemma \ref{lemma1.3}(4) we have 

$H^0(\tilde{L}_{\alpha_{3}}/\tilde{B}_{\alpha_{3}}, \mathbb{C}_{-(\alpha_{1}+2\alpha_{2}+3\alpha_{3}+2\alpha_{4})}\oplus \mathbb{C}_{-(\alpha_{1}+2\alpha_{2}+4\alpha_{3}+2\alpha_{4})})=0.$

Since
$\langle-(\alpha_{1}+2\alpha_{2}+3\alpha_{3}+\alpha_{4}), \alpha_{3} \rangle =-1,$ and  $\langle-(\alpha_{1}+3\alpha_{2}+4\alpha_{3}+2\alpha_{4}), \alpha_{3} \rangle =0,$ by Lemma \ref{lemma1.3}(2), Lemma \ref{lemma1.3}(4) we have
$H^0(s_{3}s_{2}s_{3}s_{4}s_{3}w_{1}, \alpha_{2})= \mathbb{C}_{-(\alpha_{1}+3\alpha_{2}+4\alpha_{3}+2\alpha_{4})}.$

Since $\langle
-(\alpha_{1}+3\alpha_{2}+4\alpha_{3}+2\alpha_{4}), \alpha_{2}\rangle =-1,$ by Lemma \ref{lemma1.3}(4) we have 

$H^0(s_{2}s_{3}s_{2}s_{3}s_{4}s_{3}w_{1}, \alpha_{2})=0.$

Thus by using SES and Lemma \ref{lemma1.3}(2) we have
$H^0(s_{1}s_{2}s_{3}s_{2}s_{3}s_{4}s_{3}w_{1}, \alpha_{2})=0.$

Again by using SES and Lemma \ref{lemma1.3}(2) we have
$H^0(w_{3}, \alpha_{2})=H^0([1, 4]^2, H^0(w_{1}, \alpha_{2}))=0.$

Proof of (2) follows from (1).
\end{proof}
Recall that $\omega_{4}=\alpha_{1} + 2\alpha_{2} + 3\alpha_{3} + 2\alpha_{4}.$ Now onwards we replace  $\alpha_{1} + 2\alpha_{2} + 3\alpha_{3} + 2\alpha_{4}$ by $\omega_{4}.$ 
\begin{lem}\label{lem 4.2}
\item[(1)] $H^0(w_{2}s_{3}, \alpha_{3})=\mathbb{C}_{-\omega_{4}+\alpha_{4}}.$
	
\item[(2)]  $H^0(w_{3}s_{3},\alpha_{3})=\mathbb{C}_{-\omega_{4}}.$
\end{lem}
\begin{proof}
	Proof of (1): Using SES we have $H^0(s_{3}, \alpha_{3})=\mathbb{C}_{-\alpha_{3}}\oplus \mathbb{C}h(\alpha_{3})\oplus \mathbb{C}_{\alpha_{3}}.$ 
Since $\langle \alpha_{3}, \alpha_{2}\rangle =-1,$ by using SES and Lemma \ref{lemma1.3} we have 
\begin{center}
$H^0(s_{2}s_{3}, \alpha_{3})=\mathbb{C}h(\alpha_{3})\oplus \mathbb{C}_{-\alpha_{3}}\oplus \mathbb{C}_{-(\alpha_{3} + \alpha_{2})}.$
\end{center}
Further, since $\langle \alpha_{3}, \alpha_{1}\rangle =0$ and $\langle -(\alpha_{3} + \alpha_{2}), \alpha_{1} \rangle=1,$ by using SES and Lemma \ref{lemma1.3} we have 
\begin{center}
$H^0(s_{1}s_{2}s_{3}, \alpha_{3})=\mathbb{C}h(\alpha_{3})\oplus \mathbb{C}_{-\alpha_{3}}\oplus \mathbb{C}_{-(\alpha_{2} + \alpha_{3})}\oplus \mathbb{C}_{-(\alpha_{1} + \alpha_{2} + \alpha_{3})}.$
\end{center}
Note that the computations of the module $H^0([1,4]^2, H^0(s_{1}s_{2}s_{3}, \alpha_{3}))$ is independent of the choice of a reduced expression of $[1, 4]^2.$ 
We consider the reduced expression $s_{1}s_{2}s_{1}s_{3}s_{2}s_{3}s_{4}s_{3}$ of $[1, 4]^2$ to compute $H^0([1,4]^2, H^0(s_{1}s_{2}s_{3}, \alpha_{3})).$

Since $\mathbb{C}h(\alpha_{3})\oplus \mathbb{C}_{-\alpha_{3}}$
 is two dimensional $\tilde{B}_{\alpha_{3}}$-module, by Lemma \ref{lemma 1.4}(1) we have 
\begin{center}
$\mathbb{C}h(\alpha_{3})\oplus \mathbb{C}_{-\alpha_{3}}=V\otimes \mathbb{C}_{-\omega_{3}}$ 
\end{center} 
where $V$ is the standard two dimensional irreducible  $\tilde{L}_{\alpha_{3}}$-module.

Thus by using Lemma \ref{lemma1.3}(4) we have 
 \begin{center}
 $H^0(\tilde{L}_{\alpha_{3}}/\tilde{B}_{\alpha_{3}}, \mathbb{C}h(\alpha_{3})\oplus \mathbb{C}_{-\alpha_{3}})=0.$ 
 \end{center} 
 
Since $\langle-(\alpha_{2} + \alpha_{3}), \alpha_{3} \rangle =0,$ $\langle-(\alpha_{1}+\alpha_{2} + \alpha_{3}), \alpha_{3} \rangle =0,$ by Lemma \ref{lemma1.3}(2) we have  
\begin{center}
$H^0(\tilde{L}_{\alpha_{3}}/\tilde{B}_{\alpha_{3}}, \mathbb{C}_{-(\alpha_{2}+\alpha_{3})})=\mathbb{C}_{-(\alpha_{2}+\alpha_{3})}$
\end{center}  
and 
\begin{center}
$H^0(\tilde{L}_{\alpha_{3}}/\tilde{B}_{\alpha_{3}}, \mathbb{C}_{-(\alpha_{1}+\alpha_{2}+\alpha_{3})})=\mathbb{C}_{-(\alpha_{1}+\alpha_{2}+\alpha_{3})}.$ 
\end{center}

Thus from the above discussion we have 
\begin{center}
$H^0(s_{3}s_{1}s_{2}s_{3}, \alpha_{3})= \mathbb{C}_{-(\alpha_{2} + \alpha_{3})}\oplus \mathbb{C}_{-(\alpha_{1} + \alpha_{2} + \alpha_{3})}.$	
\end{center}
Since $\langle -(\alpha_{2} + \alpha_{3}), \alpha_{4}\rangle =1,$ $\langle -(\alpha_{1} + \alpha_{2} + \alpha_{3}), \alpha_{4}\rangle =1,$ by using Lemma \ref{lemma1.3}(2) we have 
\begin{center}
	$H^0(\tilde{L}_{\alpha_{4}}/\tilde{B}_{\alpha_{4}}, \mathbb{C}_{-(\alpha_{2} + \alpha_{3})})=\mathbb{C}_{-(\alpha_{2} + \alpha_{3})}\oplus \mathbb{C}_{-(\alpha_{2} + \alpha_{3}+\alpha_{4})}$
\end{center}
and 

\begin{center}
	$H^0(\tilde{L}_{\alpha_{4}}/\tilde{B}_{\alpha_{4}}, \mathbb{C}_{-(\alpha_{1}+\alpha_{2} + \alpha_{3})})=\mathbb{C}_{-(\alpha_{1}+\alpha_{2} + \alpha_{3})}\oplus \mathbb{C}_{-(\alpha_{1}+\alpha_{2} + \alpha_{3}+\alpha_{4})}.$
\end{center}

Thus from the above discussion we have 

\begin{center}
	$H^0(s_{4}s_{3}s_{1}s_{2}s_{3}, \alpha_{3})= \mathbb{C}_{-(\alpha_{2} + \alpha_{3})}\oplus \mathbb{C}_{-(\alpha_{2} + \alpha_{3}+\alpha_{4})}\oplus \mathbb{C}_{-(\alpha_{1} + \alpha_{2} + \alpha_{3})}\oplus \mathbb{C}_{-(\alpha_{1} + \alpha_{2} + \alpha_{3}+\alpha_{4})}.$	
\end{center} 

Since $\langle -(\alpha_{2}+\alpha_{3}), \alpha_{3} \rangle =0,$ and  $\langle -(\alpha_{1}+\alpha_{2}+\alpha_{3}), \alpha_{3} \rangle =0,$ by using Lemma \ref{lemma1.3}(2) we have 
\begin{center}
	$H^0(\tilde{L}_{\alpha_{3}}/\tilde{B}_{\alpha_{3}}, \mathbb{C}_{-(\alpha_{2}+\alpha_{3})})=\mathbb{C}_{-(\alpha_{2}+\alpha_{3})}$
\end{center} 
and 
\begin{center}
	$H^0(\tilde{L}_{\alpha_{3}}/\tilde{B}_{\alpha_{3}}, \mathbb{C}_{-(\alpha_{1}+\alpha_{2}+\alpha_{3})})=\mathbb{C}_{-(\alpha_{1}+\alpha_{2}+\alpha_{3})}.$
\end{center} 

Since $\langle -(\alpha_{2}+\alpha_{3}+\alpha_{4}), \alpha_{3} \rangle =1,$ and  $\langle -(\alpha_{1}+\alpha_{2}+\alpha_{3}+\alpha_{4}), \alpha_{3} \rangle =1,$ by using Lemma \ref{lemma1.3}(2) we have 
\begin{center}
	$H^0(\tilde{L}_{\alpha_{3}}/\tilde{B}_{\alpha_{3}}, \mathbb{C}_{-(\alpha_{2}+\alpha_{3}+\alpha_{4})})=\mathbb{C}_{-(\alpha_{2}+\alpha_{3}+\alpha_{4})}\oplus \mathbb{C}_{-(\alpha_{2}+2\alpha_{3}+\alpha_{4})}$
\end{center} 
and 
\begin{center}
	$H^0(\tilde{L}_{\alpha_{3}}/\tilde{B}_{\alpha_{3}}, \mathbb{C}_{-(\alpha_{1}+\alpha_{2}+\alpha_{3}+\alpha_{4})})=\mathbb{C}_{-(\alpha_{1}+\alpha_{2}+\alpha_{3}+\alpha_{4})}\oplus \mathbb{C}_{-(\alpha_{1}+\alpha_{2}+2\alpha_{3}+\alpha_{4})}.$
\end{center} 

Thus from the above discussion we have 

\begin{center}
	$H^0(s_{3}s_{4}s_{3}s_{1}s_{2}s_{3}, \alpha_{3})= \mathbb{C}_{-(\alpha_{2} + \alpha_{3})}\oplus \mathbb{C}_{-(\alpha_{2} + \alpha_{3}+\alpha_{4})}\oplus \mathbb{C}_{-(\alpha_{2} + 2\alpha_{3}+\alpha_{4})}\oplus \mathbb{C}_{-(\alpha_{1} + \alpha_{2} + \alpha_{3})}\oplus \mathbb{C}_{-(\alpha_{1} + \alpha_{2} +\alpha_{3}+\alpha_{4})}\oplus \mathbb{C}_{-(\alpha_{1} + \alpha_{2} + 2\alpha_{3}+\alpha_{4})}.$	
\end{center} 
Since $\langle -(\alpha_{2} + \alpha_{3}), \alpha_{2} \rangle =-1,$ $\langle -(\alpha_{2} + \alpha_{3}+\alpha_{4}), \alpha_{2} \rangle =-1,$ by using Lemma \ref{lemma1.3}(4) we have
\begin{center}
	$H^0(\tilde{L}_{\alpha_{2}}/\tilde{B}_{\alpha_{2}}, \mathbb{C}_{-(\alpha_{2} + \alpha_{3})} )=0$
\end{center} 
and 
\begin{center}
	$H^0(\tilde{L}_{\alpha_{2}}/\tilde{B}_{\alpha_{2}}, \mathbb{C}_{-(\alpha_{2} + \alpha_{3}+\alpha_{4})} )=0.$
\end{center} 
Since $\langle -(\alpha_{2}+2\alpha_{3}+\alpha_{4}), \alpha_{2}\rangle =0,$ $\langle -(\alpha_{1}+\alpha_{2}+\alpha_{3}), \alpha_{2}\rangle =0,$ $\langle -(\alpha_{1}+\alpha_{2}+\alpha_{3}+\alpha_{4}), \alpha_{2}\rangle =0,$ and $\langle -(\alpha_{1}+\alpha_{2}+2\alpha_{3}+\alpha_{4}), \alpha_{2}\rangle =1,$ by using Lemma \ref{lemma1.3}(2) we have 
\begin{center}
$H^0(s_{2}s_{3}s_{4}s_{3}s_{1}s_{2}s_{3}, \alpha_{3})$=$  \mathbb{C}_{-(\alpha_{2} + 2\alpha_{3}+\alpha_{4})}\oplus \mathbb{C}_{-(\alpha_{1} + \alpha_{2} + \alpha_{3})}\oplus \mathbb{C}_{-(\alpha_{1} + \alpha_{2} +\alpha_{3}+\alpha_{4})}\oplus \mathbb{C}_{-(\alpha_{1} + \alpha_{2} + 2\alpha_{3}+\alpha_{4})}\oplus \mathbb{C}_{-(\alpha_{1} + 2\alpha_{2} + 2\alpha_{3}+\alpha_{4})}.$	
\end{center} 

Since $\mathbb{C}_{-(\alpha_{1} + \alpha_{2} +\alpha_{3}+\alpha_{4})}\oplus \mathbb{C}_{-(\alpha_{1} + \alpha_{2} + 2\alpha_{3}+\alpha_{4})}$  is the standard two dimensional irreducible $\tilde{L}_{\alpha_{3}}$-module,  $\langle -(\alpha_{1} + \alpha_{2} + \alpha_{3}), \alpha_{3} \rangle =0,$ $\langle -(\alpha_{1} + 2\alpha_{2} + 2\alpha_{3}+\alpha_{4}), \alpha_{3} \rangle =1,$  and $\langle -(\alpha_{2} + 2\alpha_{3}+\alpha_{4}), \alpha_{3} \rangle =-1,$ by using similar arguments as above and using Lemma \ref{lemma1.3}(2), Lemma \ref{lemma1.3}(4) we have 
 
\begin{center}
 	$H^0(s_{3}s_{2}s_{3}s_{4}s_{3}s_{1}s_{2}s_{3}, \alpha_{3})$=$\mathbb{C}_{-(\alpha_{1} + \alpha_{2} + \alpha_{3})}\oplus \mathbb{C}_{-(\alpha_{1} + \alpha_{2} +\alpha_{3}+\alpha_{4})}\oplus \mathbb{C}_{-(\alpha_{1} + \alpha_{2} + 2\alpha_{3}+\alpha_{4})}\oplus \mathbb{C}_{-(\alpha_{1} + 2\alpha_{2} + 2\alpha_{3}+\alpha_{4})}\oplus \oplus \mathbb{C}_{-(\alpha_{1} + 2\alpha_{2} + 3\alpha_{3}+\alpha_{4})}.$	
\end{center} 

Since $\langle -(\alpha_{1} + 2\alpha_{2} + 2\alpha_{3}+\alpha_{4}), \alpha_{1} \rangle =0,$ $\langle -(\alpha_{1} + 2\alpha_{2} + 3\alpha_{3}+\alpha_{4}), \alpha_{1} \rangle =0,$ $\langle -(\alpha_{1}+\alpha_{2} + \alpha_{3}), \alpha_{1} \rangle =-1,$ $\langle -(\alpha_{1} + \alpha_{2} + \alpha_{3}+\alpha_{4}), \alpha_{1} \rangle =-1,$ $\langle -(\alpha_{1} + \alpha_{2} + 2\alpha_{3}+\alpha_{4}), \alpha_{1} \rangle =-1,$  by using similar arguments as above and using Lemma \ref{lemma1.3}(2), Lemma \ref{lemma1.3}(4) we have 
\begin{center}
$H^0(s_{1}s_{3}s_{2}s_{3}s_{4}s_{3}s_{1}s_{2}s_{3}, \alpha_{3})=\mathbb{C}_{-(\alpha_{1} + 2\alpha_{2} + 2\alpha_{3}+\alpha_{4})}\oplus \mathbb{C}_{-(\alpha_{1} + 2\alpha_{2} + 3\alpha_{3}+\alpha_{4})}.$	
\end{center}
Since 
$\langle-(\alpha_{1}+2\alpha_{2}+3\alpha_{3}+\alpha_{4}), \alpha_{2}\rangle=0,$ $\langle-(\alpha_{1}+2\alpha_{2}+2\alpha_{3}+\alpha_{4}), \alpha_{2}\rangle=-1,$ by using Lemma \ref{lemma1.3}(2), Lemma \ref{lemma1.3}(4) we have 
\begin{center}
$H^0(s_{2}s_{1}s_{3}s_{2}s_{3}s_{4}s_{3}s_{1}s_{2}s_{3}, \alpha_{3})= \mathbb{C}_{-(\alpha_{1} + 2\alpha_{2} + 3\alpha_{3}+\alpha_{4})}.$	
\end{center}
Since
$\langle-(\alpha_{1}+2\alpha_{2}+3\alpha_{3}+\alpha_{4}), \alpha_{1}\rangle=0,$ by using Lemma \ref{lemma1.3}(2) we have 
\begin{center}
	$H^0(s_{1}s_{2}s_{1}s_{3}s_{2}s_{3}s_{4}s_{3}s_{1}s_{2}s_{3}, \alpha_{3})=\mathbb{C}_{-(\alpha_{1} + 2\alpha_{2} + 3\alpha_{3}+\alpha_{4})}.$	
\end{center}

Thus we have 
\begin{center}
$H^0(w_{2}s_{3},\alpha_{3})=\mathbb{C}_{-(\alpha_{1} + 2\alpha_{2} + 3\alpha_{3}+\alpha_{4})}=\mathbb{C}_{-\omega_{4}+\alpha_{4}}.$ 
\end{center}

Proof of (2): By the proof of (1) we have 
\begin{center}
$H^0(w_{2}s_{3},\alpha_{3})=\mathbb{C}_{-\omega_{4}+\alpha_{4}}.$ 
\end{center}

Since $\langle -\omega_{4}+\alpha_{4}, \alpha_{4}\rangle =1,$ by Lemma \ref{lemma1.3}(2) we have 
\begin{center}
$H^0(\tilde{L}_{\alpha_{4}}/\tilde{B}_{\alpha_{4}}, \mathbb{C}_{-\omega_{4}+\alpha_{4}})=\mathbb{C}_{-\omega_{4}+\alpha_{4}}\oplus \mathbb{C}_{-\omega_{4}}.$
\end{center}
Therefore we have
\begin{center} $H^0(s_{4}w_{2}s_{3},\alpha_{3})=\mathbb{C}_{-\omega_{4}+\alpha_{4}}\oplus \mathbb{C}_{-\omega_{4}}.$ 
\end{center}

Since $\langle -\omega_{4}+\alpha_{4}, \alpha_{3}\rangle =-1,$ by Lemma \ref{lemma1.3}(4) we have 
\begin{center}
$H^0(\tilde{L}_{\alpha_{3}}/\tilde{B}_{\alpha_{3}}, \mathbb{C}_{-\omega_{4}+\alpha_{4}})=0.$
\end{center}
Since $\langle -\omega_{4}, \alpha_{3}\rangle =0,$ by Lemma \ref{lemma1.3}(2) we have 
\begin{center}
$H^0(\tilde{L}_{\alpha_{3}}/\tilde{B}_{\alpha_{3}}, \mathbb{C}_{-\omega_{4}})=\mathbb{C}_{-\omega_{4}}.$
\end{center}

Thus from above disscussion we have
\begin{center}
$H^0(s_{3}s_{4}w_{2}s_{3},\alpha_{3})=\mathbb{C}_{-\omega_{4}}.$ 
\end{center}
Since $\alpha_{1},\alpha_{2}$ are othogonal to $\omega_{4},$ by Lemma \ref{lemma1.3}(2) we have 
\begin{center}
$H^0(w_{3}s_{3},\alpha_{3})=\mathbb{C}_{-\omega_{4}}.$
\end{center}

\end{proof}

\begin{cor}\label{rmk 4.3}

\item[(1)] $H^0(s_{4}w_{1}s_{3}, \alpha_{3})= \mathbb{C}_{-(\alpha_{2} + 2\alpha_{3} + \alpha_{4})}\oplus \mathbb{C}_{-(\alpha_{1} + \alpha_{2} + 2\alpha_{3} + \alpha_{4})}\oplus \mathbb{C}_{-(\alpha_{1} + 2\alpha_{2} +2 \alpha_{3} + \alpha_{4})}.$

\item[(2)] $H^0(s_{4}w_{2}s_{3}, \alpha_{3})=\mathbb{C}_{-\omega_{4}}\oplus \mathbb{C}_{-\omega_{4} + \alpha_{4}}.$

\item [(3)]$H^0(s_{4}w_{r}s_{3}, \alpha_{3})=0$ for $r= 3,4,5.$

\end{cor}	
\begin{proof}	
Proof of (1): we have 
\begin{center}
	$H^0(s_{1}s_{2}s_{3}, \alpha_{3})=\mathbb{C}h(\alpha_{3})\oplus \mathbb{C}_{-\alpha_{3}}\oplus \mathbb{C}_{-(\alpha_{2} + \alpha_{3})}\oplus \mathbb{C}_{-(\alpha_{1} + \alpha_{2} + \alpha_{3})}.$
\end{center}

Since $\langle -\alpha_{3}, \alpha_{4}\rangle =1,$ $\langle -(\alpha_{2}+\alpha_{3}), \alpha_{4}\rangle =1,$ and  $\langle -(\alpha_{1}+\alpha_{2}+\alpha_{3}), \alpha_{4}\rangle =1,$ by using SES and Lemma \ref{lemma1.3}(2) we have

\begin{center}
	$H^0(s_{4}s_{1}s_{2}s_{3},\alpha_{3})$=$\mathbb{C}h(\alpha_{3})\oplus \mathbb{C}_{-\alpha_{3}}\oplus \mathbb{C}_{-(\alpha_{3}+\alpha_{4})}\oplus \mathbb{C}_{-(\alpha_{2} + \alpha_{3})}\oplus\mathbb{C}_{-(\alpha_{2} + \alpha_{3}+\alpha_{4})}\oplus  \mathbb{C}_{-(\alpha_{1} + \alpha_{2} + \alpha_{3})}\oplus \mathbb{C}_{-(\alpha_{1} + \alpha_{2} + \alpha_{3}+\alpha_{4})}.$
\end{center}

Since $\mathbb{C}h(\alpha_{3})\oplus \mathbb{C}_{-\alpha_{3}}$
is two dimensional $\tilde{B}_{\alpha_{3}}$-module, by Lemma \ref{lemma 1.4}(1) we have 
\begin{center}
	$\mathbb{C}h(\alpha_{3})\oplus \mathbb{C}_{-\alpha_{3}}=V\otimes \mathbb{C}_{-\omega_{3}}$ 
\end{center} 
where $V$ is the standard two dimensional $\tilde{L}_{\alpha_{3}}$-module.

Thus by using Lemma \ref{lemma1.3}(4) we have 
\begin{center}
	$H^0(\tilde{L}_{\alpha_{3}}/\tilde{B}_{\alpha_{3}}, \mathbb{C}h(\alpha_{3})\oplus \mathbb{C}_{-\alpha_{3}})=0.$ 
\end{center} 
Since $\langle-(\alpha_{3} + \alpha_{4}), \alpha_{3} \rangle =-1,$ by Lemma \ref{lemma1.3}(4) we have  
\begin{center}
	$H^0(\tilde{L}_{\alpha_{3}}/\tilde{B}_{\alpha_{3}}, \mathbb{C}_{-(\alpha_{3}+\alpha_{4})})=0.$ 
\end{center}
Since $\langle-(\alpha_{2} + \alpha_{3}), \alpha_{3} \rangle =0,$ $\langle-(\alpha_{1}+\alpha_{2} + \alpha_{3}), \alpha_{3} \rangle =0,$ by Lemma \ref{lemma1.3}(2) we have  
\begin{center}
$H^0(\tilde{L}_{\alpha_{3}}/\tilde{B}_{\alpha_{3}}, \mathbb{C}_{-(\alpha_{2}+\alpha_{3})})=\mathbb{C}_{-(\alpha_{2}+\alpha_{3})}$ 
\end{center}	
and
\begin{center}
$H^0(\tilde{L}_{\alpha_{3}}/\tilde{B}_{\alpha_{3}}, \mathbb{C}_{-(\alpha_{1}+\alpha_{2}+\alpha_{3})})=\mathbb{C}_{-(\alpha_{1}+\alpha_{2}+\alpha_{3})}.$ 
\end{center}

Since $\langle-(\alpha_{2} + \alpha_{3}+\alpha_{4}), \alpha_{3} \rangle =1,$ $\langle-(\alpha_{1}+\alpha_{2} + \alpha_{3}+\alpha_{4}), \alpha_{3} \rangle =1,$ by Lemma \ref{lemma1.3}(2) we have  
\begin{center}
	$H^0(\tilde{L}_{\alpha_{3}}/\tilde{B}_{\alpha_{3}}, \mathbb{C}_{-(\alpha_{2}+\alpha_{3}+\alpha_{4})})=\mathbb{C}_{-(\alpha_{2}+\alpha_{3}+\alpha_{4})}\oplus \mathbb{C}_{-(\alpha_{2}+2\alpha_{3}+\alpha_{4})}$ 
\end{center} 
and 
\begin{center}
	$H^0(\tilde{L}_{\alpha_{3}}/\tilde{B}_{\alpha_{3}}, \mathbb{C}_{-(\alpha_{1}+\alpha_{2}+\alpha_{3}+\alpha_{4})})=\mathbb{C}_{-(\alpha_{1}+\alpha_{2}+\alpha_{3}+\alpha_{4})}\oplus \mathbb{C}_{-(\alpha_{1}+\alpha_{2}+2\alpha_{3}+\alpha_{4})}.$ 
\end{center}
Thus combining the above discussion we have 
\begin{center}
$H^0(s_{3}s_{4}s_{1}s_{2}s_{3},\alpha_{3})=\mathbb{C}_{-(\alpha_{2}+\alpha_{3})}\oplus \mathbb{C}_{-(\alpha_{1}+\alpha_{2}+\alpha_{3})}\oplus \mathbb{C}_{-(\alpha_{2}+\alpha_{3}+\alpha_{4})}\oplus \mathbb{C}_{-(\alpha_{2}+2\alpha_{3}+\alpha_{4})}\oplus \mathbb{C}_{-(\alpha_{1}+\alpha_{2}+\alpha_{3}+\alpha_{4})}\oplus \mathbb{C}_{-(\alpha_{1}+\alpha_{2}+2\alpha_{3}+\alpha_{4})}.$ \hspace{8.5cm}$(4.3.1)$
\end{center}

Since $\langle-(\alpha_{2}+\alpha_{3}),\alpha_{2} \rangle=-1,$ $\langle-(\alpha_{2}+\alpha_{3}+\alpha_{4}),\alpha_{2} \rangle=-1,$ by Lemma \ref{lemma1.3}(4) we have   
\begin{center}
$H^0(\tilde{L}_{\alpha_{2}}/\tilde{B}_{\alpha_{2}}, \mathbb{C}_{-(\alpha_{2}+\alpha_{3})})=0$
\end{center}
and
\begin{center}
	$H^0(\tilde{L}_{\alpha_{2}}/\tilde{B}_{\alpha_{2}}, \mathbb{C}_{-(\alpha_{2}+\alpha_{3}+\alpha_{4})})=0.$
\end{center}

Since $\langle-(\alpha_{1}+\alpha_{2}+\alpha_{3}),\alpha_{2} \rangle=0,$  $\langle-(\alpha_{2}+2\alpha_{3}+\alpha_{4}),\alpha_{2} \rangle=0,$ and   $\langle-(\alpha_{1}+\alpha_{2}+\alpha_{3}+\alpha_{4}),\alpha_{2} \rangle=0,$ by Lemma \ref{lemma1.3}(4) we have   
\begin{center}
$H^0(\tilde{L}_{\alpha_{2}}/\tilde{B}_{\alpha_{2}}, \mathbb{C}_{-(\alpha_{1}+\alpha_{2}+\alpha_{3})})=\mathbb{C}_{-(\alpha_{1}+\alpha_{2}+\alpha_{3})}$
\end{center}
\begin{center}
	$H^0(\tilde{L}_{\alpha_{2}}/\tilde{B}_{\alpha_{2}}, \mathbb{C}_{-(\alpha_{2}+2\alpha_{3}+\alpha_{4})})=\mathbb{C}_{-(\alpha_{2}+2\alpha_{3}+\alpha_{4})}$
\end{center}
and
\begin{center}
	$H^0(\tilde{L}_{\alpha_{2}}/\tilde{B}_{\alpha_{2}}, \mathbb{C}_{-(\alpha_{1}+\alpha_{2}+\alpha_{3}+\alpha_{4})})=\mathbb{C}_{-(\alpha_{1}+\alpha_{2}+\alpha_{3}+\alpha_{4})}.$
\end{center}

Since $\langle-(\alpha_{1}+\alpha_{2}+2\alpha_{3}+\alpha_{4}),\alpha_{2} \rangle=1,$ by Lemma \ref{lemma1.3}(2) we have   
\begin{center}
$H^0(\tilde{L}_{\alpha_{2}}/\tilde{B}_{\alpha_{2}}, \mathbb{C}_{-(\alpha_{1}+\alpha_{2}+2\alpha_{3}+\alpha_{4})})=\mathbb{C}_{-(\alpha_{1}+\alpha_{2}+2\alpha_{3}+\alpha_{4})}\oplus \mathbb{C}_{-(\alpha_{1}+2\alpha_{2}+2\alpha_{3}+\alpha_{4})}.$
\end{center}

Thus combining the above discussion we have 
\begin{center}
	$H^0(s_{2}s_{3}s_{4}s_{1}s_{2}s_{3},\alpha_{3})$=$ \mathbb{C}_{-(\alpha_{1}+\alpha_{2}+\alpha_{3})}\oplus \mathbb{C}_{-(\alpha_{2}+2\alpha_{3}+\alpha_{4})}\oplus \mathbb{C}_{-(\alpha_{1}+\alpha_{2}+\alpha_{3}+\alpha_{4})}\oplus \mathbb{C}_{-(\alpha_{1}+\alpha_{2}+2\alpha_{3}+\alpha_{4})}\oplus \mathbb{C}_{-(\alpha_{1}+2\alpha_{2}+2\alpha_{3}+\alpha_{4})}.$
\end{center}
Since $\langle-(\alpha_{1}+\alpha_{2}+\alpha_{3}),\alpha_{1} \rangle=-1$ and  $\langle-(\alpha_{1}+\alpha_{2}+\alpha_{3}+\alpha_{4}),\alpha_{1} \rangle=-1,$ by Lemma \ref{lemma1.3}(4) we have   
\begin{center}
$H^0(\tilde{L}_{\alpha_{1}}/\tilde{B}_{\alpha_{1}}, \mathbb{C}_{-(\alpha_{1}+\alpha_{2}+\alpha_{3})})=0$
\end{center}
and
\begin{center}
$H^0(\tilde{L}_{\alpha_{1}}/\tilde{B}_{\alpha_{1}}, \mathbb{C}_{-(\alpha_{1}+\alpha_{2}+\alpha_{3}+\alpha_{4})})=0.$
\end{center}
Since $\mathbb{C}_{-(\alpha_{1}+\alpha_{2}+2\alpha_{3}+\alpha_{4})}\oplus \mathbb{C}_{-(\alpha_{2}+2\alpha_{3}+\alpha_{4})}$ is the standard two dimensional irreducible $\tilde{L}_{\alpha_{1}}$-module, by using Lemma \ref{lemma1.3}(2) we have 
\begin{center}  
$H^0(\tilde{L}_{\alpha_{1}}/\tilde{B}_{\alpha_{1}}, \mathbb{C}_{-(\alpha_{1}+\alpha_{2}+2\alpha_{3}+\alpha_{4})}\oplus \mathbb{C}_{-(\alpha_{2}+2\alpha_{3}+\alpha_{4})} )=\mathbb{C}_{-(\alpha_{1}+\alpha_{2}+2\alpha_{3}+\alpha_{4})}\oplus \mathbb{C}_{-(\alpha_{2}+2\alpha_{3}+\alpha_{4})}.$
\end{center}
Since $\langle-(\alpha_{1}+2\alpha_{2}+2\alpha_{3}+\alpha_{4}),\alpha_{1} \rangle=0,$ by Lemma \ref{lemma1.3}(2) we have 
\begin{center}
	$H^0(\tilde{L}_{\alpha_{1}}/\tilde{B}_{\alpha_{1}}, \mathbb{C}_{-(\alpha_{1}+2\alpha_{2}+2\alpha_{3}+\alpha_{4})})=\mathbb{C}_{-(\alpha_{1}+2\alpha_{2}+2\alpha_{3}+\alpha_{4})}.$
\end{center}
Thus combining the above discussion we have 
\begin{center}
	$H^0(w_{1}s_{3},\alpha_{3})$=$\mathbb{C}_{-(\alpha_{2}+2\alpha_{3}+\alpha_{4})}\oplus \mathbb{C}_{-(\alpha_{1}+\alpha_{2}+2\alpha_{3}+\alpha_{4})}\oplus \mathbb{C}_{-(\alpha_{1}+2\alpha_{2}+2\alpha_{3}+\alpha_{4})}.$
\end{center}

Since $\langle-(\alpha_{2}+2\alpha_{3}+\alpha_{4}),\alpha_{4} \rangle=0,$ $\langle-(\alpha_{1}+\alpha_{2}+2\alpha_{3}+\alpha_{4}),\alpha_{4} \rangle=0,$ and $\langle-(\alpha_{1}+2\alpha_{2}+2\alpha_{3}+\alpha_{4}),\alpha_{4} \rangle=0,$ by Lemma \ref{lemma1.3}(2) we have

\begin{center}
		$H^0(\tilde{L}_{\alpha_{4}}/\tilde{B}_{\alpha_{4}}, \mathbb{C}_{-(\alpha_{2}+2\alpha_{3}+\alpha_{4})})=\mathbb{C}_{-(\alpha_{2}+2\alpha_{3}+\alpha_{4})}$
\end{center}
\begin{center}
$H^0(\tilde{L}_{\alpha_{4}}/\tilde{B}_{\alpha_{4}}, \mathbb{C}_{-(\alpha_{1}+\alpha_{2}+2\alpha_{3}+\alpha_{4})})=\mathbb{C}_{-(\alpha_{1}+\alpha_{2}+2\alpha_{3}+\alpha_{4})}$
\end{center}
and
\begin{center}
$H^0(\tilde{L}_{\alpha_{4}}/\tilde{B}_{\alpha_{4}}, \mathbb{C}_{-(\alpha_{1}+2\alpha_{2}+2\alpha_{3}+\alpha_{4})})=\mathbb{C}_{-(\alpha_{1}+2\alpha_{2}+2\alpha_{3}+\alpha_{4})}.$
\end{center}
Therefore we have 
\begin{center}
$H^0(s_{4}w_{1}s_{3},\alpha_{3})$=$\mathbb{C}_{-(\alpha_{2}+2\alpha_{3}+\alpha_{4})}\oplus \mathbb{C}_{-(\alpha_{1}+\alpha_{2}+2\alpha_{3}+\alpha_{4})}\oplus \mathbb{C}_{-(\alpha_{1}+2\alpha_{2}+2\alpha_{3}+\alpha_{4})}.$
\end{center}

Proof of (2): By Lemma \ref{lem 4.2}(1) we have 
\begin{center}
	$H^0(w_{2}s_{3}, \alpha_{3})=\mathbb{C}_{-\omega_{4}+\alpha_{4}}.$
\end{center}
Since $\langle -\omega_{4}+\alpha_{4}, \alpha_{4}\rangle =1,$ by using SES and Lemma \ref{lemma1.3}(2) we have
\begin{center}
	$H^0(s_{4}w_{2}s_{3}, \alpha_{3})=\mathbb{C}_{-\omega_{4}+\alpha_{4}}\oplus \mathbb{C}_{-\omega_{4}}.$
\end{center} 

Proof of (3): By the Lemma \ref{lem 4.2}(3) we have 

\begin{center}
	$H^0(w_{3}s_{3},\alpha_{3})=\mathbb{C}_{-\omega_{4}}.$
\end{center}
\end{proof}
Since $\langle -\omega_{4}, \alpha_{4}\rangle =-1,$ by Lemma \ref{lemma1.3}(4) we have  
\begin{center}
	$H^0(s_{4}w_{3}s_{3},\alpha_{3})=0.$
\end{center}
By using SES repeatedly we have 
\begin{center}
	$H^0(s_{4}w_{r}s_{3},\alpha_{3})=0$ for $r=4,5.$
\end{center}

\begin{cor}\label{rem 4.4}
\item[(1)]	
$H^0(s_{3}s_{4}s_{1}s_{2}s_{3},\alpha_{3})=\mathbb{C}_{-(\alpha_{2}+\alpha_{3})}\oplus\mathbb{C}_{-(\alpha_{1}+\alpha_{2}+\alpha_{3})}\oplus\mathbb{C}_{-(\alpha_{2}+\alpha_{3}+\alpha_{4})}\oplus \mathbb{C}_{-(\alpha_{2}+2\alpha_{3}+\alpha_{4})}\oplus \mathbb{C}_{-(\alpha_{1}+\alpha_{2}+\alpha_{3}+\alpha_{4})}\oplus\mathbb{C}_{-(\alpha_{1}+\alpha_{2}+2\alpha_{3}+\alpha_{4})}.$

\item[(2)] $H^0(s_{3}s_{4}w_{1}s_{3}, \alpha_{3})=\mathbb{C}_{-(\alpha_{1} + 2\alpha_{2} + 2\alpha_{3} + \alpha_{4})}\oplus \mathbb{C}_{-\omega_{4} + \alpha_{4}}.$
	
\item[(3)] $H^0(s_{3}s_{4}w_{2}s_{3}, \alpha_{3})=\mathbb{C}_{-\omega_{4}}.$	
\end{cor}
\begin{proof}
Proof of (1): Proof follows from $(4.3.1).$

Proof of (2): Proof follows from the Corollary \ref{rmk 4.3}(1).

Proof of (3): Proof follows from the Corollary \ref{rmk 4.3}(2).
\end{proof}

\begin{cor}\label{rem 4.5}
	\item [(1)]

	$H^0(s_{2}s_{3}s_{4}s_{1}s_{2}s_{3},\alpha_{3})= \mathbb{C}_{-(\alpha_{1}+\alpha_{2}+\alpha_{3})}\oplus \mathbb{C}_{-(\alpha_{2}+2\alpha_{3}+\alpha_{4})}\oplus \mathbb{C}_{-(\alpha_{1}+\alpha_{2}+\alpha_{3}+\alpha_{4})}\oplus \mathbb{C}_{-(\alpha_{1}+\alpha_{2}+2\alpha_{3}+\alpha_{4})}\oplus \mathbb{C}_{-(\alpha_{1}+2\alpha_{2}+2\alpha_{3}+\alpha_{4})}.$
	
	\item[(2)] $H^0(s_{2}s_{3}s_{4}w_{1}s_{3}, \alpha_{3})=\mathbb{C}_{-\omega_{4} +\alpha_{4}}.$
	
	\item[(3)] $H^1(s_{2}s_{3}s_{4}w_{2}s_{3}, \alpha_{3})=\mathbb{C}_{-\omega_{4}}.$
		
\end{cor}
\begin{proof}
	Proof of (1): Proof follows from Corollary \ref{rem 4.4}(1).

	Proof of (2): Proof follows from Corollary \ref{rem 4.4}(2).

	Proof of (3): Proof follows from Corollary \ref{rem 4.4}(3).
	
\end{proof}	

\begin{cor}\label{rem 4.6}
\item[(1)]	
	$H^0(s_{4}s_{3}s_{4}s_{1}s_{2}s_{3},\alpha_{3})=\mathbb{C}_{-(\alpha_{2}+\alpha_{3})}\oplus\mathbb{C}_{-(\alpha_{1}+\alpha_{2}+\alpha_{3})}\oplus\mathbb{C}_{-(\alpha_{2}+\alpha_{3}+\alpha_{4})}\oplus \mathbb{C}_{-(\alpha_{2}+2\alpha_{3}+\alpha_{4})}\oplus \mathbb{C}_{-(\alpha_{1}+\alpha_{2}+\alpha_{3}+\alpha_{4})}\oplus\mathbb{C}_{-(\alpha_{1}+\alpha_{2}+2\alpha_{3}+\alpha_{4})}.$

\item[(2)] $H^0(s_{4}s_{3}s_{4}w_{1}s_{3}, \alpha_{3})=\mathbb{C}_{-(\alpha_{1} + 2\alpha_{2} +2\alpha_{3} +\alpha_{4})}\oplus \mathbb{C}_{-\omega_{4} + \alpha_{4}} \oplus \mathbb{C}_{-\omega_{4}}.$
	
	`
	
\end{cor}
\begin{proof}

Proof of (1): By the Corollary \ref{rem 4.4}(1) we have

\begin{center}
		
	$H^0(s_{3}s_{4}s_{1}s_{2}s_{3},\alpha_{3})=\mathbb{C}_{-(\alpha_{2}+\alpha_{3})}\oplus\mathbb{C}_{-(\alpha_{1}+\alpha_{2}+\alpha_{3})}\oplus\mathbb{C}_{-(\alpha_{2}+\alpha_{3}+\alpha_{4})}\oplus \mathbb{C}_{-(\alpha_{2}+2\alpha_{3}+\alpha_{4})}\oplus \mathbb{C}_{-(\alpha_{1}+\alpha_{2}+\alpha_{3}+\alpha_{4})}\oplus\mathbb{C}_{-(\alpha_{1}+\alpha_{2}+2\alpha_{3}+\alpha_{4})}.$
	
\end{center}
Since $\mathbb{C}_{-(\alpha_{2}+\alpha_{3})}\oplus\mathbb{C}_{-(\alpha_{2}+\alpha_{3}+\alpha_{4})}$ is the standard two dimensional irreducible $\tilde{L}_{\alpha_{4}}$-module, by Lemma \ref{lemma1.3}(2) we have 
\begin{center}
$H^0(\tilde{L}_{\alpha_{4}}/\tilde{B}_{\alpha_{4}},\mathbb{C}_{-(\alpha_{2}+\alpha_{3})}\oplus\mathbb{C}_{-(\alpha_{2}+\alpha_{3}+\alpha_{4})})=\mathbb{C}_{-(\alpha_{2}+\alpha_{3})}\oplus\mathbb{C}_{-(\alpha_{2}+\alpha_{3}+\alpha_{4})}.$
\end{center}
Also, since $\mathbb{C}_{-(\alpha_{1}+\alpha_{2}+\alpha_{3})}\oplus\mathbb{C}_{-(\alpha_{1}+\alpha_{2}+\alpha_{3}+\alpha_{4})}$ is the standard two dimensional irreducible $\tilde{L}_{\alpha_{4}}$-module, by Lemma \ref{lemma1.3}(2) we have 
\begin{center}
	$H^0(\tilde{L}_{\alpha_{4}}/\tilde{B}_{\alpha_{4}},\mathbb{C}_{-(\alpha_{1}+\alpha_{2}+\alpha_{3})}\oplus\mathbb{C}_{-(\alpha_{1}+\alpha_{2}+\alpha_{3}+\alpha_{4})})=\mathbb{C}_{-(\alpha_{1}+\alpha_{2}+\alpha_{3})}\oplus\mathbb{C}_{-(\alpha_{1}+\alpha_{2}+\alpha_{3}+\alpha_{4})}.$
\end{center}
Since $\langle -(\alpha_{2}+2\alpha_{3}+\alpha_{4}),\alpha_{4} \rangle=0$ and $\langle -(\alpha_{1}+\alpha_{2}+2\alpha_{3}+\alpha_{4}),\alpha_{4} \rangle=0,$ by Lemma \ref{lemma1.3}(2) we have 

\begin{center}
	$H^0(\tilde{L}_{\alpha_{4}}/\tilde{B}_{\alpha_{4}},\mathbb{C}_{-(\alpha_{2}+2\alpha_{3}+\alpha_{4})})=\mathbb{C}_{-(\alpha_{2}+2\alpha_{3}+\alpha_{4})}$
\end{center}
and
\begin{center}
	$H^0(\tilde{L}_{\alpha_{4}}/\tilde{B}_{\alpha_{4}},\mathbb{C}_{-(\alpha_{1}+\alpha_{2}+2\alpha_{3}+\alpha_{4})})=\mathbb{C}_{-(\alpha_{1}+\alpha_{2}+2\alpha_{3}+\alpha_{4})}.$
\end{center}
Thus combining the above discussion we have 
\begin{center}
	$H^0(s_{4}s_{3}s_{4}s_{1}s_{2}s_{3},\alpha_{3})=\mathbb{C}_{-(\alpha_{2}+\alpha_{3})}\oplus\mathbb{C}_{-(\alpha_{1}+\alpha_{2}+\alpha_{3})}\oplus\mathbb{C}_{-(\alpha_{2}+\alpha_{3}+\alpha_{4})}\oplus \mathbb{C}_{-(\alpha_{2}+2\alpha_{3}+\alpha_{4})}\oplus \mathbb{C}_{-(\alpha_{1}+\alpha_{2}+\alpha_{3}+\alpha_{4})}\oplus\mathbb{C}_{-(\alpha_{1}+\alpha_{2}+2\alpha_{3}+\alpha_{4})}.$

\end{center}
Proof of (2): 	By Corollary \ref{rem 4.4}(2) we have
\begin{center}
 $H^0(s_{3}s_{4}w_{1}s_{3}, \alpha_{3})=\mathbb{C}_{-(\alpha_{1} + 2\alpha_{2} + 2\alpha_{3} + \alpha_{4})}\oplus \mathbb{C}_{-(\alpha_{1} + 2\alpha_{2} + 3 \alpha_{3} + \alpha_{4})}.$
\end{center}
Since $\langle -(\alpha_{1}+2\alpha_{2}+2\alpha_{3}+\alpha_{4}),\alpha_{4}\rangle=0,$ by Lemma \ref{lemma1.3}(2) we have 
\begin{center}
$H^0(\tilde{L}_{\alpha_{4}}/\tilde{B}_{\alpha_{4}},\mathbb{C}_{-(\alpha_{1}+2\alpha_{2}+2\alpha_{3}+\alpha_{4})})=\mathbb{C}_{-(\alpha_{1}+2\alpha_{2}+2\alpha_{3}+\alpha_{4})}.$
\end{center}
Further, since  
 $\langle -(\alpha_{1}+2\alpha_{2}+3\alpha_{3}+\alpha_{4}),\alpha_{4}\rangle=1,$ by Lemma \ref{lemma1.3}(2) we have 
 \begin{center}
 	$H^0(\tilde{L}_{\alpha_{4}}/\tilde{B}_{\alpha_{4}},\mathbb{C}_{-(\alpha_{1}+2\alpha_{2}+3\alpha_{3}+\alpha_{4})})=\mathbb{C}_{-(\alpha_{1}+2\alpha_{2}+3\alpha_{3}+\alpha_{4})}\oplus \mathbb{C}_{-(\alpha_{1}+2\alpha_{2}+3\alpha_{3}+2\alpha_{4})}.$
 \end{center}
Thus combining the above discussion we have 

\begin{center}
$H^0(s_{4}s_{3}s_{4}w_{1}s_{3}, \alpha_{3})=\mathbb{C}_{-(\alpha_{1} + 2\alpha_{2} +2\alpha_{3} +\alpha_{4})}\oplus \mathbb{C}_{-\omega_{4} + \alpha_{4}} \oplus \mathbb{C}_{-\omega_{4}},$
\end{center}
since $\omega_{4}=\alpha_{1}+2\alpha_{2}+3\alpha_{3}+2\alpha_{4}.$
\end{proof}	

\begin{cor}\label{rem 4.7}

\item[(1)]$H^0(s_{4}s_{2}s_{3}s_{4}s_{1}s_{2}s_{3},\alpha_{3})$=$ \mathbb{C}_{-(\alpha_{1}+\alpha_{2}+\alpha_{3})}\oplus \mathbb{C}_{-(\alpha_{2}+2\alpha_{3}+\alpha_{4})}\oplus \mathbb{C}_{-(\alpha_{1}+\alpha_{2}+\alpha_{3}+\alpha_{4})}\oplus \mathbb{C}_{-(\alpha_{1}+\alpha_{2}+2\alpha_{3}+\alpha_{4})}\oplus \mathbb{C}_{-(\alpha_{1}+2\alpha_{2}+2\alpha_{3}+\alpha_{4})}.$

\item[(2)] $H^0(s_{4}s_{2}s_{3}s_{4}w_{1}s_{3}, \alpha_{3})=\mathbb{C}_{-\omega_{4} + \alpha_{4}} \oplus \mathbb{C}_{-\omega_{4}}.$	
\end{cor}
\begin{proof}
Proof of (1): By Corollary \ref{rem 4.5}(1) we have
\begin{center}
$H^0(s_{2}s_{3}s_{4}s_{1}s_{2}s_{3},\alpha_{3})$=$ \mathbb{C}_{-(\alpha_{1}+\alpha_{2}+\alpha_{3})}\oplus \mathbb{C}_{-(\alpha_{2}+2\alpha_{3}+\alpha_{4})}\oplus \mathbb{C}_{-(\alpha_{1}+\alpha_{2}+\alpha_{3}+\alpha_{4})}\oplus \mathbb{C}_{-(\alpha_{1}+\alpha_{2}+2\alpha_{3}+\alpha_{4})}\oplus \mathbb{C}_{-(\alpha_{1}+2\alpha_{2}+2\alpha_{3}+\alpha_{4})}.$
\end{center}
Since $\mathbb{C}_{-(\alpha_{1}+\alpha_{2}+\alpha_{3})}\oplus \mathbb{C}_{-(\alpha_{1}+\alpha_{2}+\alpha_{3}+\alpha_{4})}$ is the standard two dimenional irreducible $\tilde{L}_{\alpha_{4}}$-module, by Lemma \ref{lemma1.3}(2) we have 
\begin{center}
$H^0(\tilde{L}_{\alpha_{4}}/\tilde{B}_{\alpha_{4}}, \mathbb{C}_{-(\alpha_{1}+\alpha_{2}+\alpha_{3})}\oplus \mathbb{C}_{-(\alpha_{1}+\alpha_{2}+\alpha_{3}+\alpha_{4})} )=\mathbb{C}_{-(\alpha_{1}+\alpha_{2}+\alpha_{3})}\oplus \mathbb{C}_{-(\alpha_{1}+\alpha_{2}+\alpha_{3}+\alpha_{4})}.$
\end{center}

Moreover, Since $\langle-(\alpha_{2}+2\alpha_{3}+\alpha_{4}),\alpha_{4} \rangle=0,$ $\langle-(\alpha_{1}+\alpha_{2}+2\alpha_{3}+\alpha_{4}),\alpha_{4} \rangle=0$ and $\langle-(\alpha_{1}+2\alpha_{2}+2\alpha_{3}+\alpha_{4}),\alpha_{4} \rangle=0,$ by Lemma \ref{lemma1.3}(2) we have
\begin{center}
$H^0(\tilde{L}_{\alpha_{4}}/\tilde{B}_{\alpha_{4}}, \mathbb{C}_{-(\alpha_{2}+2\alpha_{3}+\alpha_{4})} )= \mathbb{C}_{-(\alpha_{2}+2\alpha_{3}+\alpha_{4})}$
\end{center}
\begin{center}
$H^0(\tilde{L}_{\alpha_{4}}/\tilde{B}_{\alpha_{4}}, \mathbb{C}_{-(\alpha_{1}+\alpha_{2}+2\alpha_{3}+\alpha_{4})} )= \mathbb{C}_{-(\alpha_{1}+\alpha_{2}+2\alpha_{3}+\alpha_{4})}$
\end{center}
and
\begin{center}
$H^0(\tilde{L}_{\alpha_{4}}/\tilde{B}_{\alpha_{4}}, \mathbb{C}_{-(\alpha_{1}+2\alpha_{2}+2\alpha_{3}+\alpha_{4})} )= \mathbb{C}_{-(\alpha_{1}+2\alpha_{2}+2\alpha_{3}+\alpha_{4})}.$
\end{center}
Thus combining the above discussion we have 
\begin{center}
$H^0(s_{4}s_{2}s_{3}s_{4}s_{1}s_{2}s_{3},\alpha_{3})$=$ \mathbb{C}_{-(\alpha_{1}+\alpha_{2}+\alpha_{3})}\oplus \mathbb{C}_{-(\alpha_{2}+2\alpha_{3}+\alpha_{4})}\oplus \mathbb{C}_{-(\alpha_{1}+\alpha_{2}+\alpha_{3}+\alpha_{4})}\oplus \mathbb{C}_{-(\alpha_{1}+\alpha_{2}+2\alpha_{3}+\alpha_{4})}\oplus \mathbb{C}_{-(\alpha_{1}+2\alpha_{2}+2\alpha_{3}+\alpha_{4})}.$
\end{center}

Proof of (2): By Corollary \ref{rem 4.5}(2) we have
\begin{center}
 $H^0(s_{2}s_{3}s_{4}w_{1}s_{3}, \alpha_{3})=\mathbb{C}_{-\omega_{4} +\alpha_{4}}.$
\end{center}
Since $\langle -\omega_{4}+\alpha_{4}, \alpha_{4}\rangle =1,$ by Lemma \ref{lemma1.3}(2) we have

\begin{center}
$H^0(\tilde{L}_{\alpha_{4}}/\tilde{B}_{\alpha_{4}}, \mathbb{C}_{-\omega_{4}+\alpha_{4}})= \mathbb{C}_{-\omega_{4}+\alpha_{4}}\oplus \mathbb{C}_{-\omega_{4}}.$
\end{center}

Thus we have 
\begin{center}
$H^0(s_{4}s_{2}s_{3}s_{4}w_{1}s_{3}, \alpha_{3})=\mathbb{C}_{-\omega_{4} + \alpha_{4}} \oplus \mathbb{C}_{-\omega_{4}}.$	
\end{center}
\end{proof}
\begin{cor}\label{rem 4.8}

\item[(1)] 	$H^0(s_{3}s_{4}s_{2}s_{3}s_{4}s_{1}s_{2}s_{3},\alpha_{3})$=$ \mathbb{C}_{-(\alpha_{1}+\alpha_{2}+\alpha_{3})}\oplus \mathbb{C}_{-(\alpha_{1}+\alpha_{2}+\alpha_{3}+\alpha_{4})}\oplus \mathbb{C}_{-(\alpha_{1}+\alpha_{2}+2\alpha_{3}+\alpha_{4})}\oplus \mathbb{C}_{-(\alpha_{1}+2\alpha_{2}+2\alpha_{3}+\alpha_{4})}\oplus \mathbb{C}_{-\omega_{4} + \alpha_{4}}.$
	
\item[(2)] $H^0(s_{3}s_{4}s_{2}s_{3}s_{4}w_{1}s_{3}, \alpha_{3})=\mathbb{C}_{-\omega_{4}}.$	
\end{cor}
\begin{proof}
	Proof of (1): By Corollary \ref{rem 4.7}(1) we have
\begin{center}
$H^0(s_{4}s_{2}s_{3}s_{4}s_{1}s_{2}s_{3},\alpha_{3})$=$ \mathbb{C}_{-(\alpha_{1}+\alpha_{2}+\alpha_{3})}\oplus \mathbb{C}_{-(\alpha_{2}+2\alpha_{3}+\alpha_{4})}\oplus \mathbb{C}_{-(\alpha_{1}+\alpha_{2}+\alpha_{3}+\alpha_{4})}\oplus \mathbb{C}_{-(\alpha_{1}+\alpha_{2}+2\alpha_{3}+\alpha_{4})}\oplus \mathbb{C}_{-(\alpha_{1}+2\alpha_{2}+2\alpha_{3}+\alpha_{4})}.$
\end{center}
Since $\langle-(\alpha_{2}+2\alpha_{3}+\alpha_{4}),\alpha_{3}\rangle =-1,$ by Lemma \ref{lemma1.3}(4) we have 
\begin{center}
	$H^0(\tilde{L}_{\alpha_{3}}/\tilde{B}_{\alpha_{3}},\mathbb{C}_{-(\alpha_{2}+2\alpha_{3}+\alpha_{4})})=0.$
\end{center}

Since
$\langle-(\alpha_{1}+\alpha_{2}+\alpha_{3}),\alpha_{3}\rangle =0,$ by Lemma \ref{lemma1.3}(2) we have 
\begin{center}
	$H^0(\tilde{L}_{\alpha_{3}}/\tilde{B}_{\alpha_{3}},\mathbb{C}_{-(\alpha_{1}+\alpha_{2}+\alpha_{3})})=\mathbb{C}_{-(\alpha_{1}+\alpha_{2}+\alpha_{3})}.$
\end{center} 
Since $\mathbb{C}_{-(\alpha_{1}+\alpha_{2}+\alpha_{3}+\alpha_{4})}\oplus \mathbb{C}_{-(\alpha_{1}+\alpha_{2}+2\alpha_{3}+\alpha_{4})}$ is the standard two dimensional irreducible $\tilde{L}_{\alpha_{3}}$-module, by Lemma \ref{lemma1.3}(2) we have 
\begin{center}
$H^0(\tilde{L}_{\alpha_{3}}/\tilde{B}_{\alpha_{3}},\mathbb{C}_{-(\alpha_{1}+\alpha_{2}+\alpha_{3}+\alpha_{4})}\oplus \mathbb{C}_{-(\alpha_{1}+\alpha_{2}+2\alpha_{3}+\alpha_{4})})=\mathbb{C}_{-(\alpha_{1}+\alpha_{2}+\alpha_{3}+\alpha_{4})}\oplus \mathbb{C}_{-(\alpha_{1}+\alpha_{2}+2\alpha_{3}+\alpha_{4})}.$	
\end{center} 
Since $\langle-(\alpha_{1}+2\alpha_{2}+2\alpha_{3}+\alpha_{4}),\alpha_{3}\rangle=1,$ by Lemma \ref{lemma1.3}(2) we have

\begin{center}
	$H^0(\tilde{L}_{\alpha_{3}}/\tilde{B}_{\alpha_{3}},\mathbb{C}_{-(\alpha_{1}+2\alpha_{2}+2\alpha_{3}+\alpha_{4})})=\mathbb{C}_{-(\alpha_{1}+2\alpha_{2}+2\alpha_{3}+\alpha_{4})}\oplus \mathbb{C}_{-(\alpha_{1}+2\alpha_{2}+3\alpha_{3}+\alpha_{4})}.$	
\end{center} 
Thus combining the aove discussion we have 
\begin{center}
	$H^0(s_{4}s_{2}s_{3}s_{4}s_{1}s_{2}s_{3},\alpha_{3})$=$ \mathbb{C}_{-(\alpha_{1}+\alpha_{2}+\alpha_{3})}\oplus \mathbb{C}_{-(\alpha_{1}+\alpha_{2}+\alpha_{3}+\alpha_{4})}\oplus \mathbb{C}_{-(\alpha_{1}+\alpha_{2}+2\alpha_{3}+\alpha_{4})}\oplus \mathbb{C}_{-(\alpha_{1}+2\alpha_{2}+2\alpha_{3}+\alpha_{4})}\oplus \mathbb{C}_{-(\alpha_{1}+2\alpha_{2}+3\alpha_{3}+\alpha_{4})}.$
\end{center}

Proof of (2): By Corollary \ref{rem 4.7}(2) we have 
\begin{center}
	 $H^0(s_{4}s_{2}s_{3}s_{4}w_{1}s_{3}, \alpha_{3})=\mathbb{C}_{-\omega_{4} + \alpha_{4}} \oplus \mathbb{C}_{-\omega_{4}}.$	
\end{center}
Since $\langle -\omega_{4}+\alpha_{4}, \alpha_{3}\rangle =-1,$ by Lemma \ref{lemma1.3}(4) we have 
\begin{center}
	$H^0(\tilde{L}_{\alpha_{3}}/\tilde{B}_{\alpha_{3}},\mathbb{C}_{-\omega_{4} + \alpha_{4}})=0.$
\end{center}
Further, since $\langle -\omega_{4}, \alpha_{3}\rangle =0,$ by Lemma \ref{lemma1.3}(2) we have 
\begin{center}
	$H^0(\tilde{L}_{\alpha_{3}}/\tilde{B}_{\alpha_{3}},\mathbb{C}_{-\omega_{4}})=\mathbb{C}_{-\omega_{4}}.$
\end{center}
Thus from the above discussion we have 
\begin{center}
 $H^0(s_{3}s_{4}s_{2}s_{3}s_{4}w_{1}s_{3}, \alpha_{3})=\mathbb{C}_{-\omega_{4}}.$
\end{center}
\end{proof}	
\begin{cor}\label{rem 4.9}

\item[(1)] 
$H^0(s_{4}s_{3}s_{4}s_{2}s_{3},\alpha_{3})= \mathbb{C}_{-(\alpha_{2}+\alpha_{3})}\oplus \mathbb{C}_{-(\alpha_{2}+\alpha_{3}+\alpha_{4})}\oplus\mathbb{C}_{-(\alpha_{2}+2\alpha_{3}+\alpha_{4})}.$
\item[(2)] $H^0(s_{4}s_{3}s_{4}s_{2}s_{3}s_{4}s_{1}s_{2}s_{3},\alpha_{3})$=$ \mathbb{C}_{-(\alpha_{1}+\alpha_{2}+\alpha_{3})}\oplus \mathbb{C}_{-(\alpha_{1}+\alpha_{2}+\alpha_{3}+\alpha_{4})}\oplus \mathbb{C}_{-(\alpha_{1}+\alpha_{2}+2\alpha_{3}+\alpha_{4})}\oplus \mathbb{C}_{-(\alpha_{1}+2\alpha_{2}+2\alpha_{3}+\alpha_{4})}\oplus \mathbb{C}_{-\omega_{4}+\alpha_{4}}\oplus \mathbb{C}_{-\omega_{4}}.$
\end{cor}	
\begin{proof}
Proof of (1):It is easy to see that 
\begin{center}
	$H^0(s_{3},\alpha_{3})=\mathbb{C}_{-\alpha_{3}}\oplus \mathbb{C}h(\alpha_{3})\oplus \mathbb{C}_{\alpha_{3}}.$
\end{center}
Since $\langle -\alpha_{3}, \alpha_{2}\rangle =1,$ by using Lemma \ref{lemma1.3}(2), Lemma \ref{lemma1.3}(4) we have 
\begin{center}
	$H^0(s_{2}s_{3},\alpha_{3})= \mathbb{C}h(\alpha_{3})\oplus \mathbb{C}_{-\alpha_{3}}\oplus \mathbb{C}_{-(\alpha_{2}+\alpha_{3})}.$
\end{center}
Since $\langle -\alpha_{3}, \alpha_{4}\rangle =1,$ by using Lemma \ref{lemma1.3}(2) we have 
\begin{center}
	$H^0(s_{4}s_{2}s_{3},\alpha_{3})= \mathbb{C}h(\alpha_{3})\oplus \mathbb{C}_{-\alpha_{3}}\oplus \mathbb{C}_{-(\alpha_{3}+\alpha_{4})}\oplus \mathbb{C}_{-(\alpha_{2}+\alpha_{3})}\oplus \mathbb{C}_{-(\alpha_{2}+\alpha_{3}+\alpha_{4})}.$
\end{center}
Since $\mathbb{C}h(\alpha_{3})\oplus \mathbb{C}_{-\alpha_{3}}$ is the two dimensionl indecomposable $\tilde{B}_{\alpha_{3}}$-module, by Lemma \ref{lemma 1.4}(1) we have 
\begin{center}
	$\mathbb{C}h(\alpha_{3})\oplus \mathbb{C}_{-\alpha_{3}}=V\otimes \mathbb{C}_{-\omega_{3}}$
\end{center}
where $V$ is the standard two dimensional irreducible $\tilde{L}_{\alpha_{3}}$- module.
Thus by Lemma \ref{lemma1.3}(4) we have 
\begin{center}
	$H^0(\tilde{L}_{\alpha_{3}}/\tilde{B}_{3}, \mathbb{C}h(\alpha_{3})\oplus \mathbb{C}_{-\alpha_{3}})=0.$
\end{center} 
Also, since $\langle -(\alpha_{3}+\alpha_{4}),\alpha_{3}\rangle =-1,$ by Lemma \ref{lemma1.3}(4) we have 
\begin{center}
	$H^0(\tilde{L}_{\alpha_{3}}/\tilde{B}_{3}, \mathbb{C}_{-(\alpha_{3}+\alpha_{4})})=0.$
\end{center}  
Since $\langle -(\alpha_{2}+\alpha_{3}),\alpha_{3}\rangle =0,$ by Lemma \ref{lemma1.3}(2) we have 
\begin{center}
	$H^0(\tilde{L}_{\alpha_{3}}/\tilde{B}_{3}, \mathbb{C}_{-(\alpha_{2}+\alpha_{3})})=\mathbb{C}_{-(\alpha_{2}+\alpha_{3})}.$
\end{center} 
Since $\langle -(\alpha_{2}+\alpha_{3}+\alpha_{4}),\alpha_{3}\rangle =1,$ by Lemma \ref{lemma1.3}(2) we have 
\begin{center}
	$H^0(\tilde{L}_{\alpha_{3}}/\tilde{B}_{3}, \mathbb{C}_{-(\alpha_{2}+\alpha_{3}+\alpha_{4})})=\mathbb{C}_{-(\alpha_{2}+\alpha_{3}+\alpha_{4})}\oplus \mathbb{C}_{-(\alpha_{2}+2\alpha_{3}+\alpha_{4})}.$
\end{center} 
Thus combining the above discussion we have 
\begin{center}
	$H^0(s_{3}s_{4}s_{2}s_{3},\alpha_{3})= \mathbb{C}_{-(\alpha_{2}+\alpha_{3})}\oplus \mathbb{C}_{-(\alpha_{2}+\alpha_{3}+\alpha_{4})}\oplus\mathbb{C}_{-(\alpha_{2}+2\alpha_{3}+\alpha_{4})}.$
\end{center}
Since $\mathbb{C}_{-(\alpha_{2}+\alpha_{3})}\oplus \mathbb{C}_{-(\alpha_{2}+\alpha_{3}+\alpha_{4})}$ is the standard two dimensional irreducible $\tilde{L}_{\alpha_{4}}$-module, by Lemma \ref{lemma1.3}(2) we have 
\begin{center}
	$H^0(\tilde{L}_{\alpha_{4}}/\tilde{B}_{4},\mathbb{C}_{-(\alpha_{2}+\alpha_{3})}\oplus \mathbb{C}_{-(\alpha_{2}+\alpha_{3}+\alpha_{4})} )=\mathbb{C}_{-(\alpha_{2}+\alpha_{3})}\oplus \mathbb{C}_{-(\alpha_{2}+\alpha_{3}+\alpha_{4})}.$
\end{center} 
Further, since  $\langle -(\alpha_{2}+2\alpha_{3}+\alpha_{4}),\alpha_{4}\rangle =0,$ by Lemma \ref{lemma1.3}(2) we have 
\begin{center}
	$H^0(\tilde{L}_{\alpha_{4}}/\tilde{B}_{4}, \mathbb{C}_{-(\alpha_{2}+2\alpha_{3}+\alpha_{4})})= \mathbb{C}_{-(\alpha_{2}+2\alpha_{3}+\alpha_{4})}.$
\end{center} 
Therefore we have 
\begin{center}
	$H^0(s_{4}s_{3}s_{4}s_{2}s_{3},\alpha_{3})= \mathbb{C}_{-(\alpha_{2}+\alpha_{3})}\oplus \mathbb{C}_{-(\alpha_{2}+\alpha_{3}+\alpha_{4})}\oplus\mathbb{C}_{-(\alpha_{2}+2\alpha_{3}+\alpha_{4})}.$
\end{center}
Proof of (2): By Corollary \ref{rem 4.8}(1) we have 
\begin{center} 	$H^0(s_{4}s_{2}s_{3}s_{4}s_{1}s_{2}s_{3},\alpha_{3})$=$ \mathbb{C}_{-(\alpha_{1}+\alpha_{2}+\alpha_{3})}\oplus \mathbb{C}_{-(\alpha_{1}+\alpha_{2}+\alpha_{3}+\alpha_{4})}\oplus \mathbb{C}_{-(\alpha_{1}+\alpha_{2}+2\alpha_{3}+\alpha_{4})}\oplus \mathbb{C}_{-(\alpha_{1}+2\alpha_{2}+2\alpha_{3}+\alpha_{4})}\oplus \mathbb{C}_{-(\alpha_{1}+2\alpha_{2}+3\alpha_{3}+\alpha_{4})}.$
	
\end{center}
Since  $\mathbb{C}_{-(\alpha_{1}+\alpha_{2}+\alpha_{3}+\alpha_{4})}\oplus \mathbb{C}_{-(\alpha_{1}+\alpha_{2}+2\alpha_{3}+\alpha_{4})},$ 
$ \mathbb{C}_{-(\alpha_{1}+2\alpha_{2}+2\alpha_{3}+\alpha_{4})}\oplus \mathbb{C}_{-(\alpha_{1}+2\alpha_{2}+3\alpha_{3}+\alpha_{4})}$ are the two dimensional irreducible $\tilde{L}_{\alpha_{3}}$-modules and $\langle-(\alpha_{1}+\alpha_{2}+\alpha_{3}), \alpha_{3} \rangle =0,$ by Lemma \ref{lemma1.3}(2) we have
\begin{center} 	$H^0(s_{3}s_{4}s_{2}s_{3}s_{4}s_{1}s_{2}s_{3},\alpha_{3})$=$ \mathbb{C}_{-(\alpha_{1}+\alpha_{2}+\alpha_{3})}\oplus \mathbb{C}_{-(\alpha_{1}+\alpha_{2}+\alpha_{3}+\alpha_{4})}\oplus \mathbb{C}_{-(\alpha_{1}+\alpha_{2}+2\alpha_{3}+\alpha_{4})}\oplus \mathbb{C}_{-(\alpha_{1}+2\alpha_{2}+2\alpha_{3}+\alpha_{4})}\oplus \mathbb{C}_{-(\alpha_{1}+2\alpha_{2}+3\alpha_{3}+\alpha_{4})}.$
\end{center}
Since $\mathbb{C}_{-(\alpha_{1}+\alpha_{2}+\alpha_{3})}\oplus \mathbb{C}_{-(\alpha_{1}+\alpha_{2}+\alpha_{3}+\alpha_{4})}$ is the standard two dimensional irreducible $\tilde{L}_{\alpha_{4}}$-module and $\langle -(\alpha_{1}+2\alpha_{2}+2\alpha_{3}+\alpha_{4}),\alpha_{4} \rangle=0,$ by Lemma \ref{lemma1.3}(2) we have 
\begin{center}
	$H^0(\tilde{L}_{\alpha_{4}}/\tilde{B}_{\alpha_{4}},\mathbb{C}_{-(\alpha_{1}+\alpha_{2}+\alpha_{3})}\oplus \mathbb{C}_{-(\alpha_{1}+\alpha_{2}+\alpha_{3}+\alpha_{4})} )=\mathbb{C}_{-(\alpha_{1}+\alpha_{2}+\alpha_{3})}\oplus \mathbb{C}_{-(\alpha_{1}+\alpha_{2}+\alpha_{3}+\alpha_{4})}.$
\end{center}
Moreover, since $\langle -(\alpha_{1}+\alpha_{2}+2\alpha_{3}+\alpha_{4}),\alpha_{4}\rangle=0 $ and $\langle -(\alpha_{1}+2\alpha_{2}+2\alpha_{3}+\alpha_{4}),\alpha_{4}\rangle=0,$ by Lemma \ref{lemma1.3}(2) we have 

\begin{center}
$H^0(\tilde{L}_{\alpha_{4}}/\tilde{B}_{\alpha_{4}}, \mathbb{C}_{-(\alpha_{1}+\alpha_{2}+2\alpha_{3}+\alpha_{4})} )= \mathbb{C}_{-(\alpha_{1}+\alpha_{2}+2\alpha_{3}+\alpha_{4})}$
\end{center}
and 
\begin{center}
	$H^0(\tilde{L}_{\alpha_{4}}/\tilde{B}_{\alpha_{4}}, \mathbb{C}_{-(\alpha_{1}+2\alpha_{2}+2\alpha_{3}+\alpha_{4})} )= \mathbb{C}_{-(\alpha_{1}+2\alpha_{2}+2\alpha_{3}+\alpha_{4})}.$
\end{center}
Since $\langle -(\alpha_{1}+2\alpha_{2}+3\alpha_{3}+\alpha_{4}),\alpha_{4}\rangle=1,$ by Lemma \ref{lemma1.3} we have 
 \begin{center}
 $H^0(\tilde{L}_{\alpha_{4}}/\tilde{B}_{\alpha_{4}}, \mathbb{C}_{-(\alpha_{1}+2\alpha_{2}+3\alpha_{3}+\alpha_{4})} )= \mathbb{C}_{-(\alpha_{1}+2\alpha_{2}+3\alpha_{3}+\alpha_{4})}\oplus \mathbb{C}_{-(\alpha_{1}+2\alpha_{2}+3\alpha_{3}+2\alpha_{4})}.$
\end{center}
Therefore combinng the above discussion we have
\begin{center} 	$H^0(s_{4}s_{3}s_{4}s_{2}s_{3}s_{4}s_{1}s_{2}s_{3},\alpha_{3})$=$ \mathbb{C}_{-(\alpha_{1}+\alpha_{2}+\alpha_{3})}\oplus \mathbb{C}_{-(\alpha_{1}+\alpha_{2}+\alpha_{3}+\alpha_{4})}\oplus \mathbb{C}_{-(\alpha_{1}+\alpha_{2}+2\alpha_{3}+\alpha_{4})}\oplus \mathbb{C}_{-(\alpha_{1}+2\alpha_{2}+2\alpha_{3}+\alpha_{4})}\oplus \mathbb{C}_{-\omega_{4}+\alpha_{4}}\oplus \mathbb{C}_{-\omega_{4}}$
\end{center}
since $\omega_{4}=\alpha_{1}+2\alpha_{2}+3\alpha_{3}+2\alpha_{4}.$
\end{proof}
\section{computions of relative tangent bundles $H^1(w, \alpha_{2})$}

In this section we compute cohomology modules $H^1(w, \alpha_{2})$ corresponding to some special Weyl group elements.

\begin{lem}\label{lem 5.1}
\item[(1)]
 $H^1(w_{r}, \alpha_{2})=0$ for $r=1,2,5.$
 
\item[(2)] $H^1(w_{3}, \alpha_{2})=\mathbb{C}_{-\omega_{4}+\alpha_{4}}.$

\item[(3)] $H^1(w_{4}, \alpha_{2})=\mathbb{C}_{-\omega_{4}}.$

\end{lem}
\begin{proof} It is easy to see $H^1(s_{2}, \alpha_{2} )=0.$
 
Note that we have 
\begin{center}
$H^0(s_{2}, \alpha_{2})=\mathbb{C}h(\alpha_{2})\oplus \mathbb{C}_{-\alpha_{2}}\oplus \mathbb{C}_{\alpha_{2}}.$
\end{center}

Since $\langle-\alpha_{2},\alpha_{1} \rangle=1,$ by using Lemma \ref {lemma1.3}(2), Lemma \ref {lemma1.3}(4) we have  
\begin{center}
	$H^1(s_{1},H^0(s_{2}, \alpha_{2}))=0.$
\end{center} 
Since  $H^1(s_{2}, \alpha_{2} )=0,$ by using Lemma \ref{lemma1.3}(1) we have 
\begin{center}
	$H^0(s_{1},H^1(s_{2}, \alpha_{2} ))=0.$
\end{center}
Thus by using SES and the above discussion we have
\begin{center}
	$H^1(s_{1}s_{2}, \alpha_{2})=0.$
\end{center}

By using SES and Lemma \ref{lemma1.3}(2) we have
\begin{center}
 $H^0(s_{1}s_{2}, \alpha_{2})=\mathbb{C}h(\alpha_{2})\oplus \mathbb{C}_{-\alpha_{2}}\oplus \mathbb{C}_{-(\alpha_{1} + \alpha_{2})}$.
\end{center} 

Since $\langle \alpha_{2}, \alpha_{4}\rangle =0,$ by using Lemma \ref{lemma1.3}(2) we have 
\begin{center}
 $H^1(s_{4},H^0(s_{1}s_{2}, \alpha_{2}))=0$
\end{center} 
 
 and 

$$H^0(s_{4}s_{1}s_{2}, \alpha_{2})=\mathbb{C}h(\alpha_{2})\oplus \mathbb{C}_{-\alpha_{2}}\oplus \mathbb{C}_{-(\alpha_{1} + \alpha_{2})}.\hspace{6.3cm} (5.1.1)$$
 
Again, since  $H^1(s_{1}s_{2}, \alpha_{2})=0,$ by using Lemma \ref{lemma1.3}(1) we have
\begin{center}
	$H^0(s_{4},H^1(s_{1}s_{2}, \alpha_{2}))=0.$
\end{center}

Thus by using SES and the above discussion we have 
\begin{center}
$H^1(s_{4}s_{1}s_{2}, \alpha_{2})=0.\hspace{11.5cm} (5.1.2)$
\end{center}

Since $\langle -\alpha_{2}, \alpha_{3}\rangle=2,$ $\langle -(\alpha_{1}+\alpha_{2}), \alpha_{3}\rangle=2,$ by using $(5.1.1),$ and Lemma \ref{lemma1.3}(2) we have 
\begin{center}
	$H^1(s_{3}, H^0(s_{4}s_{1}s_{2}, \alpha_{2}))=0.$
\end{center}
Further, by $(5.1.2)$ we have 
\begin{center}
	$H^0(s_{3}, H^1(s_{4}s_{1}s_{2}, \alpha_{2}))=0.$
\end{center}
Thus by using SES we have 
\begin{center}
$H^1(s_{3}s_{4}s_{1}s_{2}, \alpha_{2})=0.\hspace{11cm}(5.1.3)$
\end{center}
Therefore we have 
\begin{center}
$H^0(s_{2}, H^1(s_{3}s_{4}s_{1}s_{2}, \alpha_{2}))=0.$
\end{center}

By using Lemma \ref{lem 3.3} we have 
\begin{center}
	$H^1(s_{2},H^0(s_{3}s_{4}s_{1}s_{2}, \alpha_{2}))=0.$
\end{center}

Thus by SES we have 
\begin{center}
$H^1(s_{2}s_{3}s_{4}s_{1}s_{2}, \alpha_{2})=0.\hspace{10.5cm}(5.1.4)$
\end{center}

Therefore  we have 
\begin{center}
	$H^0(s_{1},H^1(s_{2}s_{3}s_{4}s_{1}s_{2}, \alpha_{2}))=0.$
\end{center}

By Lemma \ref{lem 3.3} we have
\begin{center} $H^1(s_{1},H^0(s_{2}s_{3}s_{4}s_{1}s_{2},\alpha_{2}))=0.$
 \end{center}

Therefore by using SES we have  
\begin{center}
 	$H^1(w_{1},\alpha_{2})=0.$
 \end{center}

Since $H^1(w_{1},\alpha_{2})=0,$ we have 
\begin{center}
	$H^0(s_{4},H^1(w_{1},\alpha_{2}))=0.$
\end{center}

Recall that by $(4.1.3)$ we have 
\begin{center}
	$H^0(w_{1},\alpha_{2})=\mathbb{C}_{-(\alpha_{1}+2\alpha_{2}+2\alpha_{3})}\oplus \mathbb{C}_{-(\alpha_{2}+2\alpha_{3})} \oplus \mathbb{C}_{-(\alpha_{1}+\alpha_{2}+2\alpha_{3})}.$
\end{center}

Since  $\langle -(\alpha_{2}+2\alpha_{3}), \alpha_{4}\rangle=2,$ $\langle -(\alpha_{1}+\alpha_{2}+2\alpha_{3}), \alpha_{4}\rangle=2,$ $\langle -(\alpha_{1}+ 2\alpha_{2}+2\alpha_{3}), \alpha_{4}\rangle=2,$ by using Lemma \ref {lemma1.3}(2) we have
\begin{center}
	$H^1(s_{4},H^0(w_{1},\alpha_{2}))=0.$
\end{center}
Thus by using SES and the above discussion we have
\begin{center}
	$H^1(s_{4}w_{1},\alpha_{2})=0\hspace{11cm} (5.1.5)$
\end{center}
and 
\begin{center}
	$H^0(s_{4}w_{1}, \alpha_{2})=(\mathbb{C}_{-(\alpha_{1}+2\alpha_{2}+2\alpha_{3})}\oplus\mathbb{C}_{-(\alpha_{1}+2\alpha_{2}+2\alpha_{3}+\alpha_{4})} \oplus \mathbb{C}_{-(\alpha_{1}+2\alpha_{2}+2\alpha_{3}+2\alpha_{4})})\oplus (\mathbb{C}_{-(\alpha_{2}+2\alpha_{3})}\oplus \mathbb{C}_{-(\alpha_{2}+2\alpha_{3}+\alpha_{4})}\oplus \mathbb{C}_{-(\alpha_{2}+2\alpha_{3}+2\alpha_{4})})\oplus (\mathbb{C}_{-(\alpha_{1}+\alpha_{2}+2\alpha_{3})}\oplus \mathbb{C}_{-(\alpha_{1}+\alpha_{2}+2\alpha_{3}+\alpha_{4})} \oplus \mathbb{C}_{-(\alpha_{1}+\alpha_{2}+2\alpha_{3}+2\alpha_{4})}).$	
\end{center}
Since 	$H^1(s_{4}w_{1},\alpha_{2})=0,$ we have 
\begin{center}
	$H^0(s_{3},H^1(s_{4}w_{1},\alpha_{2}))=0.$
\end{center}	
Since  $\langle-(\alpha_{2}+2\alpha_{3}+2\alpha_{4}), \alpha_{3} \rangle=0,$ $\langle-(\alpha_{1}+2\alpha_{2}+2\alpha_{3}), \alpha_{3} \rangle=0 ,$ $\langle-(\alpha_{1}+\alpha_{2}+2\alpha_{3}+2\alpha_{4}),\alpha_{3} \rangle =0,$ $\langle -(\alpha_{1}+2 \alpha_{2}+2\alpha_{3}+\alpha_{4}), \alpha_{3}\rangle =1,$ 
$\langle -(\alpha_{1}+2 \alpha_{2}+2\alpha_{3}+2\alpha_{4}), \alpha_{3}\rangle =2,$  $\langle -(\alpha_{2}+2\alpha_{3}), \alpha_{3} \rangle=-2,$ $\langle -(\alpha_{1}+\alpha_{2}+2\alpha_{3}), \alpha_{3} \rangle=-2,$ $\langle -(\alpha_{2}+2\alpha_{3}+\alpha_{4}), \alpha_{3} \rangle=-1,$  $\langle-(\alpha_{1}+\alpha_{2}+2\alpha_{3}+\alpha_{4}),\alpha_{3} \rangle =-1,$
by using Lemma \ref{lemma1.3} we have  
\begin{center}	
$H^1(s_{3},H^0(s_{4}w_{1},\alpha_{2}))=\mathbb{C}_{-(\alpha_{2}+\alpha_{3})}\oplus \mathbb{C}_{-(\alpha_{1}+\alpha_{2}+\alpha_{3})}.$
\end{center}
Thus by using SES and the above discussion we have
\begin{center}
	$H^1(s_{3}s_{4}w_{1},\alpha_{2})=\mathbb{C}_{-(\alpha_{2}+\alpha_{3})}\oplus \mathbb{C}_{-(\alpha_{1}+\alpha_{2}+\alpha_{3})},\hspace{6.8cm}(5.1.6)$
\end{center}
and
\begin{center}
	$H^0(s_{3}s_{4}w_{1},\alpha_{2})=\mathbb{C}_{-(\alpha_{2}+2\alpha_{3}+2\alpha_{4})}\oplus \mathbb{C}_{-(\alpha_{1}+2\alpha_{2}+2\alpha_{3})} \oplus \mathbb{C}_{-(\alpha_{1}+\alpha_{2}+2\alpha_{3}+2\alpha_{4})}\oplus \mathbb{C}_{-(\alpha_{1}+2\alpha_{2}+2\alpha_{3}+\alpha_{4})}\oplus \mathbb{C}_{-(\alpha_{1}+2\alpha_{2}+3\alpha_{3}+\alpha_{4})}\oplus \mathbb{C}_{-(\alpha_{1}+2\alpha_{2}+2\alpha_{3}+2\alpha_{4})}\oplus \mathbb{C}_{-(\alpha_{1}+2\alpha_{2}+3\alpha_{3}+2\alpha_{4})} \oplus \mathbb{C}_{-(\alpha_{1}+2\alpha_{2}+4\alpha_{3}+2\alpha_{4})}.\hspace{11.5cm}(5.1.7)$
\end{center}
Since 	$H^1(s_{3}s_{4}w_{1},\alpha_{2})=\mathbb{C}_{-(\alpha_{2}+\alpha_{3})}\oplus \mathbb{C}_{-(\alpha_{1}+\alpha_{2}+\alpha_{3})},$ by using Lemma \ref{lemma1.3} we have
\begin{center}
$H^0(s_{2},H^1(s_{3}s_{4}w_{1},\alpha_{2}))= \mathbb{C}_{-(\alpha_{1}+\alpha_{2}+\alpha_{3})}.$	
\end{center} 
By Lemma \ref{lem 3.3} we have 
\begin{center}	
$H^1(s_{2},H^0(s_{3}s_{4}w_{1},\alpha_{2}))=0.$
\end{center}	
Thus using SES and above discussion we have
\begin{center}	
	$H^1(s_{2}s_{3}s_{4}w_{1},\alpha_{2})$=$\mathbb{C}_{-(\alpha_{1}+\alpha_{2}+\alpha_{3})}.\hspace{9cm}(5.1.8)$
\end{center}
Since $\mathbb{C}_{-(\alpha_{1}+\alpha_{2}+2\alpha_{3}+2\alpha_{4})}\oplus \mathbb{C}_{-(\alpha_{1}+2\alpha_{2}+2\alpha_{3}+2\alpha_{4})}$ is the standard two dimensional irreducible $\tilde{L}_{\alpha_{2}}$-module and  $\langle-(\alpha_{2}+2\alpha_{3}+2\alpha_{4}), \alpha_{2} \rangle=0 ,$  $\langle -(\alpha_{1}+2 \alpha_{2}+3\alpha_{3}+\alpha_{4}), \alpha_{2}\rangle =0,$ 
$\langle -(\alpha_{1}+2 \alpha_{2}+3\alpha_{3}+2\alpha_{4}), \alpha_{2}\rangle =0,$ $\langle-(\alpha_{1}+2\alpha_{2}+4\alpha_{3}+2\alpha_{4}),\alpha_{2} \rangle =1,$ $\langle -(\alpha_{1}+2\alpha_{2}+2\alpha_{3}), \alpha_{2} \rangle=-1,$ $\langle -(\alpha_{1}+2\alpha_{2}+2\alpha_{3}+\alpha_{4}), \alpha_{2} \rangle=-1,$  
by using SES and Lemma \ref{lemma1.3} we have  
\begin{center}
	$H^0(s_{2}s_{3}s_{4}w_{1},\alpha_{2})=\mathbb{C}_{-(\alpha_{2}+2\alpha_{3}+2\alpha_{4})}\oplus \mathbb{C}_{-(\alpha_{1}+2\alpha_{2}+3\alpha_{3}+\alpha_{4})} \oplus \mathbb{C}_{-(\alpha_{1}+2\alpha_{2}+3\alpha_{3}+2\alpha_{4})} \oplus \mathbb{C}_{-(\alpha_{1}+\alpha_{2}+2\alpha_{3}+2\alpha_{4})}\oplus\mathbb{C}_{-(\alpha_{1}+2\alpha_{2}+2\alpha_{3}+2\alpha_{4})} \oplus \mathbb{C}_{-(\alpha_{1}+2\alpha_{2}+4\alpha_{3}+2\alpha_{4})}\oplus \mathbb{C}_{-(\alpha_{1}+3\alpha_{2}+4\alpha_{3}+2\alpha_{4})}.\hspace{11cm}(5.1.9)$	
	\end{center}
Since $\langle-(\alpha_{1}+\alpha_{2}+\alpha_{3}),\alpha_{1} \rangle=-1,$ by using Lemma \ref{lemma1.3}(4) we have
\begin{center}
		$H^0(s_{1},H^1(s_{2}s_{3}s_{4}w_{1},\alpha_{2}))=0.$
\end{center}	
Further, by Lemma \ref{lem 3.3} we have 
\begin{center}	
$H^1(s_{1},H^0(s_{2}s_{3}s_{4}w_{1},\alpha_{2}))= 0.$
\end{center}
Thus using SES we have	
\begin{center}	
	$H^1(w_{2},\alpha_{2})= 0.$
\end{center}
Since
$\mathbb{C}_{-(\alpha_{2}+2\alpha_{3}+2\alpha_{4})}\oplus
\mathbb{C}_{-(\alpha_{1}+\alpha_{2}+2\alpha_{3}+2\alpha_{4})}$ is the standard two dimensional irreducible $\tilde{L}_{\alpha_{1}}$-module and $\langle-(\alpha_{1}+2\alpha_{2}+2\alpha_{3}+2\alpha_{4}), \alpha_{1}\rangle =0,$ $\langle-(\alpha_{1}+2\alpha_{2}+3\alpha_{3}+\alpha_{4}), \alpha_{1}\rangle =0,$ $\langle-(\alpha_{1}+2\alpha_{2}+3\alpha_{3}+2\alpha_{4}), \alpha_{1}\rangle =0,$ and $\langle-(\alpha_{1}+2\alpha_{2}+4\alpha_{3}+2\alpha_{4}), \alpha_{1}\rangle =0,$ 	
$\langle-(\alpha_{1}+3\alpha_{2}+4\alpha_{3}+2\alpha_{4}), \alpha_{1} \rangle =1,$ by using SES and Lemma \ref{lemma1.3} we have
\begin{center}
	$H^0(w_{2},\alpha_{2})=\mathbb{C}_{-(\alpha_{2}+2\alpha_{3}+2\alpha_{4})}\oplus \mathbb{C}_{-(\alpha_{1}+\alpha_{2}+2\alpha_{3}+2\alpha_{4})} \oplus \mathbb{C}_{-(\alpha_{1}+2\alpha_{2}+2\alpha_{3}+2\alpha_{4})} \oplus \mathbb{C}_{-(\alpha_{1}+2\alpha_{2}+3\alpha_{3}+\alpha_{4})}\oplus\mathbb{C}_{-(\alpha_{1}+2\alpha_{2}+3\alpha_{3}+2\alpha_{4})} \oplus \mathbb{C}_{-(\alpha_{1}+2\alpha_{2}+4\alpha_{3}+2\alpha_{4})}\oplus \mathbb{C}_{-(\alpha_{1}+3\alpha_{2}+4\alpha_{3}+2\alpha_{4})}\oplus \mathbb{C}_{-(2\alpha_{1}+3\alpha_{2}+4\alpha_{3}+2\alpha_{4})}.$
\end{center}

Since 	$H^1(w_{2},\alpha_{2})= 0,$ we have
\begin{center}	
	$H^0(s_{4},H^1(w_{2},\alpha_{2}))= 0.$
\end{center} 

Since  $\mathbb{C}_{-(\alpha_{1}+2\alpha_{2}+3\alpha_{3}+\alpha_{4})}\oplus \mathbb{C}_{-(\alpha_{1}+2\alpha_{2}+3\alpha_{3}+2\alpha_{4})}$
is the standard two dimensional irreducible $\tilde{L}_{\alpha_{4}}$-module and $\langle-(\alpha_{1}+2\alpha_{2}+4\alpha_{3}+2\alpha_{4}), \alpha_{4} \rangle=0,$ $\langle-(\alpha_{1}+3\alpha_{2}+4\alpha_{3}+2\alpha_{4}), \alpha_{4} \rangle=0,$  $\langle-(2\alpha_{1}+3\alpha_{2}+4\alpha_{3}+2\alpha_{4}), \alpha_{4} \rangle=0,$ $\langle-(\alpha_{2}+2\alpha_{3}+2\alpha_{4}), \alpha_{4} \rangle=-2,$ $\langle-(\alpha_{1}+\alpha_{2}+2\alpha_{3}+2\alpha_{4}), \alpha_{4} \rangle=-2 ,$   $\langle-(\alpha_{1}+2\alpha_{2}+2\alpha_{3}+2\alpha_{4}), \alpha_{4} \rangle=-2,$  by using SES and Lemma \ref{lemma1.3} we have 
\begin{center}
$H^1(s_{4}, H^0(w_{2},\alpha_{2}))= \mathbb{C}_{-(\alpha_{2}+2\alpha_{3}+\alpha_{4})}\oplus\mathbb{C}_{-(\alpha_{1}+\alpha_{2}+2\alpha_{3}+\alpha_{4})} \oplus \mathbb{C}_{-(\alpha_{1}+2\alpha_{2}+2\alpha_{3}+\alpha_{4})}.$
\end{center}

Thus from the above discussion we have 
\begin{center}
	$H^1(s_{4}w_{2},\alpha_{2})$=$ \mathbb{C}_{-(\alpha_{2}+2\alpha_{3}+\alpha_{4})}\oplus\mathbb{C}_{-(\alpha_{1}+\alpha_{2}+2\alpha_{3}+\alpha_{4})} \oplus \mathbb{C}_{-(\alpha_{1}+2\alpha_{2}+2\alpha_{3}+\alpha_{4})}\hspace{11cm} (5.1.10)$
\end{center}
and 

\begin{center}	
$H^0(s_{4}w_{2},\alpha_{2})= \mathbb{C}_{-(\alpha_{1}+2\alpha_{2}+3\alpha_{3}+\alpha_{4})}\oplus\mathbb{C}_{-(\alpha_{1}+2\alpha_{2}+3\alpha_{3}+2\alpha_{4})} \oplus \mathbb{C}_{-(\alpha_{1}+2\alpha_{2}+4\alpha_{3}+2\alpha_{4})}\oplus \mathbb{C}_{-(\alpha_{1}+3\alpha_{2}+4\alpha_{3}+2\alpha_{4})}\oplus \mathbb{C}_{-(2\alpha_{1}+3\alpha_{2}+4\alpha_{3}+2\alpha_{4})}.$
\end{center}

Since $\langle -(\alpha_{1}+2\alpha_{2}+2\alpha_{3}+\alpha_{4}), \alpha_{3} \rangle=1,$ $\langle -(\alpha_{2}+2\alpha_{3}+\alpha_{4}), \alpha_{3} \rangle=-1,$ and $\langle -(\alpha_{1}+\alpha_{2}+2\alpha_{3}+\alpha_{4}), \alpha_{3} \rangle=-1,$ by using Lemma \ref{lemma1.3} we have 
\begin{center}
$H^0(s_{3},H^1(s_{4}w_{2},\alpha_{2}))= \mathbb{C}_{-(\alpha_{1}+2\alpha_{2}+2\alpha_{3}+\alpha_{4})}\oplus \mathbb{C}_{-(\alpha_{1}+2\alpha_{2}+3\alpha_{3}+\alpha_{4})}.$	
\end{center} 

Since $\langle-(\alpha_{1}+3\alpha_{2}+4\alpha_{3}+2\alpha_{4}), \alpha_{3} \rangle=0,$ $\langle-(2\alpha_{1}+3\alpha_{2}+4\alpha_{3}+2\alpha_{4}), \alpha_{3} \rangle=0,$ $\langle-(\alpha_{1}+2\alpha_{2}+3\alpha_{3}+\alpha_{4}),\alpha_{3} \rangle =-1,$ and $\mathbb{C}_{-(\alpha_{1}+2\alpha_{2}+3\alpha_{3}+2\alpha_{4})}\oplus \mathbb{C}_{-(\alpha_{1}+2\alpha_{2}+4\alpha_{3}+2\alpha_{4})}=V\otimes\mathbb{C}_{-\omega_{3}}$(where $V$ is the standard two dimensional irreducible $\tilde{L}_{\alpha_{3}}$-module), by Lemma \ref{lemma1.3}  we have
\begin{center}
	 $H^1(s_{3}, H^0(s_{4}w_{2},\alpha_{2}))=0.$
\end{center}
and  

\begin{center}
$H^0(s_{3}s_{4}w_{2},\alpha_{2})$=$\mathbb{C}_{-(\alpha_{1}+3\alpha_{2}+4\alpha_{3}+2\alpha_{4})}\oplus \mathbb{C}_{-(2\alpha_{1}+3\alpha_{2}+4\alpha_{3}+2\alpha_{4})}.\hspace{4.5cm}(5.1.11)$
\end{center} 

Therefore we have 
\begin{center}
$H^1(s_{3}s_{4}w_{2},\alpha_{2})$=$ \mathbb{C}_{-(\alpha_{1}+2\alpha_{2}+2\alpha_{3}+\alpha_{4})}\oplus \mathbb{C}_{-(\alpha_{1}+2\alpha_{2}+3\alpha_{3}+\alpha_{4})}.\hspace{5cm}(5.1.12)$ 	
\end{center}

Since  $\langle -(\alpha_{1}+2\alpha_{2}+3\alpha_{3}+\alpha_{4}),\alpha_{2} \rangle =0,$ and $\langle -(\alpha_{1}+2\alpha_{2}+2\alpha_{3}+\alpha_{4}),\alpha_{2} \rangle =-1,$  by using Lemma \ref{lemma1.3} we have 
\begin{center}
$H^0(s_{2},	H^1(s_{3}s_{4}w_{2},\alpha_{2}))= \mathbb{C}_{-(\alpha_{1}+2\alpha_{2}+3\alpha_{3}+\alpha_{4})}.$	
\end{center}

By Lemma \ref{lem 3.3} we have
\begin{center}
$H^1(s_{2},H^0(s_{3}s_{4}w_{2},\alpha_{2}))=0.$			
\end{center}

Thus from the above discussion we have 
\begin{center}
$H^1(s_{2}s_{3}s_{4}w_{2},\alpha_{2})$=$\mathbb{C}_{-(\alpha_{1}+2\alpha_{2}+3\alpha_{3}+\alpha_{4})}$=$\mathbb{C}_{-\omega_{4}+\alpha_{4}}.\hspace{6.4cm}(5.1.13)$	
\end{center}

Since $\langle -(2\alpha_{1}+3\alpha_{2}+4\alpha_{3}+2\alpha_{4}),\alpha_{2}\rangle=0$ and $\langle-(\alpha_{1}+3\alpha_{2}+4\alpha_{3}+2\alpha_{4}),\alpha_{2}\rangle=-1,$ by using SES and Lemma \ref{lemma1.3} we have 
\begin{center}
	$H^0(s_{2}s_{3}s_{4}w_{2},\alpha_{2})$=$\mathbb{C}_{-(2\alpha_{1}+3\alpha_{2}+4\alpha_{3}+2\alpha_{4})}.\hspace{8cm}(5.1.14)$
\end{center}

Since $\langle-(\alpha_{1}+2\alpha_{2}+3\alpha_{3}+\alpha_{4}),\alpha_{1}\rangle=0,$ by using Lemma \ref{lemma1.3}(2) we have
\begin{center}
$H^0(s_{1},H^1(s_{2}s_{3}s_{4}w_{2},\alpha_{2}))=\mathbb{C}_{-(\alpha_{1}+2\alpha_{2}+3\alpha_{3}+\alpha_{4})}.$	
\end{center} 	

By Lemma \ref{lem 3.3} we have 
\begin{center}
$H^1(s_{1},H^0(s_{2}s_{3}s_{4}w_{2},\alpha_{2}))=0.$
\end{center}	

Thus from the above discussion we have 
\begin{center}
$H^1(w_{3},\alpha_{2})=\mathbb{C}_{-\omega_{4}+\alpha_{4}}$
\end{center}
since $\omega_{4}=\alpha_{1}+2\alpha_{2}+3\alpha_{3}+2\alpha_{4}.$
This proves (2).

Since we have $H^0(w_{3},\alpha_{2})=0$ (see Lemma \ref{lem 4.1}), by using SES we have 
\begin{center}
$H^1(s_{4}w_{3},\alpha_{2})=H^0(s_{4},	H^1(w_{3},\alpha_{2})).$
\end{center}

Since $\langle -\omega_{4} + \alpha_{4},\alpha_{4}\rangle =1,$ by using Lemma \ref{lemma1.3}(2) we have 
\begin{center}
$H^1(s_{4}w_{3},\alpha_{2})=\mathbb{C}_{-\omega_{4}+\alpha_{4}}\oplus \mathbb{C}_{-\omega_{4}}.\hspace{9cm}(5.1.15)$
\end{center}

Since we have $H^0(w_{3},\alpha_{2})=0$ (see Lemma \ref{lem 4.1}), by using SES we have 
\begin{center}
	$H^1(s_{3}s_{4}w_{3},\alpha_{2})=H^0(s_{3},	H^1(s_{4}w_{3},\alpha_{2})).$
\end{center}
Since $\langle -\omega_{4},\alpha_{3}\rangle =0$ and $\langle -\omega_{4}+\alpha_{4},\alpha_{3}\rangle =-1,$ by using Lemma \ref{lemma1.3}(2), Lemma \ref{lemma1.3}(4) we have 
\begin{center}
$H^1(s_{3}s_{4}w_{3},\alpha_{2})=\mathbb{C}_{-\omega_{4}}.\hspace{10.5cm}(5.1.16)$
\end{center}
Since we have $H^0(w_{3},\alpha_{2})=0$ (see Lemma \ref{lem 4.1}) and $\alpha_{1},\alpha_{2}$ are orthogonal to $\omega_{4},$ by Lemma \ref{lemma1.3}(2) we have 
\begin{center}
	$H^1(w_{4},\alpha_{2})=\mathbb{C}_{-\omega_{4}}.$
\end{center}
This gives the proof of (3).

Since we have $H^0(w_{3},\alpha_{2})=0$ (see Lemma \ref{lem 4.1}) and $\langle -\omega_{4}, \alpha_{4}\rangle =-1,$ by using Lemma \ref{lemma1.3}(4) we have 
\begin{center}
	$H^1(s_{4}w_{4},\alpha_{2})=0.$
\end{center}
Since by Lemma \ref{lem 4.1} we have $H^0(w_{3},\alpha_{2})=0,$ and 	$H^1(s_{4}w_{4},\alpha_{2})=0,$ by using SES repeatedly we have 
\begin{center}
	$H^1(w_{5},\alpha_{2})=0.$
\end{center}
This completes the proof of $(1)$.
\end{proof}

\begin{cor}\label{cor 5.2}

\item [(1)]	$H^1(s_{4}s_{1}s_{2}, \alpha_{2})=0.$
	
\item [(2)] $H^1(s_{4}w_{r}, \alpha_{2})=0$ for $r=1,4,5.$
 
\item[(3)] $H^1(s_{4}w_{2}, \alpha_{2})=\mathbb{C}_{-(\alpha_{2} + 2\alpha_{3} + \alpha_{4})}\oplus \mathbb{C}_{-(\alpha_{1} + \alpha_{2} + 2\alpha_{3} + \alpha_{4})}\oplus \mathbb{C}_{-(\alpha_{1} + 2\alpha_{2} +2 \alpha_{3} + \alpha_{4})}.$
 
\item[(4)] $H^1(s_{4}w_{3}, \alpha_{2})=\mathbb{C}_{-\omega_{4}}\oplus \mathbb{C}_{-\omega_{4} + \alpha_{4}}.$

\end{cor}	
\begin{proof}
	Proof of (1) follows from $(5.1.2)$. 
	
	Proof of (2) for $r=1,$ proof follows from $(5.1.5).$  For $r=4,5$ proof follows by using SES, Lemma \ref{lem 5.1} and Lemma \ref{lem 4.1}. 
		
	Proof of (3) follows from $(5.1.10)$. 
	
	Proof of (4) follows from $(5.1.15).$ 
	 
\end{proof}
\begin{cor}\label{cor 5.3}
	
	\item[(1)] $H^1(s_{3}s_{4}s_{1}s_{2}, \alpha_{2})=0.$

	\item[(2)] $H^1(s_{3}s_{4}w_{1}, \alpha_{2})=\mathbb{C}_{-(\alpha_{2} + \alpha_{3})}\oplus \mathbb{C}_{-(\alpha_{1} + \alpha_{2} + \alpha_{3})}.$
	 
	\item[(3)] $H^1(s_{3}s_{4}w_{2}, \alpha_{2})=\mathbb{C}_{-(\alpha_{1} + 2\alpha_{2} + 2\alpha_{3} + \alpha_{4})}\oplus \mathbb{C}_{-\omega_{4}+ \alpha_{4}}.$
	
	\item[(4)] $H^1(s_{3}s_{4}w_{3}, \alpha_{2})=\mathbb{C}_{-\omega_{4}}.$
	
	\item [(5)]
	$H^1(s_{3}s_{4}w_{r}, \alpha_{2})=0$ for $r=4,5.$

\end{cor}
\begin{proof}
Proof of (1) follows from $(5.1.3).$
	
Proof of (2) follows from $(5.1.6).$
	
Proof of (3) follows from $(5.1.12).$

Proof of (4) follows from $(5.1.16).$
	
Proof of (5): By Lemma \ref{lem 4.1} we have $H^0(w_{r},\alpha_{2})=0$ for $r=4,5.$ Therefore $H^0(s_{4}w_{r},\alpha_{2})=0$ for $r=4,5.$ Hence we have $H^1(s_{3},H^0(s_{4}w_{r},\alpha_{2}))=0$ for $r=4,5.$ On the other hand, by Corollary \ref{cor 5.2}(2) we have $H^1(s_{4}w_{r}, \alpha_{2})=0$ for $r=4,5.$ Therefore $H^0(s_{3}, H^1(s_{4}w_{r}, \alpha_{2}))=0$ for $r=4,5.$ Thus by SES we have $H^1(s_{3}s_{4}w_{r}, \alpha_{2})=0$ for $r=4,5.$	
\end{proof}

\begin{cor}\label{cor 5.4}
\item[(1)] $H^1(s_{2}s_{3}s_{4}s_{1}s_{2}, \alpha_{2})=0.$

\item[(2)] $H^1(s_{2}s_{3}s_{4}w_{1}, \alpha_{2})=\mathbb{C}_{-(\alpha_{1} + \alpha_{2} + \alpha_{3} )}.$
	
\item[(3)] $H^1(s_{2}s_{3}s_{4}w_{2}, \alpha_{2})=\mathbb{C}_{-\omega_{4} +\alpha_{4}}.$

\item[(4)] $H^1(s_{2}s_{3}s_{4}w_{3}, \alpha_{2})=\mathbb{C}_{-\omega_{4}}.$
	
\item [(5)] $H^1(s_{2}s_{3}s_{4}w_{4}, \alpha_{2})=0.$
\end{cor}	
\begin{proof}
Proof of (1) follows from $(5.1.4).$
	
Proof of (2) follows from $(5.1.8).$
	
Proof of (3) follows from $(5.1.13).$

Proof of (4): By Lemma \ref{lem 4.1} we have $H^0(w_{3},\alpha_{2})=0.$ Therefore $H^0(s_{3}s_{4}w_{3},\alpha_{2})=0.$ Hence we have $H^1(s_{2},H^0(s_{3}s_{4}w_{3},\alpha_{2}))=0.$ On the other hand, by Corollary \ref{cor 5.3}(4) we have $H^1(s_{3}s_{4}w_{3}, \alpha_{2})=\mathbb{C}_{-\omega_{4}}.$ Since $\omega_{4}$ is orthogonal to $\alpha_{2},$ by Lemma\ref{lemma1.3}(2) we have $H^0(s_{2}, H^1(s_{3}s_{4}w_{3}, \alpha_{2}))=\mathbb{C}_{-\omega_{4}}.$ Thus by SES we have $H^1(s_{2}s_{3}s_{4}w_{3}, \alpha_{2})=\mathbb{C}_{-\omega_{4}}.$

Proof of (5): By Lemma \ref{lem 4.1}(2) we have $H^0(w_{4},\alpha_{2})=0.$ 

Therefore $H^0(s_{3}s_{4}w_{4},\alpha_{2})=0.$ Hence we have $H^1(s_{2},H^0(s_{3}s_{4}w_{4},\alpha_{2}))=0.$ On the other hand, by Corollary \ref{cor 5.3}(5) we have $H^1(s_{3}s_{4}w_{4}, \alpha_{2})=0.$ Therefore $H^0(s_{2}, H^1(s_{3}s_{4}w_{4}, \alpha_{2}))=0.$ Thus by SES we have $H^1(s_{2}s_{3}s_{4}w_{4}, \alpha_{2})=0.$	
\end{proof}
\begin{cor}\label{cor 5.5}
\item[(1)] $H^1(s_{4}s_{3}s_{4}s_{1}s_{2}, \alpha_{2})=0.$

\item[(2)] $H^1(s_{4}s_{3}s_{4}w_{1}, \alpha_{2})=\mathbb{C}_{-(\alpha_{2} + \alpha_{3} )}\oplus \mathbb{C}_{-(\alpha_{2} + \alpha_{3} +\alpha_{4})} \oplus \mathbb{C}_{-(\alpha_{1} + \alpha_{2} + \alpha_{3} )} \oplus \mathbb{C}_{-(\alpha_{1} + \alpha_{2} + \alpha_{3} + \alpha_{4})}\oplus \mathbb{C}_{-(\alpha_{2}+2\alpha_{3}+\alpha_{4})}\oplus\mathbb{C}_{-(\alpha_{1}+\alpha_{2}+2\alpha_{3}+\alpha_{4})} .$
	
\item[(3)] $H^1(s_{4}s_{3}s_{4}w_{2}, \alpha_{2})=\mathbb{C}_{-(\alpha_{1} + 2\alpha_{2} +2\alpha_{3} +\alpha_{4})}\oplus \mathbb{C}_{-\omega_{4} + \alpha_{4}} \oplus \mathbb{C}_{-\omega_{4}}.$
	
\item [(4)] $H^1(s_{4}s_{3}s_{4}w_{r}, \alpha_{2})=0$ for $r= 3,4.$
\end{cor}
\begin{proof}
Proof of (1): By $(4.1.1)$ if $H^0(s_{3}s_{4}s_{1}s_{2}, \alpha_{2} )_{\mu}\neq0,$ then we have $\langle \mu ,\alpha_{4} \rangle \ge 0.$
Thus using Lemma \ref{lemma1.3}(3) we have
\begin{center}
$H^1(s_{4},H^0(s_{3}s_{4}s_{1}s_{2}, \alpha_{2} ))=0.$		
\end{center} 
On the other hand, by using Corollary \ref{cor 5.3}(1) we have 
\begin{center}
$H^0(s_{4},H^1(s_{3}s_{4}s_{1}s_{2}, \alpha_{2} ))=0.$		
\end{center}
Hence we have $H^1(s_{4}s_{3}s_{4}s_{1}s_{2}, \alpha_{2})=0.$

Proof of (2): By $(5.1.7)$ the $\tilde{B}_{\alpha_{4}}$-indecomposable summands $V$ of $H^0(s_{3}s_{4}w_{1}, \alpha_{2})$ for which $H^1(s_{4}, V)\neq0$ are $\mathbb{C}_{-(\alpha_{2}+2\alpha_{3}+2\alpha_{4})}$ and $\mathbb{C}_{-(\alpha_{1}+\alpha_{2}+2\alpha_{3}+2\alpha_{4})}.$
Thus using Lemma \ref{lemma1.3}(3) we have
\begin{center}
	$H^1(s_{4},H^0(s_{3}s_{4}w_{1}, \alpha_{2} ))=\mathbb{C}_{-(\alpha_{2}+2\alpha_{3}+\alpha_{4})}\oplus\mathbb{C}_{-(\alpha_{1}+\alpha_{2}+2\alpha_{3}+\alpha_{4})} .$		
\end{center} 
On the other hand, by using Corollary \ref{cor 5.3}(2) and Lemma \ref{lemma1.3}(2) we have 
\begin{center}
	$H^0(s_{4},H^1(s_{3}s_{4}w_{1}, \alpha_{2} ))=\mathbb{C}_{-(\alpha_{2} + \alpha_{3} )}\oplus \mathbb{C}_{-(\alpha_{2} + \alpha_{3} +\alpha_{4})} \oplus \mathbb{C}_{-(\alpha_{1} + \alpha_{2} + \alpha_{3} )} \oplus \mathbb{C}_{-(\alpha_{1} + \alpha_{2} + \alpha_{3} + \alpha_{4})}.$		
\end{center}
Hence we have 
\begin{center}$H^1(s_{4}s_{3}s_{4}w_{1}, \alpha_{2})=\mathbb{C}_{-(\alpha_{2} + \alpha_{3} )}\oplus \mathbb{C}_{-(\alpha_{2} + \alpha_{3} +\alpha_{4})} \oplus \mathbb{C}_{-(\alpha_{1} + \alpha_{2} + \alpha_{3} )} \oplus \mathbb{C}_{-(\alpha_{1} + \alpha_{2} + \alpha_{3} + \alpha_{4})}\oplus \mathbb{C}_{-(\alpha_{2}+2\alpha_{3}+\alpha_{4})}\oplus\mathbb{C}_{-(\alpha_{1}+\alpha_{2}+2\alpha_{3}+\alpha_{4})}.$
\end{center}

Proof of (3):  By $(5.1.11)$ we have if $H^0(s_{3}s_{4}w_{2}, \alpha_{2} )_{\mu}\neq0,$ then $\langle \mu ,\alpha_{4} \rangle=0.$
Thus using Lemma \ref{lemma1.3}(3) we have
\begin{center}
$H^1(s_{4},H^0(s_{3}s_{4}w_{2}, \alpha_{2} ))=0.$		
\end{center} 
On the other hand, by using Corollary \ref{cor 5.3}(3) and Lemma \ref{lemma1.3}(2) we have 
\begin{center}
$H^0(s_{4},H^1(s_{3}s_{4}w_{2}, \alpha_{2} ))=\mathbb{C}_{-(\alpha_{1}+2\alpha_{2} +2 \alpha_{3} +\alpha_{4})}\oplus \mathbb{C}_{-(\alpha_{1}+2\alpha_{2} + 3\alpha_{3} +\alpha_{4})} \oplus \mathbb{C}_{-(\alpha_{1} + 2\alpha_{2} + 3\alpha_{3} + 2\alpha_{4})}.$		
\end{center}
Hence we have 
\begin{center}	
$H^1(s_{4}s_{3}s_{4}w_{2}, \alpha_{2})=\mathbb{C}_{-(\alpha_{1}+2\alpha_{2} +2 \alpha_{3} +\alpha_{4})}\oplus \mathbb{C}_{-\omega_{4}+\alpha_{4}} \oplus \mathbb{C}_{-\omega_{4}}.$
\end{center}	
	
Proof of (4):  By Lemma \ref{lem 4.1} we have $H^0(w_{r},\alpha_{2})=0$ for $r=3,4.$

Therefore $H^0(s_{3}s_{4}w_{r},\alpha_{2})=0$ for $r=3,4.$ Hence we have $H^1(s_{4},H^0(s_{3}s_{4}w_{r},\alpha_{2}))=0$ for $r=3,4.$ 

On the other hand, by Corollary \ref{cor 5.3}(4), Corollary \ref{cor 5.3}(5), we have $H^0(s_{4}, H^1(s_{3}s_{4}w_{r}, \alpha_{2}))=0$ for $r=3,4.$ Thus by using SES we have $H^1(s_{4}s_{3}s_{4}w_{r}, \alpha_{2})=0$ for $r=3,4.$	
\end{proof}

\begin{cor}\label{cor 5.6}
	
\item[(1)] $H^1(s_{4}s_{2}s_{3}s_{4}s_{1}s_{2}, \alpha_{2})= 0.$

\item[(2)]  $H^1(s_{4}s_{2}s_{3}s_{4}w_{1}, \alpha_{2})=\mathbb{C}_{-(\alpha_{1}+\alpha_{2}+\alpha_{3})}\oplus \mathbb{C}_{-(\alpha_{1}+\alpha_{2}+\alpha_{3}+\alpha_{4})} \oplus   \mathbb{C}_{-(\alpha_{2}+2\alpha_{3}+\alpha_{4})}\oplus \mathbb{C}_{-(\alpha_{1}+\alpha_{2}+2\alpha_{3}+\alpha_{4})}\oplus \mathbb{C}_{-(\alpha_{1}+2\alpha_{2}+2\alpha_{3}+\alpha_{4})}.$
	
\item[(3)] $H^1(s_{4}s_{2}s_{3}s_{4}w_{2}, \alpha_{2})=\mathbb{C}_{-\omega_{4} + \alpha_{4}} \oplus \mathbb{C}_{-\omega_{4}}.$	
	
\item [(4)] $H^1(s_{4}s_{2}s_{3}s_{4}w_{r}, \alpha_{2})=0$ for $r= 3,4.$
		
\end{cor}
\begin{proof} 
Proof of (1): By Lemma \ref{lem 3.3} we have 
\begin{center}
$H^1(s_{2},H^0(s_{4}s_{3}s_{4}s_{1}s_{2}, \alpha_{2} ))=0.$		
\end{center} 
On the other hand, by using Corollary \ref{cor 5.5}(1) we have 
\begin{center}
$H^0(s_{2},H^1(s_{4}s_{3}s_{4}s_{1}s_{2}, \alpha_{2} ))=0.$		
\end{center}
Hence we have $H^1(s_{4}s_{2}s_{3}s_{4}s_{1}s_{2}, \alpha_{2})=H^1(s_{2}s_{4}s_{3}s_{4}s_{1}s_{2}, \alpha_{2})=0.$

Proof of (2): By Corollary \ref{cor 5.5}(2) we have 
\begin{center}
$H^0(s_{2}, H^1(s_{4}s_{3}s_{4}w_{1}, \alpha_{2}))$=$ \mathbb{C}_{-(\alpha_{1} + \alpha_{2} + \alpha_{3} )}\oplus \mathbb{C}_{-(\alpha_{1} + \alpha_{2} + \alpha_{3} +\alpha_{4})}\oplus \mathbb{C}_{-(\alpha_{2}+2\alpha_{3}+\alpha_{4})}\oplus\mathbb{C}_{-(\alpha_{1}+\alpha_{2}+2\alpha_{3}+\alpha_{4})}\oplus \mathbb{C}_{-(\alpha_{1}+2\alpha_{2}+2\alpha_{3}+\alpha_{4})}.$
\end{center}
Now the proof of $(2)$ follows from Lemma \ref{lem 3.3} and SES.

Proof of (3): 
By Corollary \ref{cor 5.5}(3), using SES, and Lemma \ref{lemma1.3} we have 
\begin{center}
$H^0(s_{2},H^1(s_{4}s_{3}s_{4}w_{2},\alpha_{2}))=\mathbb{C}_{-\omega_{4}+\alpha_{4}}\oplus \mathbb{C}_{-\omega_{4}}.$		
\end{center}
Now the proof of $(3)$ follows from Lemma \ref{lem 3.3} and SES.

Proof of (4):  By Lemma \ref{lem 3.3} we have $H^1(s_{2},H^0(s_{4}s_{3}s_{4}w_{r},\alpha_{2}))=0$ for $r=3,4.$ On the other hand, by Corollary \ref{cor 5.5}(5) we have $H^0(s_{2}, H^1(s_{4}s_{3}s_{4}w_{r}, \alpha_{2}))=0$ for $r=3,4.$ Thus by using SES we have $H^1(s_{4}s_{2}s_{3}s_{4}w_{r}, \alpha_{2})=0$ for $r=3,4.$	
\end{proof}

\begin{lem}\label{cor 5.7}

\item[(1)] 	$H^1(s_{3}s_{4}s_{2}s_{3}s_{4}s_{1}s_{2}, \alpha_{2})=\mathbb{C}_{-(\alpha_{2}+\alpha_{3})}.$
	
\item[(2)] $H^1(s_{3}s_{4}s_{2}s_{3}s_{4}w_{1}, \alpha_{2})= \mathbb{C}_{-(\alpha_{1} + \alpha_{2} + \alpha_{3} )} \oplus \mathbb{C}_{-(\alpha_{1} + \alpha_{2} + \alpha_{3} + \alpha_{4})} \oplus \mathbb{C}_{-(\alpha_{1} + \alpha_{2} + 2\alpha_{3} + \alpha_{4})} \oplus \mathbb{C}_{-(\alpha_{1} + 2\alpha_{2} + 2\alpha_{3} + \alpha_{4})} \oplus \mathbb{C}_{-\omega_{4} + \alpha_{4}} .$
	
\item[(3)] $H^1(s_{3}s_{4}s_{2}s_{3}s_{4}w_{2}, \alpha_{2})=\mathbb{C}_{-\omega_{4}}.$
	
\item [(4)] $H^1(s_{3}s_{4}s_{2}s_{3}s_{4}w_{r}, \alpha_{2})=0$ for $r= 3,4.$	
\end{lem}
\begin{proof}
Proof of (1): Recall from $(4.1.2)$ that 
\begin{center}
$H^0(s_{2}s_{3}s_{4}s_{1}s_{2}, \alpha_{2})=\mathbb{C}_{-(\alpha_{2}+2\alpha_{3})}\oplus \mathbb{C}_{-(\alpha_{1}+\alpha_{2}+\alpha_{3})}\oplus \mathbb{C}_{-(\alpha_{1}+\alpha_{2}+2\alpha_{3})}\oplus \mathbb{C}_{-(\alpha_{1}+2\alpha_{2}+2\alpha_{3})}.$
\end{center}
Since $\langle -(\alpha_{2}+2\alpha_{3}), \alpha_{4}\rangle =2,$ $\langle -(\alpha_{1}+\alpha_{2}+\alpha_{3}), \alpha_{4}\rangle =1,$ $\langle -(\alpha_{1}+\alpha_{2}+2\alpha_{3}), \alpha_{4}\rangle =2,$ and  $\langle -(\alpha_{1}+2\alpha_{2}+2\alpha_{3}), \alpha_{4}\rangle =2,$ by using SES and Lemma \ref{lemma1.3}(2) we have
\begin{center}
$H^0(s_{4}s_{2}s_{3}s_{4}s_{1}s_{2}, \alpha_{2})$=$\mathbb{C}_{-(\alpha_{2}+2\alpha_{3})}\oplus
\mathbb{C}_{-(\alpha_{2}+2\alpha_{3}+\alpha_{4})}\oplus \mathbb{C}_{-(\alpha_{2}+2\alpha_{3}+2\alpha_{4})}\oplus
\mathbb{C}_{-(\alpha_{1}+\alpha_{2}+\alpha_{3})}\oplus \mathbb{C}_{-(\alpha_{1}+\alpha_{2}+\alpha_{3}+\alpha_{4})}\oplus \mathbb{C}_{-(\alpha_{1}+\alpha_{2}+2\alpha_{3})}\oplus \mathbb{C}_{-(\alpha_{1}+\alpha_{2}+2\alpha_{3}+\alpha_{4})}\oplus \mathbb{C}_{-(\alpha_{1}+\alpha_{2}+2\alpha_{3}+2\alpha_{4})}\oplus \mathbb{C}_{-(\alpha_{1}+2\alpha_{2}+2\alpha_{3})}\oplus \mathbb{C}_{-(\alpha_{1}+2\alpha_{2}+2\alpha_{3}+\alpha_{4})}\oplus\mathbb{C}_{-(\alpha_{1}+2\alpha_{2}+2\alpha_{3}+2\alpha_{4})}.\hspace{10cm}(5.7.1)$
 \end{center}

Since $\mathbb{C}_{-(\alpha_{1}+\alpha_{2}+\alpha_{3})}\oplus \mathbb{C}_{-(\alpha_{1}+\alpha_{2}+2\alpha_{3})}$ is the indecomposable $\tilde{B}_{\alpha_{3}}$-module, by using Lemma \ref{lemma 1.4}(1) we have
$\mathbb{C}_{-(\alpha_{1}+\alpha_{2}+\alpha_{3})}\oplus \mathbb{C}_{-(\alpha_{1}+\alpha_{2}+2\alpha_{3})}=V\otimes \mathbb{C}_{-\omega_{3}}$ (where $V$ is the standard two dimensional irreducible $\tilde{L}_{\alpha_{3}}$-module), and $\langle -(\alpha_{2}+2\alpha_{3}), \alpha_{3}\rangle=-2,$  by using SES and Lemma \ref{lemma1.3}(3) we have
 \begin{center}
 	$H^1(s_{3}, H^0(s_{4}s_{2}s_{3}s_{4}s_{1}s_{2}, \alpha_{2}))=\mathbb{C}_{-(\alpha_{2}+\alpha_{3})}.$
 \end{center}
 By using SES and Corollary \ref{cor 5.6}(1) we have   
  \begin{center}
  	$H^0(s_{3}, H^1(s_{4}s_{2}s_{3}s_{4}s_{1}s_{2}, \alpha_{2}))=0.$
  \end{center}
 Thus we have 
 \begin{center}
 	$H^1(s_{3}s_{4}s_{2}s_{3}s_{4}s_{1}s_{2}, \alpha_{2})=\mathbb{C}_{-(\alpha_{2}+\alpha_{3})}.$
 \end{center}
Proof of (2): Recall from $(5.1.9)$ that 
\begin{center}
$H^0(s_{2}s_{3}s_{4}w_{1},\alpha_{2})=\mathbb{C}_{-(\alpha_{2}+2\alpha_{3}+2\alpha_{4})}\oplus \mathbb{C}_{-(\alpha_{1}+2\alpha_{2}+3\alpha_{3}+\alpha_{4})} \oplus \mathbb{C}_{-(\alpha_{1}+2\alpha_{2}+3\alpha_{3}+2\alpha_{4})} \oplus \mathbb{C}_{-(\alpha_{1}+\alpha_{2}+2\alpha_{3}+2\alpha_{4})}\oplus\mathbb{C}_{-(\alpha_{1}+2\alpha_{2}+2\alpha_{3}+2\alpha_{4})} \oplus \mathbb{C}_{-(\alpha_{1}+2\alpha_{2}+4\alpha_{3}+2\alpha_{4})}\oplus \mathbb{C}_{-(\alpha_{1}+3\alpha_{2}+4\alpha_{3}+2\alpha_{4})}.$	
\end{center}
Since $\mathbb{C}_{-(\alpha_{1}+2\alpha_{2}+3\alpha_{3}+\alpha_{4})} \oplus \mathbb{C}_{-(\alpha_{1}+2\alpha_{2}+3\alpha_{3}+2\alpha_{4})} $ is the standard two dimensional irreducible $\tilde{L}_{\alpha_{4}}$-module, $\langle -(\alpha_{1}+2\alpha_{2}+4\alpha_{3}+2\alpha_{4}), \alpha_{4} \rangle =0,$ $\langle -(\alpha_{1}+3\alpha_{2}+4\alpha_{3}+2\alpha_{4}), \alpha_{4} \rangle =0,$   $\langle -(\alpha_{2}+2\alpha_{3}+2\alpha_{4}), \alpha_{4} \rangle =-2,$ $\langle -(\alpha_{1}+\alpha_{2}+2\alpha_{3}+2\alpha_{4}), \alpha_{4} \rangle =-2,$ and $\langle -(\alpha_{1}+2\alpha_{2}+2\alpha_{3}+2\alpha_{4}), \alpha_{4} \rangle =-2,$ by using Lemma \ref{lemma1.3} we have  
\begin{center}
$H^0(s_{4}s_{2}s_{3}s_{4}w_{1},\alpha_{2})$=$ \mathbb{C}_{-(\alpha_{1}+2\alpha_{2}+3\alpha_{3}+\alpha_{4})} \oplus \mathbb{C}_{-(\alpha_{1}+2\alpha_{2}+3\alpha_{3}+2\alpha_{4})}  \oplus \mathbb{C}_{-(\alpha_{1}+2\alpha_{2}+4\alpha_{3}+2\alpha_{4})}\oplus \mathbb{C}_{-(\alpha_{1}+3\alpha_{2}+4\alpha_{3}+2\alpha_{4})}.\hspace{11cm}(5.7.2)$	
\end{center}
By Corollary \ref{cor 5.6}(2) we have 
\begin{center}
$H^1(s_{4}s_{2}s_{3}s_{4}w_{1},\alpha_{2})$=$ \mathbb{C}_{-(\alpha_{1}+\alpha_{2}+\alpha_{3})}\oplus \oplus \mathbb{C}_{-(\alpha_{1}+\alpha_{2}+\alpha_{3}+\alpha_{4})}\oplus  \mathbb{C}_{-(\alpha_{2}+2\alpha_{3}+\alpha_{4})}\oplus \mathbb{C}_{-(\alpha_{1}+\alpha_{2}+2\alpha_{3}+\alpha_{4})}\oplus \mathbb{C}_{-(\alpha_{1}+2\alpha_{2}+2\alpha_{3}+\alpha_{4})}.$	
\end{center}
Since $\mathbb{C}_{-(\alpha_{1}+2\alpha_{2}+3\alpha_{3}+2\alpha_{4})}  \oplus \mathbb{C}_{-(\alpha_{1}+2\alpha_{2}+4\alpha_{3}+2\alpha_{4})}$ is indeomposable $\tilde{B}_{\alpha_{3}}$-module, by Lemma \ref{lemma 1.4}(1) we have 
\begin{center}
 $ \mathbb{C}_{-(\alpha_{1}+2\alpha_{2}+3\alpha_{3}+2\alpha_{4})}  \oplus \mathbb{C}_{-(\alpha_{1}+2\alpha_{2}+4\alpha_{3}+2\alpha_{4})}=V\otimes \mathbb{C}_{-\omega_{3}}$
 \end{center}
 where $V$ is the standard two dimensional irreducible $\tilde{L}_{\alpha_{3}}$-module.

 Further, since   $\langle -(\alpha_{1}+3\alpha_{2}+4\alpha_{3}+2\alpha_{4}), \alpha_{3}\rangle =0,$ $\langle -(\alpha_{1}+2\alpha_{2}+3\alpha_{3}+\alpha_{4}), \alpha_{3}\rangle =-1,$ by using Lemma \ref{lemma1.3} we have 
\begin{center}
$H^1(s_{3}, H^0(s_{4}s_{2}s_{3}s_{4}w_{1}, \alpha_{2}))=0.$
\end{center}

Since $\mathbb{C}_{-(\alpha_{1}+\alpha_{2}+\alpha_{3}+\alpha_{4})}\oplus \mathbb{C}_{-(\alpha_{1}+\alpha_{2}+2\alpha_{3}+\alpha_{4})}$ is the standard two dimensional irreduible $\tilde{L}_{\alpha_{3}}$-module, $\langle -(\alpha_{1}+\alpha_{2}+\alpha_{3}), \alpha_{3}\rangle =0,$ $\langle -(\alpha_{1}+2\alpha_{2}+2\alpha_{3}+\alpha_{4}), \alpha_{3}\rangle =1,$ $\langle -(\alpha_{2}+2\alpha_{3}+\alpha_{4}), \alpha_{3}\rangle =-1,$ by using Lemma \ref{lemma1.3}(2) we have 
\begin{center}
$H^0(s_{3}, H^1(s_{4}s_{2}s_{3}s_{4}w_{1}, \alpha_{2}))$=$\mathbb{C}_{-(\alpha_{1}+\alpha_{2}+\alpha_{3})}\oplus \oplus \mathbb{C}_{-(\alpha_{1}+\alpha_{2}+\alpha_{3}+\alpha_{4})}\oplus \mathbb{C}_{-(\alpha_{1}+\alpha_{2}+2\alpha_{3}+\alpha_{4})}\oplus \mathbb{C}_{-(\alpha_{1}+2\alpha_{2}+2\alpha_{3}+\alpha_{4})}\oplus \mathbb{C}_{-(\alpha_{1}+2\alpha_{2}+3\alpha_{3}+\alpha_{4})}.$
\end{center}
Thus we have
\begin{center}
$H^1(s_{3}s_{4}s_{2}s_{3}s_{4}w_{1}, \alpha_{2})$=$\mathbb{C}_{-(\alpha_{1}+\alpha_{2}+\alpha_{3})}\oplus \oplus \mathbb{C}_{-(\alpha_{1}+\alpha_{2}+\alpha_{3}+\alpha_{4})}\oplus \mathbb{C}_{-(\alpha_{1}+\alpha_{2}+2\alpha_{3}+\alpha_{4})}\oplus \mathbb{C}_{-(\alpha_{1}+2\alpha_{2}+2\alpha_{3}+\alpha_{4})}\oplus \mathbb{C}_{-(\alpha_{1}+2\alpha_{2}+3\alpha_{3}+\alpha_{4})}.$
\end{center}
  
Proof of (3): Recall from $(5.1.14)$ that
\begin{center}
$H^0(s_{2}s_{3}s_{4}w_{2}, \alpha_{2})=\mathbb{C}_{-(2\alpha_{1}+3\alpha_{2}+4\alpha_{3}+2\alpha_{4})}=\mathbb{C}_{-\omega_{1}}.$
\end{center} 
Since $\alpha_{4}$ is orthogonal to $\omega_{1},$ by using Lemma \ref{lemma1.3}(2) we have 
\begin{center}
$H^0(s_{4}s_{2}s_{3}s_{4}w_{2}, \alpha_{2})=\mathbb{C}_{-(2\alpha_{1}+3\alpha_{2}+4\alpha_{3}+2\alpha_{4})}=\mathbb{C}_{-\omega_{1}}.$
\end{center} 
Since $\alpha_{3}$ is orthogonal to $\omega_{1},$ by using Lemma \ref{lemma1.3}(2) we have 
\begin{center}
$H^1(s_{3},H^0(s_{4}s_{2}s_{3}s_{4}w_{2}, \alpha_{2}))=0.$
\end{center} 
On the other hand, by Corollary \ref{cor 5.6}(3) we have 
\begin{center}
 	$H^1(s_{4}s_{2}s_{3}s_{4}w_{2}, \alpha_{2})=\mathbb{C}_{-\omega_{4}+\alpha_{4}}\oplus \mathbb{C}_{-\omega_{4}}.$
\end{center}   
Since $\langle -\omega_{4}, \alpha_{3}\rangle =0$ and $\langle -\omega_{4}+\alpha_{4}, \alpha_{3}\rangle =-1,$ by using Lemma \ref{lemma1.3} we have 
\begin{center}
$H^0(s_{3}, H^1(s_{4}s_{2}s_{3}s_{4}w_{2}, \alpha_{2}))= \mathbb{C}_{-\omega_{4}}.$
\end{center}
Thus we have 
\begin{center}
$H^1(s_{3}s_{4}s_{2}s_{3}s_{4}w_{2}, \alpha_{2})= \mathbb{C}_{-\omega_{4}}.$
\end{center}

Proof of (4): By Lemma \ref{lem 4.1} we have $H^0(w_{r},\alpha_{2})=0$ for $r=3,4.$ Therefore we have

$H^0(s_{4}s_{2}s_{3}s_{4}w_{r},\alpha_{2})=0$ for $r=3,4.$ Hence we have $H^1(s_{3},H^0(s_{4}s_{2}s_{3}s_{4}w_{r},\alpha_{2}))=0$ for $r=3,4.$ On the other hand, Corollary \ref{cor 5.6}(4) we have $H^0(s_{3}, H^1(s_{4}s_{2}s_{3}s_{4}w_{r}, \alpha_{2}))=0$ for $r=3,4.$ Thus by using SES we have $H^1(s_{3}s_{4}s_{2}s_{3}s_{4}w_{r}, \alpha_{2})=0$ for $r=3,4.$	
\end{proof}	
\begin{lem}\label{cor 5.8}	
\item[(1)] $H^1(s_{4}s_{3}s_{4}s_{2}, \alpha_{2})=0.$
	
\item[(2)] $H^1(s_{4}s_{3}s_{4}s_{2}s_{3}s_{4}s_{1}s_{2}, \alpha_{2})=\mathbb{C}_{-(\alpha_{2} +\alpha_{3})} \oplus \mathbb{C}_{-(\alpha_{2} + \alpha_{3} +\alpha_{4})}\oplus \mathbb{C}_{-(\alpha_{2} + 2\alpha_{3} +\alpha_{4})}.$
		
\item[(3)] $H^1(s_{4}s_{3}s_{4}s_{2}s_{3}s_{4}w_{1}, \alpha_{2})= \mathbb{C}_{-(\alpha_{1} + \alpha_{2} + \alpha_{3} )} \oplus \mathbb{C}_{-(\alpha_{1} + \alpha_{2} + \alpha_{3} + \alpha_{4})}\oplus \mathbb{C}_{-(\alpha_{1} + \alpha_{2} + 2\alpha_{3} + \alpha_{4})} \oplus \mathbb{C}_{-(\alpha_{1} + 2\alpha_{2} + 2\alpha_{3} + \alpha_{4})}\oplus \mathbb{C}_{-\omega_{4} + \alpha_{4}} \oplus \mathbb{C}_{-\omega_{4}}.$

\item [(4)] $H^1(s_{4}s_{3}s_{4}s_{2}s_{3}s_{4}w_{r}, \alpha_{2})=0$ for $r= 2, 3.$
\end{lem}
\begin{proof}
	Proof of (1): By using SES it is easy to see that 
\begin{center}
$H^0(s_{3}s_{4}s_{2}, \alpha_{2})=\mathbb{C}h(\alpha_{2})\oplus \mathbb{C}_{-\alpha_{2}}\oplus \mathbb{C}_{-(\alpha_{2}+\alpha_{3})}\oplus \mathbb{C}_{-(\alpha_{2}+\alpha_{3})}$	
\end{center}
 and
\begin{center} 
	$H^1(s_{3}s_{4}s_{2}, \alpha_{2})=\mathbb{C}_{\alpha_{2}+\alpha_{3}}.$
\end{center}
Since $H^0(s_{3}s_{4}s_{2}, \alpha_{2})_{\mu}\neq0$ implies $\langle \mu, \alpha_{4} \rangle \ge0,$ by using Lemma \ref{lemma1.3}(2) we have 
\begin{center}
	$H^1(s_{4}, H^0(s_{3}s_{4}s_{2}, \alpha_{2}))=0.$
\end{center}
Since $\langle \alpha_{2}+\alpha_{3}, \alpha_{4}\rangle =-1,$ by using Lemma \ref{lemma1.3}(4) we have $H^0(s_{4}, H^1(s_{3}s_{4}s_{2}, \alpha_{2}))=0.$ Therefore by using SES we have $H^1(s_{4}s_{3}s_{4}s_{2}, \alpha_{2})=0.$

Proof of (2): By the Corollary \ref{cor 5.7}(1) we have 
\begin{center}
	$H^1(s_{3}s_{4}s_{2}s_{3}s_{4}s_{1}s_{2}, \alpha_{2})=\mathbb{C}_{-(\alpha_{2}+\alpha_{3})}.$
\end{center}
Since $\langle -(\alpha_{2}+\alpha_{3}), \alpha_{4}\rangle =1,$ by using Lemma \ref{lemma1.3}(2) we have

\begin{center}
	$H^0(s_{4},H^1(s_{3}s_{4}s_{2}s_{3}s_{4}s_{1}s_{2}, \alpha_{2}))=\mathbb{C}_{-(\alpha_{2}+\alpha_{3})}\oplus \mathbb{C}_{-(\alpha_{2}+\alpha_{3}+\alpha_{4})}.$
\end{center}
Recall from $(5.7.1)$ that 
\begin{center}
$H^0(s_{4}s_{2}s_{3}s_{4}s_{1}s_{2}, \alpha_{2})=\mathbb{C}_{-(\alpha_{2}+2\alpha_{3})}\oplus
\mathbb{C}_{-(\alpha_{2}+2\alpha_{3}+\alpha_{4})}\oplus \mathbb{C}_{-(\alpha_{2}+2\alpha_{3}+2\alpha_{4})}\oplus
\mathbb{C}_{-(\alpha_{1}+\alpha_{2}+\alpha_{3})}\oplus \mathbb{C}_{-(\alpha_{1}+\alpha_{2}+\alpha_{3}+\alpha_{4})}\oplus \mathbb{C}_{-(\alpha_{1}+\alpha_{2}+2\alpha_{3})}\oplus \mathbb{C}_{-(\alpha_{1}+\alpha_{2}+2\alpha_{3}+\alpha_{4})}\oplus \mathbb{C}_{-(\alpha_{1}+\alpha_{2}+2\alpha_{3}+2\alpha_{4})}\oplus \mathbb{C}_{-(\alpha_{1}+2\alpha_{2}+2\alpha_{3})}\oplus \mathbb{C}_{-(\alpha_{1}+2\alpha_{2}+2\alpha_{3}+\alpha_{4})}\oplus \mathbb{C}_{-(\alpha_{1}+2\alpha_{2}+2\alpha_{3}+2\alpha_{4})}.$
\end{center}
Since $\mathbb{C}_{-(\alpha_{1}+\alpha_{2}+\alpha_{3}+\alpha_{4})}\oplus  \mathbb{C}_{-(\alpha_{1}+\alpha_{2}+2\alpha_{3}+\alpha_{4})}$ is the standard two dimensional irreducible $\tilde{L}_{\alpha_{3}}$-module, $\langle-(\alpha_{2}+2\alpha_{3}+2\alpha_{4}), \alpha_{3} \rangle =0,$ $\langle-(\alpha_{1}+\alpha_{2}+2\alpha_{3}+2\alpha_{4}), \alpha_{3} \rangle =0,$ $\langle-(\alpha_{1}+2\alpha_{2}+2\alpha_{3}), \alpha_{3} \rangle =0,$ $\langle-(\alpha_{1}+2\alpha_{2}+2\alpha_{3}+\alpha_{4}), \alpha_{3} \rangle =1,$  $\langle-(\alpha_{1}+2\alpha_{2}+2\alpha_{3}+2\alpha_{4}), \alpha_{3} \rangle =2,$ $\langle-(\alpha_{2}+2\alpha_{3}), \alpha_{3} \rangle =0,$ $\langle-(\alpha_{2}+2\alpha_{3}+\alpha_{4}), \alpha_{3} \rangle =-1,$ $\langle-(\alpha_{2}+2\alpha_{3}+2\alpha_{4}), \alpha_{3} \rangle =0,$ and $\mathbb{C}_{-(\alpha_{1}+\alpha_{2}+\alpha_{3})}\oplus \mathbb{C}_{-(\alpha_{1}+\alpha_{2}+2\alpha_{3})}=V\otimes \mathbb{C}_{-\omega_{3}}$ (where $V$ is the standard twi dimensional irreducible $\tilde{L}_{\alpha_{3}}$), by using SES and Lemma \ref{lemma1.3} we have
\begin{center}
$H^0(s_{3}s_{4}s_{2}s_{3}s_{4}s_{1}s_{2}, \alpha_{2})$=$
\mathbb{C}_{-(\alpha_{2}+2\alpha_{3}+2\alpha_{4})}\oplus
\mathbb{C}_{-(\alpha_{1}+\alpha_{2}+\alpha_{3}+\alpha_{4})}\oplus  \mathbb{C}_{-(\alpha_{1}+\alpha_{2}+2\alpha_{3}+\alpha_{4})}\oplus \mathbb{C}_{-(\alpha_{1}+\alpha_{2}+2\alpha_{3}+2\alpha_{4})}\oplus \mathbb{C}_{-(\alpha_{1}+2\alpha_{2}+2\alpha_{3})}\oplus \mathbb{C}_{-(\alpha_{1}+2\alpha_{2}+2\alpha_{3}+\alpha_{4})}\oplus \mathbb{C}_{-(\alpha_{1}+2\alpha_{2}+3\alpha_{3}+\alpha_{4})}\oplus \mathbb{C}_{-(\alpha_{1}+2\alpha_{2}+2\alpha_{3}+2\alpha_{4})}\oplus \mathbb{C}_{-(\alpha_{1}+2\alpha_{2}+3\alpha_{3}+2\alpha_{4})}\oplus \mathbb{C}_{-(\alpha_{1}+2\alpha_{2}+4\alpha_{3}+2\alpha_{4})}.$\hspace{3cm}$(5.8.1)$
\end{center}

The $\tilde{B}_{\alpha_{4}}$-indecomposable summands $V$ of $H^0(s_{3}s_{4}s_{2}s_{3}s_{4}s_{1}s_{2}, \alpha_{2})$ for which $H^1(s_{4}, V)\neq0$ is $\mathbb{C}_{-(\alpha_{2}+2\alpha_{3}+2\alpha_{4})}.$
Thus by using SES and Lemma \ref{lemma1.3}(3) we have 
\begin{center}
$H^1(s_{4}, H^0(s_{3}s_{4}s_{2}s_{3}s_{4}s_{1}s_{2}, \alpha_{2}))=\mathbb{C}_{-(\alpha_{2}+2\alpha_{3}+\alpha_{4})}.$
\end{center}
Therefore by using SES we have
\begin{center}
$H^1(s_{4}s_{3}s_{4}s_{2}s_{3}s_{4}s_{1}s_{2}, \alpha_{2})=\mathbb{C}_{-(\alpha_{2}+\alpha_{3})}\oplus \mathbb{C}_{-(\alpha_{2}+\alpha_{3}+\alpha_{4})}\oplus \mathbb{C}_{-(\alpha_{2}+2\alpha_{3}+\alpha_{4})}.$
\end{center}

Proof of (3): Recall that from Corollary \ref{cor 5.7}(2) we have 
\begin{center}
$H^1(s_{3}s_{4}s_{2}s_{3}s_{4}w_{1}, \alpha_{2})$=$ \mathbb{C}_{-(\alpha_{1} + \alpha_{2} + \alpha_{3} )} \oplus \mathbb{C}_{-(\alpha_{1} + \alpha_{2} + \alpha_{3} + \alpha_{4})} \oplus \mathbb{C}_{-(\alpha_{1} + \alpha_{2} + 2\alpha_{3} + \alpha_{4})} \oplus \mathbb{C}_{-(\alpha_{1} + 2\alpha_{2} + 2\alpha_{3} + \alpha_{4})} \oplus \mathbb{C}_{-(\alpha_{1} + 2\alpha_{2} + 3\alpha_{3} + \alpha_{4})} .$
\end{center}
Since $\mathbb{C}_{-(\alpha_{1} + \alpha_{2} + \alpha_{3} )} \oplus \mathbb{C}_{-(\alpha_{1} + \alpha_{2} + \alpha_{3} + \alpha_{4})}$ is the standard two dimensional irreducible $\tilde{L}_{\alpha_{4}}$-module, $\langle -(\alpha_{1} + \alpha_{2} + 2\alpha_{3} + \alpha_{4}), \alpha_{4}\rangle =0,$ $\langle -(\alpha_{1} + 2\alpha_{2} + 2\alpha_{3} + \alpha_{4}), \alpha_{4}\rangle =0,$ and $\langle -(\alpha_{1} + 2\alpha_{2} + 3\alpha_{3} + \alpha_{4}), \alpha_{4}\rangle =1,$ by using SES and Lemma \ref{lemma1.3} we have 
\begin{center}
	$H^0(s_{4},H^1(s_{3}s_{4}s_{2}s_{3}s_{4}w_{1}, \alpha_{2}))$=$ \mathbb{C}_{-(\alpha_{1} + \alpha_{2} + \alpha_{3} )} \oplus \mathbb{C}_{-(\alpha_{1} + \alpha_{2} + \alpha_{3} + \alpha_{4})} \oplus \mathbb{C}_{-(\alpha_{1} + \alpha_{2} + 2\alpha_{3} + \alpha_{4})} \oplus \mathbb{C}_{-(\alpha_{1} + 2\alpha_{2} + 2\alpha_{3} + \alpha_{4})} \oplus \mathbb{C}_{-(\alpha_{1} + 2\alpha_{2} + 3\alpha_{3} + \alpha_{4})}\oplus \mathbb{C}_{-(\alpha_{1} + 2\alpha_{2} + 3\alpha_{3} +2\alpha_{4})}.$
\end{center}
On the other hand, from $(5.7.2)$ we have
\begin{center}
$H^0(s_{4}s_{2}s_{3}s_{4}w_{1},\alpha_{2})$=$ \mathbb{C}_{-(\alpha_{1}+2\alpha_{2}+3\alpha_{3}+\alpha_{4})} \oplus \mathbb{C}_{-(\alpha_{1}+2\alpha_{2}+3\alpha_{3}+2\alpha_{4})} \oplus \mathbb{C}_{-(\alpha_{1}+2\alpha_{2}+4\alpha_{3}+2\alpha_{4})}\oplus \mathbb{C}_{-(\alpha_{1}+3\alpha_{2}+4\alpha_{3}+2\alpha_{4})}.$	
\end{center}
Since $\mathbb{C}_{-(\alpha_{1}+2\alpha_{2}+3\alpha_{3}+2\alpha_{4})} \oplus\mathbb{C}_{-(\alpha_{1}+2\alpha_{2}+4\alpha_{3}+2\alpha_{4})}=V\otimes \mathbb{C}_{-\omega_{3}},$ where $V$ is the standard two dimensional irreducible $\tilde{L}_{\alpha_{3}}$-module, $\langle -(\alpha_{1}+3\alpha_{2}+4\alpha_{3}+2\alpha_{4}), \alpha_{3}\rangle =0,$ and $\langle -(\alpha_{1}+2\alpha_{2}+3\alpha_{3}+\alpha_{4}), \alpha_{3}\rangle =-1,$  by using SES and Lemma \ref{lemma1.3} we have
\begin{center}
$H^0(s_{3}s_{4}s_{2}s_{3}s_{4}w_{1},\alpha_{2})$=$ \mathbb{C}_{-(\alpha_{1}+3\alpha_{2}+4\alpha_{3}+2\alpha_{4})}.$	
\end{center}
Since $\langle -(\alpha_{1}+3\alpha_{2}+4\alpha_{3}+2\alpha_{4}), \alpha_{4}\rangle =0,$ by using SES and Lemma \ref{lemma1.3} we have
\begin{center}
$H^1(s_{4},H^0(s_{3}s_{4}s_{2}s_{3}s_{4}w_{1},\alpha_{2}))=0.$	
\end{center}
Therefore by SES we have
\begin{center}
$H^1(s_{4}s_{3}s_{4}s_{2}s_{3}s_{4}w_{1}, \alpha_{2})$=$ \mathbb{C}_{-(\alpha_{1} + \alpha_{2} + \alpha_{3} )} \oplus \mathbb{C}_{-(\alpha_{1} + \alpha_{2} + \alpha_{3} + \alpha_{4})} \oplus \mathbb{C}_{-(\alpha_{1} + \alpha_{2} + 2\alpha_{3} + \alpha_{4})} \oplus \mathbb{C}_{-(\alpha_{1} + 2\alpha_{2} + 2\alpha_{3} + \alpha_{4})} \oplus \mathbb{C}_{-(\alpha_{1} + 2\alpha_{2} + 3\alpha_{3} + \alpha_{4})}\oplus \mathbb{C}_{-(\alpha_{1} + 2\alpha_{2} + 3\alpha_{3} +2\alpha_{4})}.$
\end{center}

Proof of (4): For $r=2,$ we recall that from $(5.1.14)$ that $H^0(s_{2}s_{3}s_{4}w_{2}, \alpha_{2})=\mathbb{C}_{-\omega_{1}}.$
Since $\alpha_{4},\alpha_{3}$ are orthogonal to $\omega_{1},$ by using SES we have $H^0(s_{3}s_{4}s_{2}s_{3}s_{4}w_{2}, \alpha_{2})=\mathbb{C}_{-\omega_{1}}.$
Further, using the orthogonality of $\alpha_{4}$ and $\omega_{1}$ we have $H^1(s_{4},H^0(s_{3}s_{4}s_{2}s_{3}s_{4}w_{2}, \alpha_{2}))=0.$
On the other hand, by Corollary \ref{cor 5.7}(3) we have $H^0(s_{4},H^1(s_{3}s_{4}s_{2}s_{3}s_{4}w_{2}, \alpha_{2}))=0.$ Thus we have $H^0(s_{4}s_{3}s_{4}s_{2}s_{3}s_{4}w_{2}, \alpha_{2})=0.$
For $r=3,$ By Lemma \ref{lem 4.1} we have $H^0(w_{3},\alpha_{2})=0.$ Therefore $H^0(s_{3}s_{4}s_{2}s_{3}s_{4}w_{3},\alpha_{2})=0.$ Hence we have $H^1(s_{4},H^0(s_{3}s_{4}s_{2}s_{3}s_{4}w_{3},\alpha_{2}))=0.$ On the other hand,by  Corollary \ref{cor 5.7}(4) we have $H^0(s_{4}, H^1(s_{3}s_{4}s_{2}s_{3}s_{4}w_{3}, \alpha_{2}))=0.$ Thus by using SES we have $H^1(s_{4}s_{3}s_{4}s_{2}s_{3}s_{4}w_{r}, \alpha_{2})=0.$
\end{proof}

We denote $v_{r}=[1, 4]^{r}$ for $1\le r\le 6$ and $\tau_{r}=[1, 4]^{r}1$ for $1\le r\le 5.$
\begin{lem}\label{lem 5.9} We have 
\item[(1)] $H^i(\tau_{r}, \alpha_{1})=0$ for all $i \ge 0,1\le r\le 5.$

\item[(2)] $H^i(v_{r}, \alpha_{4})=0$ for all $i \ge 0,2\le r\le 6.$
\end{lem}
\begin{proof}
Proof of (1): By \cite[Corollary 6.4, p.780]{Ka} we have 
\begin{center}
	$H^i(\tau_{r}, \alpha_{1})=0$ for all $i \ge 2, r\ge 1.$
\end{center}
 Note that $H^i(s_{1}s_{2}s_{3}s_{4}s_{1}, \alpha_{1})=H^i(s_{1}s_{2}s_{1}, \alpha_{1})=H^i(s_{2}s_{1}s_{2}, \alpha_{1})=0$ for $i=0,1$ (see Lemma \ref{lemma1.3}(4)). Now by using SES repeatedly we have the required result.
 
Proof of (2): By \cite[Corollary 6.4, p.780]{Ka} we have 
\begin{center}
 $H^i(v_{r}, \alpha_{4})=0$ for all $i \ge 2, r\ge 1.$
\end{center}
We note that 
\begin{center}
$H^i(s_{4}s_{1}s_{2}s_{3}s_{4}, \alpha_{4})=H^i(s_{1}s_{2}s_{4}s_{3}s_{4}, \alpha_{4})$=$H^i(s_{1}s_{2}s_{3}s_{4}s_{3}, \alpha_{4})=0$ \hspace{3cm}$(5.9.1)$ 
\end{center} 
for $i=0,1$ (see Lemma \ref{lemma1.3}(4)).
 
Since $2\le r\le 6,$ we have $v_{r}=us_{4}s_{1}s_{2}s_{3}s_{4}$ for some $u\in W$ such that $l(v_{r})=l(u)+5.$ Thus by using SES repeatedly we have the required result.
\end{proof}\
\begin{cor}\label{cor 5.10}
We have the following:
\item [(1)] $H^i(s_{4}\tau_{r}, \alpha_{1})=0$ for $i\ge 0, 1\le r\le 5.$

 $H^i(s_{4}v_{r}, \alpha_{4})=0$ for $i\ge0, 1\le r\le 5.$

\item [(2)] $H^i(s_{3}s_{4}\tau_{r}, \alpha_{1})=0$ for $i\ge0, 1\le r\le 5.$

 $H^i(s_{3}s_{4}v_{r}, \alpha_{4})=0$ for $i\ge0, 1\le r\le 5.$

\item [(3)] $H^i(s_{2}s_{3}s_{4}\tau_{r}, \alpha_{1})=0$ for $i\ge0, 1\le r\le 5.$

 $H^i(s_{2}s_{3}s_{4}v_{r}, \alpha_{4})=0$ for $i\ge0, 1\le r\le 5.$

\item [(4)] $H^i(s_{4}s_{3}s_{4}\tau_{r}, \alpha_{1})=0$ for $i\ge0, 1\le r \le 4.$

$H^i(s_{4}s_{3}s_{4}v_{r}, \alpha_{4})=0$ for $i\ge0, 1\le r \le 4.$

\item [(5)] $H^i(s_{4}s_{2}s_{3}s_{4}\tau_{r},\alpha_{1})=0$ for $i\ge0, 1\le r \le 4.$

$H^i(s_{4}s_{2}s_{3}s_{4}v_{r},\alpha_{4})=0$ for $i\ge0, 1\le r \le 4.$

\item [(6)] $H^i(s_{4}s_{3}s_{4}s_{2}s_{3}s_{4}\tau_{r}, \alpha_{1})=0$ for $i\ge0, 1\le r\le 3.$

$H^i(s_{4}s_{3}s_{4}s_{2}s_{3}s_{4}v_{r}, \alpha_{4})=0$ for $i\ge 0, 1\le r\le 3.$

\end{cor}
\begin{proof}
Proofs of (1): By using SES and Lemma \ref{lem 5.9}(1) we have $H^i(s_{4}\tau_{r}, \alpha_{1})=0$ for all $1\le r\le 5,i\ge 0.$ 

By $(5.9.1)$ we have  $H^i(s_{4}v_{1}, \alpha_{4})=0$ for $i\ge 0.$ On the other hand by using SES and Lemma \ref{lem 5.9}(2) we have $H^i(s_{4}v_{r}, \alpha_{4})=0$ for $i\ge 0, 2\le r\le 5.$ Thus combining we have $H^i(s_{4}v_{r}, \alpha_{4})=0$ for $i\ge 0, 1\le r\le 5.$

Proofs of $(2),(3),(4),(5),$ and $(6)$ follow by using SES and $(1).$
\end{proof}
\section{Surjectivity of some maps}
Let $w\in W$ and let $w=s_{i_{1}}s_{i_{2}}\cdots s_{i_{r}}$ be a reduced expression for $w$ and let $\underline{i}=(i_{1},i_{2},\dots ,i_{r})$. Let $\tau=s_{i_{1}}s_{i_{2}}\cdots s_{i_{r-1}}$ and $\underline{i'}=(i_{1},i_{2},\dots, i_{r-1}).$

Recall the following long exact sequence of $B$-modules from \cite{CKP} (see \cite[Proposition 3.1, p.673]{CKP}):

$$0\rightarrow H^0(w,\alpha_{i_{r}})\longrightarrow H^0(Z(w,\underline{i}), T_{(w,\underline{i})})\longrightarrow H^0(Z(\tau,\underline{i'}), T_{(\tau,\underline{i'})}) \longrightarrow$$ $$H^1(w,\alpha_{i_{r}})\longrightarrow H^1(Z(w,\underline{i}), T_{(w,\underline{i})})\longrightarrow H^1(Z(\tau,\underline{i'}), T_{(\tau,\underline{i'})})\longrightarrow H^2(w,\alpha_{i_{r}})\longrightarrow $$
$$H^2(Z(w,\underline{i}), T_{(w,\underline{i})})\longrightarrow H^2(Z(\tau,\underline{i'}), T_{(\tau,\underline{i'})})\longrightarrow H^3(w,\alpha_{i_{r}})\longrightarrow \cdots$$

By \cite[Corollary 6.4, p.780]{Ka}, we have $H^j(w, \alpha_{i_{r}})=0$ for every $j\ge 2.$ Thus we have the following exact sequence of $B$-modules:

$$0\rightarrow H^0(w,\alpha_{i_{r}})\longrightarrow H^0(Z(w,\underline{i}), T_{(w,\underline{i})})\longrightarrow H^0(Z(\tau,\underline{i'}), T_{(\tau,\underline{i'})}) \longrightarrow$$ $$H^1(w,\alpha_{i_{r}})\longrightarrow H^1(Z(w,\underline{i}), T_{(w,\underline{i})})\longrightarrow H^1(Z(\tau,\underline{i'}), T_{(\tau,\underline{i'})})\longrightarrow 0 $$

Now onwards we call this exact sequence by LES.

Let $w_{0}=s_{j_{1}}s_{j_{2}}\cdots s_{j_{N}}$ be a reduced expression of $w_{0}.$ Let $w=s_{j_{1}}s_{j_{2}}\cdots s_{j_{r}},$  $\underline{i}=(j_{1},j_{2},\dots, j_{r}),$ and $\underline{j}=(j_{1},j_{2},\dots,j_{N}).$

\begin{lem}\label{lemma 6.0}
The natural homomorphism 
$$ f:H^0(Z(w_{0}, \underline{j}), T_{(w_{0}, \underline{j})})\longrightarrow H^0(Z(w, \underline{i}), T_{(w, \underline{i})})$$
of $B$-modules is injective if and only if $w^{-1}(\alpha_{0})<0.$
\end{lem}
\begin{proof}
Suppose $w^{-1}(\alpha_{0})<0.$ By \cite[Lemma 6.2, p.667]{CKP}, we have $H^0(Z(w, \underline{i}), T_{(w, \underline{i})})_{-\alpha_{0}}\neq0.$ By \cite[Theorem 7.1]{CKP}, $H^0(Z(w_{0}, \underline{i}), T_{(w_{0}, \underline{i})})$ is a parabolic subalgebra of $\mathfrak{g}$ and hence there is a unique $B$-stable line in $H^0(Z(w_{0},\underline{i}), T_{(w_{0},\underline{i})}),$ namely $\mathfrak{g}_{-\alpha_{0}}.$ Therefore we conclude that the natural homomorphism
\begin{center}
$H^0(Z(w_{0}, \underline{i}), T_{(w_{0}, \underline{i})})\rightarrow H^0(Z(w, \underline{i}), T_{(w, \underline{i})})$
\end{center}
	is injective.

Conversely, suppose the natural homomorphism
\begin{center}
$H^0(Z(w_{0}, \underline{i}), T_{(w_{0}, \underline{i})})\rightarrow H^0(Z(w, \underline{i}), T_{(w, \underline{i})})$
\end{center}
is injective. Then by \cite[Lemma 6.2, p.667]{CKP}, we have $w^{-1}(\alpha_{0})<0.$
\end{proof}

\begin{lem}\label{lemma 6.1}
	The natural homomorphism 
	$$ f:H^1(Z(w_{0}, \underline{j}), T_{(w_{0}, \underline{j})})\longrightarrow H^1(Z(w, \underline{i}), T_{(w, \underline{i})})$$
	of $B$-modules is surjective.
\end{lem}
\begin{proof}
(see \cite[Lemma 7.1, p.459]{CK}).	
\end{proof}
For $1\le r\le 5$ let $\underline{j}_{r}$ be the reduced expression of $\tau_{r}=[1,4]^rs_{1},$    $\underline{i}_{r}=(\underline{j_{r}}, 2)$ be the reduced expression of $w_{r}=[1, 4]^{r}s_{1}s_{2},$ and $\underline{l}_{r}=(\underline{i_{r}}, 3)$ be the reduced expression of $w_{r}s_{3}=[1, 4]^{r}s_{1}s_{2}s_{3}.$
\begin{lem}\label{lem 6.2}
\item[(1)] We have dim$H^0(Z(\tau_{4},\underline{j}_{4}), T_{(\tau_{4}, \underline{j}_{4})})_{-\omega_{4}}=2.$ Further, the natural map
	
$H^0(Z(\tau_{4}, \underline{j}_{4}), T_{(\tau_{4}, \underline{j}_{4})})\longrightarrow H^1(w_{4}, \alpha_{2})$ is surjective.
	
\item[(2)]   
We have dim$H^0(Z(\tau_{3},\underline{j}_{3}), T_{(\tau_{3}, \underline{j}_{3})})_{-\omega_{4}+\alpha_{4}}=2.$ Further, the natural map
	
$H^0(Z(\tau_{3}, \underline{j}_{3}), T_{(\tau_{3}, \underline{j}_{3})})\longrightarrow H^1(w_{3}, \alpha_{2})$ is surjective.

\end{lem}
\begin{proof}
Proof of (1):
Since $w_{4}^{-1}(\alpha_{0})<0,$ by Lemma \ref{lemma 6.0} we conclude that the natural homomorphism
\begin{center}
$H^0(Z(w_{0}, \underline{i}), T_{(w_{0}, \underline{i})})\rightarrow H^0(Z(w_{4}, \underline{i_{4}}), T_{(w_{4}, \underline{i_{4}})})$
\end{center}
is injective. 

Since $\alpha_{3}$ is a short simple root, by \cite[Corollary 5.6, p.778]{Ka} we have $H^1(w_{r}s_{3}, \alpha_{3})=0$ for $r=4,5.$
On the other hand, by Lemma \ref{lem 5.1} we have $H^1(w_{5}, \alpha_{2})=0,$ and by Lemma \ref{lem 5.9} $H^1(v_{r}, \alpha_{4})=0,$ and $H^1(\tau_{r}, \alpha_{1})=0$ for $r=4,5.$ 

Thus from above observations and using LES the natural map
\begin{center}
$H^0(Z(w_{0}, \underline{i}), T_{(w_{0}, \underline{i})})\rightarrow H^0(Z(w_{4}, \underline{i_{4}}), T_{(w_{4}, \underline{i_{4}})})$ \hspace{5.4cm} $(6.3.1)$
\end{center}
is surjective, hence an isomorphism.

By \cite[Theorem 7.1]{CKP} $H^0(Z(w_{0}, \underline{i}), T_{(w_{0}, \underline{i})})$ is parabolic subalgebra of $\mathfrak{g}.$ Hence for any $\mu \in X(T)\setminus \{0\},$ we have 
\begin{center}
	dim$H^0(Z(w_{0}, \underline{i}), T_{(w_{0}, \underline{i})})_{\mu}\le 1.$ 
\end{center}
By using LES repeatedly and using Lemma \ref{lem 5.9} we have 
\begin{center}
$H^0(Z(\tau_{4}, \underline{j_{4}}), T_{(\tau_{4}, \underline{j_{4}})})=H^0(Z(w_{3}s_{3}, \underline{l_{3}}), T_{(l_{3}, \underline{l_{3}})}).$\hspace{5.4cm}$(6.3.2)$
\end{center}
By using LES and \cite[Corollary 5.6, p.778]{Ka} we have an exact sequence 
\begin{center}
$0\rightarrow H^0(w_{3}s_{3}, \alpha_{3})\rightarrow H^0(Z(w_{3}s_{3}, \underline{l_{3}}), T_{(w_{3}s_{3}, \underline{l_{3}})})\rightarrow H^0(Z(w_{3}, \underline{i_{3}}), T_{(w_{3}, \underline{i_{3}})})\rightarrow 0$ \hspace{.1cm} $(6.3.3)$
\end{center}
of $B$-modules.

Since $w_{3}^{-1}(\alpha_{0})<0,$ by using Lemma \ref{lemma 6.0} we conclude that the natural homomorphism
\begin{center}
$H^0(Z(w_{0}, \underline{i}), T_{(w_{0}, \underline{i})})\rightarrow H^0(Z(w_{3}, \underline{i_{3}}), T_{(w_{3}, \underline{i_{3}})})$\hspace{6cm}$(6.3.4)$
\end{center}
is injective.

Thus by $(6.3.4)$ we have dim$H^0(Z(w_{3}, \underline{i_{3}}), T_{(w_{3}, \underline{i_{3}})})_{-\omega_{4}}\ge 1.$ Hence by Lemma \ref{lem 4.2}(2), $(6.3.3)$ we have 
\begin{center}
dim$H^0(Z(w_{3}s_{3}, \underline{l_{3}}), T_{(w_{3}s_{3}, \underline{l_{3}})})_{-\omega_{4}}\ge 2.$ \hspace{7.3cm}$(6.3.5)$
\end{center}
By $(6.3.1)$ we have dim$H^0(Z(w_{4}, \underline{i_{4}}), T_{(w_{4}, \underline{i_{4}})})_{-\omega_{4}}\le 1.$ Therefore by using LES we see that dim$H^0(Z(\tau_{4},\underline{j}_{4}), T_{(\tau_{4}, \underline{j}_{4})})_{-\omega_{4}}\le2.$

Thus by $(6.3.2),(6.3.5)$ we have 
dim$H^0(Z(\tau_{4},\underline{j}_{4}),T_{(\tau_{4},\underline{j}_{4})})_{-\omega_{4}}=2.$ Therefore by LES the natural map $H^1(Z(\tau_{4}, \underline{j}_{4}),T_{(\tau_{4},\underline{j}_{4})})_{-\omega_{4}}\longrightarrow H^1(w_{4}, \alpha_{2})_{-\omega_{4}}$ is surjective. Hence by Lemma \ref{lem 5.1}(3) the natural map $H^1(Z(\tau_{4}, \underline{j}_{4}), T_{(\tau_{4}, \underline{j}_{4})})\longrightarrow H^1(w_{4}, \alpha_{2})$ is surjective.

Proof of (2): 
By using LES repeatedly and using Lemma \ref{lem 5.9} we have 
\begin{center}
	$H^0(Z(\tau_{3}, \underline{j_{3}}), T_{(\tau_{3}, \underline{j_{3}})})=H^0(Z(w_{2}s_{3}, \underline{l_{2}}), T_{(w_{2}s_{3}, \underline{l_{2}})}).\hspace{5.5cm}(6.3.6)$
\end{center}
By using LES and \cite[Corollary 5.6, p.778]{Ka} we have an exact sequence 
\begin{center}
$0\rightarrow H^0(w_{2}s_{3}, \alpha_{3})\rightarrow H^0(Z(w_{2}s_{3}, \underline{l_{2}}), T_{(w_{2}s_{3}, \underline{l_{2}})})\rightarrow H^0(Z(w_{2}, \underline{i_{2}}), T_{(w_{2}, \underline{i_{2}})})\rightarrow 0$ \hspace{.2cm} $(6.3.7)$
\end{center}
of $B$-modules.

Since $w_{2}^{-1}(\alpha_{0})<0,$ by using Lemma \ref{lemma 6.0} we conclude that the natural homomorphism
\begin{center}
$H^0(Z(w_{0}, \underline{i}), T_{(w_{0}, \underline{i})})\rightarrow H^0(Z(w_{2}, \underline{i_{2}}), T_{(w_{2}, \underline{i_{2}})})$ \hspace{6cm} $(6.3.8)$
\end{center}
is injective.

Thus by $(6.3.8)$ we have dim$H^0(Z(w_{2}, \underline{i_{2}}), T_{(w_{2}, \underline{i_{2}})})_{-\omega_{4}+\alpha_{4}}\ge 1.$ Hence by Lemma \ref{lem 4.2}(1), $(6.3.7)$ we have 
\begin{center}
dim$H^0(Z(w_{2}s_{3}, \underline{l_{2}}), T_{(w_{2}s_{3}, \underline{l_{2}})})_{-\omega_{4}+\alpha_{4}}\ge 2.$ \hspace{7cm}(6.3.9)
\end{center}
 By Lemma \ref{lem 5.1}, we have  $H^1(w_{4}, \alpha_{2})_{-\omega_{4}+\alpha_{4}}=0.$ Since $\alpha_{3}$ is a short simple root, by \cite[Corollary 5.6, p.778]{Ka} we have $H^1(w_{3}s_{3}, \alpha_{3})=0.$ On the other hand, by Lemma \ref{lem 5.9} we have $H^1(v_{4}, \alpha_{4})=0$ and $H^1(\tau_{4}, \alpha_{1})=0.$ 
Thus by using LES and from above discussion we have the natural map 

$$H^0(Z(w_{4}, \underline{i_{4}}), T_{(w_{4}, \underline{i_{4}})})_{-\omega_{4}+\alpha_{4}}\longrightarrow H^0(Z(w_{3}, \underline{i_{3}}), T_{(w_{3}, \underline{i_{3}})})_{-\omega_{4}+\alpha_{4}}$$
is surjective.

Thus by using $(6.3.1)$ and above surjectivity we have dim$H^0(Z(w_{3}, \underline{i_{3}}), T_{(w_{3}, \underline{i_{3}})})_{-\omega_{4}+\alpha_{4}}\le 1.$ Therefore by using LES we see that dim$H^0(Z(\tau_{3},\underline{j_{3}}), T_{(\tau_{3}, \underline{j_{3}})})_{\mu}\le2.$ Thus by $(6.3.6),(6.3.8)$ we have 
dim$H^0(Z(\tau_{3},\underline{j}_{3}),T_{(\tau_{3},\underline{j}_{3})})_{\mu}=2.$ Therefore by LES the natural map 

$H^0(Z(\tau_{3}, \underline{j}_{3}),T_{(\tau_{3},\underline{j}_{3})})_{-\omega_{4}+\alpha_{4}}\longrightarrow H^1(w_{3}, \alpha_{2})_{-\omega_{4}+\alpha_{4}}$ is surjective. Hence by Lemma \ref{lem 5.1}(2) the natural map $H^0(Z(\tau_{3}, \underline{j_{3}}), T_{(\tau_{3}, \underline{j_{3}})})\longrightarrow H^1(w_{3}, \alpha_{2})$ is surjective.
\end{proof}

\begin{lem}\label{lem 6.3}
\item[(1)] Let $\mu=-\omega_{4}, -\omega_{4}+\alpha_{4}.$ Then we have 
dim$H^0(Z(s_{4}\tau_{3},(4,\underline{j}_{3})), T_{(s_{4}\tau_{3}, (4,\underline{j}_{3}))})_{\mu}=2.$ Further, the natural map
	
$H^0(Z(s_{4}\tau_{3}, (4,\underline{j}_{3})), T_{(s_{4}\tau_{3}, (4,\underline{j}_{3}))})\longrightarrow H^1(s_{4}w_{3}, \alpha_{2})$ is surjective.
	
\item[(2)] Let $\mu=-(\alpha_{2}+2\alpha_{3}+\alpha_{4}),-(\alpha_{1}+\alpha_{2}+2\alpha_{3}+\alpha_{4}),-(\alpha_{1}+2\alpha_{2}+2\alpha_{3}+\alpha_{4}).$ Then we have dim$H^0(Z(s_{4}\tau_{2},(4,\underline{j}_{2})), T_{(s_{4}\tau_{2}, (4,\underline{j}_{2}))})_{\mu}=2.$ Further, the natural map
	
$H^0(Z(s_{4}\tau_{2}, (4,\underline{j}_{2})), T_{(s_{4}\tau_{2}, (4,\underline{j}_{2}))})\longrightarrow H^1(s_{4}w_{2}, \alpha_{2})$ is surjective.	
\end{lem}
\begin{proof} Since $(s_{4}w_{3})^{-1}(\alpha_{0})<0,$ by Lemma \ref{lemma 6.0} we conclude that the natural homomorphism
\begin{center}
$H^0(Z(w_{0}, (4,\underline{l_{5}})), T_{(w_{0}, (4,\underline{l_{5}})})\rightarrow H^0(Z(s_{4}w_{3}, (4,\underline{i_{3}})), T_{(s_{4}w_{3}, (4,\underline{i_{3}}))})$
\end{center}
is injective. 
	
Since $\alpha_{3}$ is a short simple root, by \cite[Corollary 5.6, p.778]{Ka} we have $H^1(s_{4}w_{r}s_{3}, \alpha_{3})=0$ for $r=3,4,5.$
On the other hand, by Corollary \ref{cor 5.2} we have $H^1(s_{4}w_{r}, \alpha_{2})=0$ for $r=4,5,$ and by Corollary \ref{cor 5.10}(1) we have $H^1(s_{4}v_{r}, \alpha_{4})=0$ and $H^1(s_{4}\tau_{r}, \alpha_{1})=0$ for $r=4,5.$ 

Thus from above observations and using LES the natural map
\begin{center}
$H^0(Z(w_{0}, (4,\underline{l_{5}})), T_{(w_{0}, (4,\underline{l_{5}}))})\rightarrow H^0(Z(s_{4}w_{3}, (4, \underline{i_{3}})), T_{(s_{4}w_{3}, (4,\underline{i_{3}}))})$ \hspace{3cm} $(6.4.1)$
\end{center}
is surjective, hence an isomorphism.
	
Proof of (1): 
By using LES repeatedly and using Corollary \ref{cor 5.10}(1) we have 
\begin{center}
$H^0(Z(s_{4}\tau_{3}, (4,\underline{j_{3}})), T_{(s_{4}\tau_{3}, (4,\underline{j_{3}}))})$=$H^0(Z(s_{4}w_{2}s_{3}, (4,\underline{l_{2}})), T_{(s_{4}w_{2}s_{3}, (4,\underline{l_{2}}))}).$\hspace{2.5cm}$(6.4.2)$
\end{center}
By using LES and \cite[Corollary 5.6, p.778]{Ka} we have an exact sequence 
\begin{center}
$0\rightarrow H^0(s_{4}w_{2}s_{3}, \alpha_{3})\rightarrow H^0(Z(s_{4}w_{2}s_{3}, (4,\underline{l_{2}})), T_{(s_{4}w_{2}s_{3}, (4,\underline{l_{2}}))})\rightarrow H^0(Z(s_{4}w_{2}, (4,\underline{i_{2}})), T_{(s_{4}w_{2}, (4,\underline{i_{2}}))})\rightarrow 0$ \hspace{8cm}$(6.4.3)$
\end{center}
of $B$-modules. 
On the other hand, since $(s_{4}w_{2})^{-1}(\alpha_{0})<0,$ by using Lemma \ref{lemma 6.0}, we conclude that the natural homomorphism
\begin{center}
$H^0(Z(w_{0}, (4,\underline{l_{5}})), T_{(w_{0}, (4,\underline{l_{5}})})\rightarrow H^0(Z(s_{4}w_{2}, (4,\underline{i_{2}})), T_{(s_{4}w_{2}, (4,\underline{i_{2}}))})$ \hspace{3.5cm}$(6.4.4)$
\end{center}
is injective.

Let $\mu=-\omega_{4}, -\omega_{4}+\alpha_{4}.$ Thus by $(6.4.4),$ we have dim$H^0(Z(s_{4}w_{2}, (4,\underline{i_{2}})), T_{(s_{4}w_{2}, (4,\underline{i_{2}}))})_{\mu}\ge 1.$ Hence by $(6.4.2),$ by Corollary \ref{rmk 4.3}(2) we have 
\begin{center}
dim$H^0(Z(s_{4}w_{2}s_{3}, (4,\underline{l_{2}})), T_{(s_{4}w_{2}s_{3}, (4,\underline{l_{2}}))})_{\mu}\ge 2.$ \hspace{6.4cm} $(6.4.5)$
\end{center}

By $(6.4.1),$ dim$H^0(Z(s_{4}w_{3}, (4,\underline{i_{3}})), T_{(s_{4}w_{3}, (4,\underline{i_{3}}))})_{\mu}\le 1.$ 

By using LES we have
dim$H^0(Z(s_{4}\tau_{3},(4,\underline{j_{3}})), T_{(s_{4}\tau_{3}, (4,\underline{j_{3}}))})_{\mu}\le2.$ Thus by $(6.4.2),(6.4.5)$ we have 
dim$H^0(Z(s_{4}\tau_{3},(4,\underline{j}_{3})),T_{(s_{4}\tau_{3},(4,\underline{j}_{3}))})_{\mu}=2.$ Therefore by LES the natural map $H^0(Z(s_{4}\tau_{3},(4,\underline{j}_{3})),T_{(s_{4}\tau_{3},(4,\underline{j}_{3}))})_{\mu}\longrightarrow H^1(s_{4}w_{3}, \alpha_{2})_{\mu}$ is surjective. Hence by Corollary \ref{cor 5.2}(4) the natural map $H^0(Z(s_{4}\tau_{3}, (4,\underline{j_{3}})), T_{(s_{4}\tau_{3}, (4,\underline{j_{3}}))})\longrightarrow H^1(s_{4}w_{3}, \alpha_{2})$ is surjective.

Proof of (2):
By using LES repeatedly and using Corollary \ref{cor 5.10}(1) we have 
\begin{center}
$H^0(Z(s_{4}\tau_{2}, (4,\underline{j_{2}})), T_{(s_{4}\tau_{2}, (4,\underline{j_{2}}))})$=$H^0(Z(s_{4}w_{1}s_{3}, (4,\underline{l_{1}})), T_{(s_{4}w_{1}s_{3}, (4,\underline{l_{1}}))}).\hspace{2.5cm}(6.4.6)$
\end{center}
By using LES and \cite[Corollary 5.6, p.778]{Ka} we have an exact sequence 
\begin{center}
$0\rightarrow H^0(s_{4}w_{1}s_{3}, \alpha_{3})\rightarrow H^0(Z(s_{4}w_{1}s_{3}, (4,\underline{l_{1}})), T_{(s_{4}w_{1}s_{3}, (4,\underline{l_{1}}))})\rightarrow H^0(Z(s_{4}w_{1}, (4,\underline{i_{1}})), T_{(s_{4}w_{1}, (4,\underline{i_{1}}))})\rightarrow 0,$  \hspace{7.8cm} $(6.4.7)$
\end{center}
of $B$-modules.

Let $\mu=-(\alpha_{2}+2\alpha_{3}+\alpha_{4}),-(\alpha_{1}+\alpha_{2}+2\alpha_{3}+\alpha_{4}),-(\alpha_{1}+2\alpha_{2}+2\alpha_{3}+\alpha_{4}).$ Since $H^1(s_{4}w_{2}, \alpha_{2})_{\mu}\neq0,$ by Corollary \ref{cor 5.2}, $(5.1.5)$ the same weight appears in $H^0(s_{4}w_{1}, \alpha_{2}),$ i.e. $H^0(s_{4}w_{1}, \alpha_{2})_{\mu} \neq0.$ This implies $H^0(Z(s_{4}w_{1}, (4,\underline{i_{1}})), T_{(s_{4}w_{1}, (4,\underline{i_{1}}))})_{\mu}\neq0.$

Thus by $(6.4.7),$ Corollary \ref{rmk 4.3}(1) we have

 dim$H^0(Z(s_{4}w_{1}s_{3}, (4,\underline{l_{1}})), T_{(s_{4}w_{1}s_{3}, (4,\underline{l_{1}}))})_{\mu}\ge 2.$ \hspace{6cm} $(6.4.8)$

Since $H^1(s_{4}w_{2}, \alpha_{2})_{\mu}\neq0,$ by Corollary \ref{cor 5.2} we have  $H^1(s_{4}w_{3}, \alpha_{2})_{\mu}=0.$ Since $\alpha_{3}$ is a short simple root, by \cite[Corollary 5.6, p.778]{Ka} we have $H^1(s_{4}w_{2}s_{3}, \alpha_{3})=0.$ On the other hand, by using Corollary \ref{cor 5.10}(1) we have $H^1(s_{4}v_{3}, \alpha_{4})=0$ and $H^1(s_{4}\tau_{3}, \alpha_{1})=0.$ 
Thus by using LES and from above discussion we have the natural map 

$$H^0(Z(s_{4}w_{3}, (4,\underline{i_{3}})), T_{(s_{4}w_{3}, (4,\underline{i_{3}}))})_{\mu}\longrightarrow H^0(Z(s_{4}w_{2}, (4,\underline{i_{2}})), T_{(s_{4}w_{2}, (4,\underline{i_{2}}))})_{\mu}$$
is surjective. 

By $(6.4.1)$ and above surjectivity we have
dim$H^0(Z(s_{4}w_{2}, (4,\underline{i_{2}})), T_{(s_{4}w_{2}, (4,\underline{i_{2}}))})_{\mu}\le 1.$

By using LES we see that
dim$H^0(Z(s_{4}\tau_{2},(4,\underline{j_{2}})), T_{(s_{4}\tau_{2}, (4,\underline{j_{2}}))})_{\mu}\le 2.$ Thus by $(6.4.5),(6.4.8)$ we have 
dim$H^0(Z(s_{4}\tau_{2},(4,\underline{j}_{2})),T_{(s_{4}\tau_{2},(4,\underline{j}_{2}))})_{\mu}=2.$ Therefore by LES the natural map $H^0(Z(s_{4}\tau_{2},(4,\underline{j}_{2})),T_{(s_{4}\tau_{2},(4,\underline{j}_{2}))})_{\mu}\longrightarrow H^1(s_{4}w_{2}, \alpha_{2})_{\mu}$ is surjective. Hence by Corollary \ref{cor 5.2}(3) the natural map $H^0(Z(s_{4}\tau_{2}, (4,\underline{j_{2}})), T_{(s_{4}\tau_{2}, (4,\underline{j_{2}}))})\longrightarrow H^1(s_{4}w_{2}, \alpha_{2})$ is surjective.
\end{proof}

\begin{lem}\label{lem 6.4}
\item[(1)] We have dim$H^0(Z(s_{3}s_{4}\tau_{3},(3,4,\underline{j}_{3})), T_{(s_{3}s_{4}\tau_{3}, (3,4,\underline{j}_{3}))})_{-\omega_{4}}=2.$ Further, the natural map
	
$H^0(Z(s_{3}s_{4}\tau_{3}, (3,4,\underline{j}_{3})), T_{(s_{3}s_{4}\tau_{3}, (3,4,\underline{j}_{3}))})\longrightarrow H^1(s_{3}s_{4}w_{3}, \alpha_{2})$ is surjective.
	
\item[(2)] Let $\mu =-(\alpha_{1}+ 2\alpha_{2}+ 2\alpha_{3}+\alpha_{4}), -(\alpha_{1}+ 2\alpha_{2}+ 3\alpha_{3}+\alpha_{4}).$ Then we have	

dim$H^0(Z(s_{3}s_{4}\tau_{2},(3,4,\underline{j}_{2})), T_{(s_{3}s_{4}\tau_{2}, (3,4,\underline{j}_{2}))})_{\mu}=2.$ Further, the natural map
	
$H^0(Z(s_{3}s_{4}\tau_{2}, (3,4,\underline{j}_{2})), T_{(s_{3}s_{4}\tau_{2}, (3,4,\underline{j}_{2}))})\longrightarrow H^1(s_{3}s_{4}w_{2}, \alpha_{2})$ is surjective.
	
\item[(3)]  
Let $\mu =-(\alpha_{2}+ \alpha_{3}), -(\alpha_{1}+ \alpha_{2}+ \alpha_{3}).$ Then we have

dim$H^0(Z(s_{3}s_{4}\tau_{1},(3,4,\underline{j}_{1})), T_{(s_{3}s_{4}\tau_{1}, (3,4,\underline{j}_{1}))})_{\mu}=2.$ Further, the natural map

$H^0(Z(s_{3}s_{4}\tau_{1}, (3,4,\underline{j}_{1})), T_{(s_{3}s_{4}\tau_{1}, (3,4,\underline{j}_{1}))})\longrightarrow H^1(s_{3}s_{4}w_{1}, \alpha_{2})$ is surjective.

\end{lem}
\begin{proof}
Proofs of Lemma \ref{lem 6.4}(1), Lemma \ref{lem 6.4}(2), Lemma \ref{lem 6.4}(3) are similar to that of Lemma \ref{lem 6.3} with using \cite[Corollary \ref{cor 5.6}, p,778]{Ka}, Corollary \ref{cor 5.3} and Corollary \ref{cor 5.10}(2) appropriately.
\end{proof}
\begin{lem}\label{lem 6.5}
\item[(1)] We have dim$H^0(Z(s_{2}s_{3}s_{4}\tau_{3}, (2,3,4,\underline{j_{3}})), T_{(s_{2}s_{3}s_{4}\tau_{3}, (2,3,4, \underline{j_{3}}))})_{-\omega_{4}}=2.$ Further, the natural map
	
$H^0(Z(s_{2}s_{3}s_{4}\tau_{3}, (2,3,4,\underline{j_{3}})), T_{(s_{2}s_{3}s_{4}\tau_{3}, (2,3,4,\underline{j_{3}}))})\longrightarrow H^1(s_{2}s_{3}s_{4}w_{3}, \alpha_{2})$ is surjective.
	
\item[(2)] We have		 
dim$H^0(Z(s_{2}s_{3}s_{4}\tau_{2}, (2,3,4,\underline{j_{2}})), T_{(s_{2}s_{3}s_{4}\tau_{2}, (2,3,4,\underline{j_{2}}))})_{-\omega_{4}+\alpha_{4}}=2.$
Further, the natural map
	
$H^0(Z(s_{2}s_{3}s_{4}\tau_{2}, (2,3,4,\underline{j_{2}})), T_{(s_{2}s_{3}s_{4}\tau_{2}, (2,3,4,\underline{j_{3}}))})\longrightarrow H^1(s_{2}s_{3}s_{4}w_{2}, \alpha_{2})$ is surjective.
	
\item[(3)] We have 
dim$H^0(Z(s_{2}s_{3}s_{4}\tau_{1}, (2,3,4,\underline{j_{1}})), T_{(s_{2}s_{3}s_{4}\tau_{1}, (2,3,4,\underline{j_{1}}))})_{-(\alpha_{1} + \alpha_{2} + \alpha_{3})}=2.$ Further, the natural map
	
$H^0(Z(s_{2}s_{3}s_{4}\tau_{1}, (2,3,4, \underline{j_{1}})), T_{(s_{2}s_{3}s_{4}\tau_{1}, (2,3,4,\underline{j_{1}}))})\longrightarrow H^1(s_{2}s_{3}s_{4}w_{1}, \alpha_{2})$ is surjective.
\end{lem}

\begin{proof}
Proofs of Lemma \ref{lem 6.5}(1), Lemma \ref{lem 6.5}(2), Lemma \ref{lem 6.5}(3) are similar to that of Lemma \ref{lem 6.3} with using \cite[Corollary \ref{cor 5.6}, p,778]{Ka}, Corollary \ref{cor 5.4} and Corollary \ref{cor 5.10}(3) appropriately.	
\end{proof}

\begin{lem}\label{lem 6.6}
\item[(1)]Let $\mu=-(\alpha_{1}+2\alpha_{2}+2\alpha_{3}+\alpha_{4}),-\omega_{4}+\alpha_{4}, -\omega_{4}.$ Then we have

dim$H^0(Z(s_{4}s_{3}s_{4}\tau_{2},(4,3,4,\underline{j}_{2})), T_{(s_{4}s_{3}s_{4}\tau_{2}, (4,3,4,\underline{j}_{2}))})_{\mu}=2.$ Further, the natural map
	
$H^0(Z(s_{4}s_{3}s_{4}\tau_{2}, (4,3,4,\underline{j}_{2})), T_{(s_{4}s_{3}s_{4}\tau_{2}, (4,3,4,\underline{j}_{2}))})\longrightarrow H^1(s_{4}s_{3}s_{4}w_{2}, \alpha_{2})$ is surjective.
	
\item[(2)]Let $\mu=-(\alpha_{2}+\alpha_{3}), -(\alpha_{2}+\alpha_{3}+\alpha_{4}), -(\alpha_{1}+\alpha_{2}+\alpha_{3}), -(\alpha_{1}+\alpha_{2}+\alpha_{3}+\alpha_{4}),-(\alpha_{2}+2\alpha_{3}+\alpha_{4}), -(\alpha_{1}+\alpha_{2}+2\alpha_{3}+\alpha_{4}).$ Then we have

dim$H^0(Z(s_{4}s_{3}s_{4}\tau_{1}, (4,3,4,\underline{j}_{1})), T_{(s_{4}s_{3}s_{4}\tau_{1}, (4,3,4,\underline{j}_{1}))})_{\mu}=2.$ Further, the natural map
	
$H^0(Z(s_{4}s_{3}s_{4}\tau_{1}, (4,3,4,\underline{j}_{1})), T_{(s_{4}s_{3}s_{4}\tau_{1}, (4,3,4,\underline{j}_{1}))})\longrightarrow H^1(s_{4}s_{3}s_{4}w_{1}, \alpha_{2})$ is surjective.
\end{lem}
\begin{proof}
Proofs of Lemma \ref{lem 6.6}(1), Lemma \ref{lem 6.6}(2), are similar to that of Lemma \ref{lem 6.3} with using \cite[Corollary \ref{cor 5.6}, p,778]{Ka}, Corollary \ref{cor 5.5} and Corollary \ref{cor 5.10}(4) appropriately.
\end{proof}
\begin{lem}\label{lem 6.7}
\item[(1)]Let $\mu=-\omega_{4}+\alpha_{4}, -\omega_{4}.$ Then we have 

dim$H^0(Z(s_{4}s_{2}s_{3}s_{4}\tau_{2},(4,2,3,4,\underline{j}_{2})), T_{(s_{4}s_{2}s_{3}s_{4}\tau_{2}, (4,2,3,4,\underline{j}_{2}))})_{\mu}=2.$ Further, the natural map
	
$H^0(Z(s_{4}s_{2}s_{3}s_{4}\tau_{2}, (4,2,3,4,\underline{j}_{2})), T_{(s_{4}s_{2}s_{3}s_{4}\tau_{2}, (4,2,3,4,\underline{j}_{2}))})\longrightarrow H^1(s_{4}s_{2}s_{3}s_{4}w_{2}, \alpha_{2})$ is surjective.
	
\item[(2)] Let $\mu=-(\alpha_{1}+\alpha_{2}+\alpha_{3}), -(\alpha_{1}+\alpha_{2}+\alpha_{3}+\alpha_{4}), -(\alpha_{2}+2\alpha_{3}+\alpha_{4}),-(\alpha_{1}+\alpha_{2}+2\alpha_{3}+\alpha_{4}), -(\alpha_{1}+2\alpha_{2}+2\alpha_{3}+\alpha_{4}).$ Then we have dim$H^0(Z(s_{4}s_{2}s_{3}s_{4}\tau_{1},\underline{j'}_{1}), T_{(s_{4}s_{2}s_{3}s_{4}\tau_{1}, \underline{j'}_{1})})_{\mu}=2.$ Further, the natural map
	
$H^0(Z(s_{4}s_{2}s_{3}s_{4}\tau_{1}, (4,2,3,4,\underline{j}_{1})), T_{(s_{4}s_{2}s_{3}s_{4}\tau_{1}, (4,2,3,4,\underline{j}_{1}))})\longrightarrow H^1(s_{4}s_{2}s_{3}s_{4}w_{1}, \alpha_{2})$ is surjective.
	
\end{lem}
\begin{proof}
Proofs of Lemma \ref{lem 6.7}(1), Lemma \ref{lem 6.7}(2), are similar to that of Lemma \ref{lem 6.3} with using \cite[Corollary \ref{cor 5.6}, p,778]{Ka}, Corollary \ref{cor 5.6} and Corollary \ref{cor 5.10}(5) appropriately.
\end{proof}\
\begin{lem}\label{lem case 4} Let $\underline{j'_{1}}=(4,3,4,2,3,4,\underline{j}_{1})$ and $\underline{j'}=(4,3,4,2,3,4,1).$
\item[(1)]Let $\Lambda=\{-(\alpha_{1}+\alpha_{2}+\alpha_{3}),-(\alpha_{1}+\alpha_{2}+\alpha_{3}+\alpha_{4}),-(\alpha_{1}+\alpha_{2}+2\alpha_{3}+\alpha_{4}),-(\alpha_{1}+2\alpha_{2}+2\alpha_{3}+\alpha_{4}),-(\alpha_{1}+2\alpha_{2}+3\alpha_{3}+\alpha_{4}),-(\alpha_{1}+2\alpha_{2}+3\alpha_{3}+2\alpha_{4})\}.$
 Then we have

dim$H^0(Z(s_{4}s_{3}s_{4}s_{2}s_{3}s_{4}\tau_{1},\underline{j'_{1}}), T_{(s_{4}s_{3}s_{4}s_{2}s_{3}s_{4}\tau_{1}, \underline{j'_{1}})})_{\mu}=2$ for all $\mu \in \Lambda.$  Further, the natural map

$H^0(Z(s_{4}s_{3}s_{4}s_{2}s_{3}s_{4}\tau_{1}, \underline{j'_{1}}), T_{(s_{4}s_{3}s_{4}s_{2}s_{3}s_{4}\tau_{1}, \underline{j'_{1}})})\longrightarrow H^1(s_{4}s_{3}s_{4}s_{2}s_{3}s_{4}w_{1}, \alpha_{2})$ is surjective.
	
\item[(2)]Let $\Pi=\{-(\alpha_{2}+\alpha_{3}), -(\alpha_{2}+\alpha_{3}+\alpha_{4}), -(\alpha_{2}+2\alpha_{3}+\alpha_{4})\}.$ Then we have

dim$H^0(Z(s_{4}s_{3}s_{4}s_{2}s_{3}s_{4}s_{1},\underline{j'}), T_{(s_{4}s_{3}s_{4}s_{2}s_{3}s_{4}s_{1}, \underline{j'})})_{\mu}=2$ for all $\mu \in \Pi.$  Further, the natural map
	
$H^0(Z(s_{4}s_{3}s_{4}s_{2}s_{3}s_{4}s_{1}, \underline{j'}), T_{(s_{4}s_{3}s_{4}s_{2}s_{3}s_{4}s_{1}, \underline{j'})})\longrightarrow H^1(s_{4}s_{3}s_{4}s_{2}s_{3}s_{4}s_{1}s_{2}, \alpha_{2})$ is surjective.
\end{lem}
\begin{proof}
Let $u_{1}=s_{4}s_{3}s_{4}s_{2}s_{3}s_{4}\tau_{1}$ and $u=s_{4}s_{3}s_{4}s_{2}s_{3}s_{4}s_{1}.$

Note that $w_{0}=s_{4}s_{3}s_{4}s_{2}s_{3}s_{4}w_{3}s_{3}s_{1}s_{2}.$ Let $\underline{i'}$ be this reduced expression of $w_{0}.$

By Lemma \ref{lem 4.1}(2) and Corollary \ref{cor 5.2}(2) we have $H^i(s_{4}w_{4}, \alpha_{2})=0$ for $i\ge0.$ Since $s_{4}$ commutes with $s_{1},s_{2},$ we have $H^i(s_{4}w_{4},\alpha_{2})$=$H^i(s_{4}w_{3}s_{3}s_{4}s_{1}s_{2}, \alpha_{2})$=$H^i(s_{4}w_{3}s_{3}s_{1}s_{2}, \alpha_{2})$ for $i\ge0.$ Thus we have $H^i(s_{4}w_{3}s_{3}s_{1}s_{2}, \alpha_{2})$ for $i\ge0.$ 

Therefore by using SES we have $H^i(s_{4}s_{3}s_{4}s_{2}s_{3}s_{4}w_{3}s_{3}s_{1}s_{2}, \alpha_{2})=0$ for $i\ge 0.$ Since $s_{3}$ commutes with $s_{1}$ we have $H^i(s_{4}s_{3}s_{4}s_{2}s_{3}s_{4}w_{3}s_{3}s_{1}, \alpha_{1})=H^i(s_{4}s_{3}s_{4}s_{2}s_{3}s_{4}w_{3}s_{1}, \alpha_{1})$ for $i\ge 0.$
$H^i(s_{4}s_{3}s_{4}s_{2}s_{3}s_{4}w_{3}s_{1}, \alpha_{1})=H^i(s_{4}s_{3}s_{4}s_{2}s_{3}s_{4}[1, 4]^3s_{2}s_{1}s_{2}, \alpha_{1})=0$ for $i\ge 0$ (see Lemma \ref{lemma1.3}(4)). Thus we have $H^i(s_{4}s_{3}s_{4}s_{2}s_{3}s_{4}w_{3}s_{3}s_{1}, \alpha_{1})=0$ for $i\ge 0.$
By Lemma \ref{cor 5.8}(4) we have $H^1(s_{4}s_{3}s_{4}s_{2}s_{3}s_{4}w_{r}, \alpha_{2})=0$ for $r=2,3.$
Since $\alpha_{3}$ is a short simple root, by \cite[Corollary 5.6, p.778]{Ka} we have $H^1(s_{4}s_{3}s_{4}s_{2}s_{3}s_{4}w_{r}s_{3}, \alpha_{3})=0$ for $r=1,2,3.$ On the other hand, by using Corollary \ref{cor 5.10}(6) we have $H^1(s_{4}s_{3}s_{4}s_{2}s_{3}s_{4}v_{r}, \alpha_{4})=0,$ and $H^1(s_{4}s_{3}s_{4}s_{2}s_{3}s_{4}\tau_{r}, \alpha_{1})=0$ for $r=2,3.$ 

Thus by using LES and the above discussion we have the natural map
\begin{center}
$H^0(Z(w_{0}, \underline{i'}), T_{(w_{0}, \underline{i'} )})\rightarrow H^0(Z(s_{4}s_{3}s_{4}s_{2}s_{3}s_{4}w_{1},(\underline{j'_{1}},2)), T_{(s_{4}s_{3}s_{4}s_{2}s_{3}s_{4}w_{1}, (\underline{j'_{1}},2))}).$ \hspace{1cm}$(6.9.1)$
\end{center}
is surjective. 

Proof of (1): By using LES repeatedly and Corollary \ref{cor 5.10}(6) we have

$H^0(Z(u_{1},\underline{j'_{1}}),T_{(u_{1},\underline{j'_{1}})})=H^0(Z(us_{2}s_{3},(\underline{j'},2,3)),T_{(us_{2}s_{3},(\underline{j'},2,3))}).$ \hspace{3.5cm}$(6.9.2)$

By using LES and \cite[Corollary 5.6, p.778]{Ka} we have an exact sequence 
\begin{center}
$0\rightarrow H^0(us_{2}s_{3}, \alpha_{3})\rightarrow H^0(Z(us_{2}s_{3}), T_{(us_{2}s_{3},(\underline{j'},2,3))})\rightarrow H^0(Z(us_{2}), T_{(us_{2},(\underline{j'},2))})\rightarrow0.$\hspace{.5cm}(6.9.3)	
\end{center}
of $B$-modules.

Let 
$\Lambda_{1}=\{-(\alpha_{1}+2\alpha_{2}+2\alpha_{3}+\alpha_{4}),-(\alpha_{1}+2\alpha_{2}+3\alpha_{3}+\alpha_{4}),-(\alpha_{1}+2\alpha_{2}+3\alpha_{3}+2\alpha_{4})\}.$

Let
$\Lambda_{2}=\{-(\alpha_{1}+\alpha_{2}+\alpha_{3}),-(\alpha_{1}+\alpha_{2}+\alpha_{3}+\alpha_{4}),-(\alpha_{1}+\alpha_{2}+2\alpha_{3}+\alpha_{4})\}.$

By $(5.8.1)$ we have 
\begin{center}
$H^0(s_{3}s_{4}s_{2}s_{3}s_{4}s_{1}s_{2}, \alpha_{2})$=$
\mathbb{C}_{-(\alpha_{2}+2\alpha_{3}+2\alpha_{4})}\oplus
\mathbb{C}_{-(\alpha_{1}+\alpha_{2}+\alpha_{3}+\alpha_{4})}\oplus  \mathbb{C}_{-(\alpha_{1}+\alpha_{2}+2\alpha_{3}+\alpha_{4})}\oplus \mathbb{C}_{-(\alpha_{1}+\alpha_{2}+2\alpha_{3}+2\alpha_{4})}\oplus \mathbb{C}_{-(\alpha_{1}+2\alpha_{2}+2\alpha_{3})}\oplus \mathbb{C}_{-(\alpha_{1}+2\alpha_{2}+2\alpha_{3}+\alpha_{4})}\oplus \mathbb{C}_{-(\alpha_{1}+2\alpha_{2}+3\alpha_{3}+\alpha_{4})}\oplus \mathbb{C}_{-(\alpha_{1}+2\alpha_{2}+2\alpha_{3}+2\alpha_{4})}\oplus \mathbb{C}_{-(\alpha_{1}+2\alpha_{2}+3\alpha_{3}+2\alpha_{4})}\oplus \mathbb{C}_{-(\alpha_{1}+2\alpha_{2}+4\alpha_{3}+2\alpha_{4})}.$
\end{center}

Thus by using SES we see that 
$H^0(us_{2}, \alpha_{2})_{\mu}\neq0$
for all $\mu\in \Lambda_{1}.$ 

By using LES and Lemma \ref{cor 5.8}(2) we have an exact sequence 
\begin{center}
$0\rightarrow H^0(us_{2}, \alpha_{2})_{\mu}\rightarrow H^0(Z(us_{2},(\underline{j'},2)), T_{(us_{2},(\underline{j'},2))})_{\mu}\rightarrow H^0(Z(u,\underline{j'}), T_{(u,(\underline{j'}))})_{\mu}\rightarrow0$ 	
\end{center}
for all $\mu\in \Lambda.$

Note that $H^0(u, \alpha_{1})=H^0(s_{4}s_{3}s_{2}s_{1}, \alpha_{1}).$ Now it is easy to see that $H^0(s_{4}s_{3}s_{2}s_{1}, \alpha_{1})_{\mu}\neq0$ for $\mu\in \Lambda_{2}.$ Therefore we have $H^0(Z(u,\underline{j'}),T_{(u,\underline{j'})})_{\mu}\neq0,$ for all $\mu\in \Lambda_{2}.$ Thus combining above discussion we have 
$H^0(Z(us_{2}, (\underline{j'},2)), T_{(us_{2},(\underline{j'},2))})_{\mu}\neq0$ for all $\mu \in \Lambda.$ \hspace{2.1cm}$(6.9.4)$

Therefore by using $(6.9.3),(6.9.4)$ and Corollary \ref{rem 4.9}(2) we have

dim$H^0(Z(us_{2}s_{3},(\underline{j'},2,3)), T_{(us_{2}s_{3},(\underline{j'},2,3))})_{\mu}\ge 2$ for all $\mu \in \Lambda.$ \hspace{4.4cm}$(6.9.5)$

By $(6.9.1)$ we have $H^0(Z(s_{4}s_{3}s_{4}s_{2}s_{3}s_{4}w_{1},(\underline{j'_{1}},2)), T_{(s_{4}s_{3}s_{4}s_{2}s_{3}s_{4}w_{1}, (\underline{j'_{1}},2))})_{\mu}\le 1$ for all $\mu \in \Lambda.$

Therefore by using LES, Lemma \ref{cor 5.8}(3) we have dim$H^0(Z(u_{1}, \underline{j'_{1}}), T_{(u_{1}, \underline{j'_{1}})})_{\mu}\le2$ for all $\mu \in \Lambda.$

Thus by $(6.9.2),(6.9.5)$ we have dim$H^0(Z(u_{1}, \underline{j'_{1}}), T_{(u_{1}, \underline{j'_{1}})})_{\mu}=2$ for all $\mu \in \Lambda.$ 

By using LES we have 
	
$H^0(Z(u_{1}, \underline{j'_{1}}), T_{(u_{1}, \underline{j'_{1}})})_{\mu}\longrightarrow H^1(s_{4}s_{3}s_{4}s_{2}s_{3}s_{4}w_{1}, \alpha_{2})_{\mu}$ is surjective for all $\mu \in \Lambda$. Hence by Lemma \ref{cor 5.8}(3) the natural map $H^0(Z(u_{1}, \underline{j'_{1}}), T_{(u_{1}, \underline{j'_{1}})})\longrightarrow H^1(s_{4}s_{3}s_{4}s_{2}s_{3}s_{4}w_{1}, \alpha_{2})$ is surjective.

Proof of (2): It is easy to see that $H^1(u, \alpha_{1})=H^1(s_{4}s_{3}s_{2}s_{1}, \alpha_{1})=0$ and  $H^0(u, \alpha_{1})=H^0(s_{4}s_{3}s_{2}s_{1}, \alpha_{1})_{\mu}=0$ for all $\mu \in \Pi.$ 

Further, we have $H^i(s_{4}s_{3}s_{4}s_{2}s_{3}s_{4}, \alpha_{4})=H^i(s_{4}s_{3}s_{2}s_{3}s_{4}s_{3}, \alpha_{3})=0$ for all $i\ge 0$ (see Lemma \ref{lemma1.3}(4)). 

From above disussions and using LES repeatedly  we have 
\begin{center}
$H^0(Z(u,\underline{j'}), T_{(u, \underline{j'})})_{\mu}=H^0(Z(s_{4}s_{3}s_{4}s_{2}s_{3},(4,3,4,2,3)),T_{(s_{4}s_{3}s_{4}s_{2}s_{3}, (4,3,4,2,3))})_{\mu}$ \hspace{1.4cm}$(6.9.6)$
\end{center}

for all $\mu \in \Pi.$

By using LES and \cite[Corollary 5.6, p.778]{Ka} we have an exact sequence 
\begin{center}
$0\rightarrow H^0(s_{4}s_{3}s_{4}s_{2}s_{3}, \alpha_{3})\rightarrow H^0(Z(s_{4}s_{3}s_{4}s_{2}s_{3}), T_{(s_{4}s_{3}s_{4}s_{2}s_{3},(4,3,4,2,3))})\rightarrow H^0(Z(s_{4}s_{3}s_{4}s_{2}), T_{(s_{4}s_{3}s_{4}s_{2},(4,3,4,2))})\rightarrow0.$
\end{center}
\hspace{14.5cm}$(6.9.7)$ 	

It is easy to see that $H^0(s_{4}s_{3}s_{4}s_{2}, \alpha_{2})_{\mu}\neq0$ for all $\mu \in \Pi.$ Therefore we have 

$H^0(Z(s_{4}s_{3}s_{4}s_{2}), T_{(s_{4}s_{3}s_{4}s_{2},(4,3,4,2))})_{\mu}\neq0$ for all $\mu \in \Pi.$ Thus from $(6.9.7)$ and Corollary \ref{rem 4.9}(1) we have

dim$H^0(Z(s_{4}s_{3}s_{4}s_{2}s_{3}), T_{(s_{4}s_{3}s_{4}s_{2}s_{3},(4,3,4,2,3))})_{\mu}\ge 2$ for $\mu \in \Pi.$\hspace{3.8cm}$(6.9.8)$

Since $\alpha_{3}$ is a short simple root, by \cite[Corollary 5.6, p.778]{Ka} we have
$H^1(us_{2}s_{3}, \alpha_{3})=0.$ 

By using Corollary \ref{cor 5.10}(6) we have

$H^1(s_{4}s_{3}s_{4}s_{2}s_{3}s_{4}\tau_{1}, \alpha_{1})=0$ and 
$H^1(s_{4}s_{3}s_{4}s_{2}s_{3}s_{4}v_{1}, \alpha_{4})=0.$

By Lemma \ref{cor 5.8} we have 
$H^1(s_{4}s_{3}s_{4}s_{2}s_{3}s_{4}w_{1}, \alpha_{2})_{\mu}=0$ for all $\mu \in \Pi.$

Thus combining above discussion we have the natural map 
\begin{center}
$H^0(Z(s_{4}s_{3}s_{4}s_{2}s_{3}s_{4}w_{1}, (\underline{j'_{1}},2)), T_{(s_{4}s_{3}s_{4}s_{2}s_{3}s_{4}w_{1}, (\underline{j'_{1}},2))})_{\mu}\rightarrow H^0(Z(us_{2}, (\underline{j'},2)), T_{(us_{2}, (\underline{j'},2))})_{\mu},$ 
\end{center}
is surjective for all $\mu \in \Pi.$  

Now, using $(6.9.1)$ and above surjectivity we have $H^0(Z(us_{2}, (\underline{j'},2)), T_{(us_{2}, (\underline{j'},2))})_{\mu}\le 1$ for all $\mu \in \Pi.$ Further, by Lemma \ref{cor 5.8}(2) dim$H^1(us_{2}, \alpha_{2})_{\mu}=1$ for all $\mu \in \Pi.$

Therefore by using LES
$$0\longrightarrow H^0(us_{2}, \alpha_{2})\longrightarrow H^0(Z(us_{2},(\underline{j'},2)), T_{(us_{2},(\underline{j'},2))})\longrightarrow H^0(Z(u,\underline{j'}), T_{(u,\underline{j'})})\longrightarrow$$
$$H^1(us_{2}, \alpha_{2})\longrightarrow H^1(Z(us_{2},(\underline{j'},2)), T_{(us_{2},(\underline{j'},2))})\longrightarrow H^1(Z(u,\underline{j'}), T_{(u,\underline{j'})})\longrightarrow0$$

 we have $H^0(Z(u, \underline{j'}),T_{(u,\underline{j'})})_{\mu}\le 2$ for all $\mu \in \Pi.$ 

Therefore by $(6.9.6),(6.9.8)$ we have dim$H^0(Z(s_{4}s_{3}s_{4}s_{2}s_{3}), T_{(s_{4}s_{3}s_{4}s_{2}s_{3},(4,3,4,2,3))})_{\mu}=2$ for all $\mu \in \Pi.$

Therefore $H^0(Z(u, \underline{j'}), T_{(u, \underline{j'})})_{\mu}\rightarrow H^1(us_{2}, \alpha_{2})_{\mu}$ is surjective for all $\mu \in \Pi.$

Hence by Lemma \ref{cor 5.8}(2) the natural map
$H^0(Z(u, \underline{j'}), T_{(u,\underline{j'})})\rightarrow H^1(us_{2}, \alpha_{2})$ is surjective.

\end{proof}
\section{main theorem}
In this section we prove the main theorem. Let $c$ be a Coxeter element of $W.$ Then there exists a decreasing  sequence $4\ge a_{1}>a_{2}>\cdots >a_{k}=1$ of positive integers such that $c=[a_{1}, 4][a_{2}, a_{1}-1]\cdots [a_{k}, a_{k-1}-1],$ where for $i\le j$ denotes $[i, j]=s_{i}s_{i+1}\cdots s_{j}.$ 

\begin{thm}\label{theorem 8.1}
$H^j(Z(w_{0},\underline{i}), T_{(w_{0}, \underline{i})})=0$ for all $j\ge 1$ if and only if $a_{1}\neq3$ or $a_{2}\neq 2.$	
\end{thm}
\begin{proof}
From \cite[Proposition 3.1, p. 673]{CKP}, we have $H^j(Z(w_{0}, \underline{i}), T_{(w_{0}, \underline{i})})=0$ for all $j\ge 2.$ It is enough to prove the following:
$H^1(Z(w_{0}, \underline{i}), T_{(w_{0}, \underline{i})})=0$ if and only if $c$ is of the form $[a_{1}, 4][a_{2}, a_{1}-1]\cdots[a_{k}, a_{k-1}-1]$ with $a_{1}\neq3$ or $a_{2}\neq 2.$
	
Proof of $(\implies)$: If  $a_{1}=3,$ and $a_{2}=2,$ then $c=s_{3}s_{4}s_{2}s_{1},$ Let $u=s_{3}s_{4}s_{2}$. Then $c=us_{1}.$
Let $\underline{j}=(3,4,2)$ be the sequence corresponding to $u.$ Then
using LES, we have: 
	
$$0\rightarrow H^0(u,\alpha_{2})  \rightarrow H^0(Z(u,\underline{j}), T_{(u,\underline{j})})\rightarrow H^0(Z(s_{3}s_{4}, (3,4)), T_{(s_{3}s_{4},(3,4))})\rightarrow$$
	
$$ H^1(u,\alpha_{2})\xrightarrow f H^1(Z(u,\underline{j}), T_{(u,\underline{j})})\rightarrow H^1(Z(s_{3}s_{4}, (3,4)), T_{(s_{3}s_{4},(3,4))})\rightarrow0.$$
	
We see that $H^1(u,\alpha_{2})=\mathbb{C}_{\alpha_{2} + \alpha_{3}},$  $H^0(s_{3}, \alpha_{3})_{\alpha_{2} + \alpha_{3}}=0,$ and $H^0(s_{3}s_{4}, \alpha_{4})_{\alpha_{2} + \alpha_{3}}=0.$

Therefore by LES we have $H^0(Z(s_{3}s_{4}, (3,4)), T_{(s_{3}s_{4},(3,4))})_{\alpha_{2} + \alpha_{3}}=0.$ Hence $f$ is non zero homomorphism.
Hence $H^1(Z(u, \underline{j}), T_{(u,\underline{j})}))\neq 0.$ 
By Lemma \ref {lemma 6.1}, the natural homomorphism 
	
$$H^1(Z(w_{0}, \underline{i}), T_{(w_{0}, \underline{i})})\longrightarrow H^1(Z(u,\underline{j}), T_{(u, \underline{j})})$$ is surjective.

Hence we have $H^1(Z(w_{0}, \underline{i}), T_{(w_{0}, \underline{i})})\neq 0.$
	
Proof of $(\impliedby):$ Assume that $a_{1}\neq3$ or $a_{2}\neq2.$ We prove the result by studying case by case.
Note that by using Lemma \ref{lemma1.3}(4) we have  $H^1(w_{0}, \alpha_{i})=0$ for $i=1,2,3,4.$ In each of the following cases we use these appropriately.

{\bf Case 1:} $c=s_{1}s_{2}s_{3}s_{4}.$ Then in this case we have $w_{0}=v_{6}=[1, 4]^6.$ 
By using LES and \cite[Corollary 5.6, p.778]{Ka} we have 
\begin{center}
$H^1(Z(w_{0}, \underline{i}), T_{(w_{0}, \underline{i})})=H^1(Z(w_{5}, \underline{i_{5}}), T_{(w_{5}, \underline{i_{5}})}).$
\end{center}

By using LES, Lemma \ref{lem 5.1}, Lemma \ref{lem 5.9}, and \cite[Corollary 5.6, p.778]{Ka} we have 
\begin{center}
$H^1(Z(w_{5}, \underline{i_{5}}), T_{(w_{5}, \underline{i_{5}})})=H^1(Z(w_{4}, \underline{i_{4}}), T_{(w_{4}, \underline{i_{4}})}).$
\end{center}

By using LES and Lemma \ref{lem 6.2}(1) we have	
\begin{center}
$H^1(Z(w_{4}, \underline{i_{4}}), T_{(w_{4}, \underline{i_{4}})})=H^1(Z(\tau_{4}, \underline{j_{4}}), T_{(\tau_{4}, \underline{j_{4}})}).$ 
\end{center}
By using LES, Lemma \ref{lem 5.9}, and \cite[Corollary 5.6, p.778]{Ka} we have 
\begin{center}
$H^1(Z(\tau_{4}, \underline{j_{4}}), T_{(\tau_{4}, \underline{j_{4}})})=H^1(Z(w_{3}, \underline{i_{3}}), T_{(w_{3}, \underline{i_{3}})}).$
\end{center}
By using LES and Lemma \ref{lem 6.2}(2) we have	
\begin{center}
$H^1(Z(w_{3}, \underline{i_{3}}), T_{(w_{3}, \underline{i_{3}})})=H^1(Z(\tau_{3}, \underline{j_{3}}), T_{(\tau_{3}, \underline{j_{3}})}).$ 
\end{center}
By using LES, Lemma \ref{lem 5.9}, and \cite[Corollary 5.6, p.778]{Ka} we have 
\begin{center}
$H^1(Z(\tau_{3}, \underline{j_{3}}), T_{(\tau_{3}, \underline{j_{3}})})=H^1(Z(w_{2}, \underline{i_{2}}), T_{(w_{2}, \underline{i_{2}})}).$
\end{center}
By using LES, Lemma \ref{lem 5.1}, Lemma \ref{lem 5.9}, and \cite[Corollary 5.6, p.778]{Ka} we have

\begin{center} 
$H^1(Z(w_{2}, \underline{i_{2}}), T_{(w_{2}, \underline{i_{2}})})= H^1(Z(w_{1}, \underline{i_{1}}), T_{(w_{1}, \underline{i_{1}})}).$
\end{center}

By using LES, Lemma \ref{lem 5.1}, Lemma \ref{lem 5.9}, and \cite[Corollary 5.6, p.778]{Ka} we have
\begin{center}
$H^1(Z(w_{1},\underline{i_{1}}),T_{(w_{1},\underline{i_{1}})})= H^1(Z(s_{1}s_{2}, (1,2)), T_{(s_{1}s_{2}, (1,2))}).$
\end{center}

We see that $H^1(s_{1}, \alpha_{1})=0,$ $H^1(s_{1}s_{2}, \alpha_{2})=0.$ Thus by using LES we have

$H^1(Z(s_{1}s_{2}, (1,2)), T_{(s_{1}s_{2}, (1,2))})=0.$ Thus combining all we have $H^1(Z(w_{0},\underline{i}),T_{(w_{0},\underline{i})})=0.$

{\bf Case 2:} $c=s_{4}s_{1}s_{2}s_{3}.$ Then in this case we have $w_{0}=s_{4}w_{5}s_{3}.$ 

By using LES and \cite[Corollary 5.6, p.778]{Ka} we have  
\begin{center}
$H^1(Z(w_{0}, (4,\underline{l_{5}})), T_{(w_{0}, (4,\underline{l_{5}}))})=H^1(Z(s_{4}w_{5}, (4,\underline{i_{5}})), T_{(s_{4}w_{5}, (4,\underline{i_{5}})})).$
\end{center}
By using LES, Corollary \ref{cor 5.2}, Corollary \ref{cor 5.10}(1), and \cite[Corollary 5.6, p.778]{Ka} we have 
\begin{center}
$H^1(Z(s_{4}w_{5}, (4,\underline{i_{5}})), T_{(s_{4}w_{5}, (4,\underline{i_{5}}))})=H^1(Z(s_{4}w_{4}, (4, \underline{i_{4}})), T_{(s_{4}w_{4}, (4,\underline{i_{4}}))}).$
\end{center}
By using LES, Corollary \ref{cor 5.2}, Corollary \ref{cor 5.10}(1), and \cite[Corollary 5.6, p.778]{Ka} we have 
\begin{center}
$H^1(Z(s_{4}w_{4}, (4,\underline{i_{4}})), T_{(s_{4}w_{4}, (4,\underline{i_{4}}))})=H^1(Z(s_{4}w_{3}, (4, \underline{i_{3}})), T_{(s_{4}w_{3}, (4,\underline{i_{3}}))}).$
\end{center}
By using LES and Lemma \ref{lem 6.3}(1) we have 
\begin{center}
$H^1(Z(s_{4}w_{3}, (4,\underline{i_{3}})), T_{(s_{4}w_{3},(4,\underline{i_{3}}))})=H^1(Z(s_{4}\tau_{3}, (4,\underline{j_{3}})),T_{(s_{4}\tau_{3},(4,\underline{j_{3}}))}).$
\end{center} 
By using LES, Corollary \ref{cor 5.10}(1), and \cite[Corollary 5.6, p.778]{Ka} we have
\begin{center}
$H^1(Z(s_{4}\tau_{3}, (4,\underline{j_{3}})), T_{(s_{4}\tau_{3},(4,\underline{j_{3}}))})=H^1(Z(s_{4}w_{2}, (4,\underline{i_{2}})), T_{(s_{4}w_{2}, (4,\underline{i_{2}}))}).$
\end{center} 
By using LES, Lemma \ref{lem 6.3}(2) we have
\begin{center}
$H^1(Z(s_{4}w_{2}, (4,\underline{i_{2}})), T_{(s_{4}w_{2},(4,\underline{i_{2}}))})=H^1(Z(s_{4}\tau_{2}, (4,\underline{j_{2}})),T_{(s_{4}\tau_{2},(4,\underline{j_{2}}))}).$
\end{center}
By using LES, Corollary \ref{cor 5.10}(1), and \cite[Corollary 5.6, p.778]{Ka} we have
\begin{center}
$H^1(Z(s_{4}\tau_{2},(4,\underline{j_{2}})),T_{(s_{4}\tau_{2},(4,\underline{j_{2}}))})=H^1(Z(s_{4}w_{1}, (4,\underline{i_{1}})),T_{(s_{4}w_{1},(4,\underline{i_{1}}))}).$
\end{center}
By using LES, Corollary \ref{cor 5.2}, Corollary \ref{cor 5.10}(1), and \cite[Corollary 5.6, p.778]{Ka} we have 
\begin{center}
$H^1(Z(s_{4}w_{1}, (4,\underline{i_{1}})), T_{(s_{4}w_{1},(4,\underline{i_{1}}))})=H^1(Z(s_{4}s_{1}s_{2}, (4,1,2)), T_{(s_{4}s_{1}s_{2}, (4,1,2))}).$
\end{center}
We see that $H^1(s_{4}s_{1}, \alpha_{1})=0,$ $H^1(s_{4}s_{1}s_{2}, \alpha_{2})=0.$ Thus by using LES we have 
\begin{center}
$H^1(Z(s_{4}s_{1}s_{2}, (4,1,2)), T_{(s_{4}s_{1}s_{2}, (4,1,2))})=0.$
\end{center}
Thus combining all we have $H^1(Z(w_{0},(4,\underline{l_{5}})),T_{(w_{0},(4,\underline{l_{5}}))})=0.$

{\bf Case 3:} $c=s_{3}s_{4}s_{1}s_{2}.$ Then we have  $w_{0}= s_{3}s_{4}w_{5}.$

By using LES, Corollary \ref{cor 5.3}, Corollary \ref{cor 5.10}(2), and \cite[Corollary 5.6, p.778]{Ka} we have 
\begin{center}
$H^1(Z(s_{3}s_{4}w_{5}, (3,4,\underline{i_{5}})), T_{(s_{3}s_{4}w_{5}, (3,4,\underline{i_{5}}))})=H^1(Z(s_{3}s_{4}w_{4}, (3,4, \underline{i_{4}})), T_{(s_{3}s_{4}w_{4}, (3,4,\underline{i_{4}}))}).$
\end{center} 
By using LES, Corollary \ref{cor 5.3}, Corollary \ref{cor 5.10}(2), and \cite[Corollary 5.6, p.778]{Ka} we have 
\begin{center}
$H^1(Z(s_{3}s_{4}w_{4}, (3,4,\underline{i_{4}})), T_{(s_{3}s_{4}w_{4}, (3,4,\underline{i_{4}}))})=H^1(Z(s_{3}s_{4}w_{3}, (3,4, \underline{i_{3}})), T_{(s_{3}s_{4}w_{3}, (3,4,\underline{i_{3}}))}).$
\end{center} 
By using LES and Lemma \ref{lem 6.4}(1) we have 
\begin{center}
$H^1(Z(s_{3}s_{4}w_{3}, (3,4,\underline{i_{3}})), T_{(s_{3}s_{4}w_{3},(3,4,\underline{i_{3}}))})=H^1(Z(s_{3}s_{4}\tau_{3}, (3,4,\underline{j_{3}})),T_{(s_{3}s_{4}\tau_{3},(3,4,\underline{j_{3}}))}).$
\end{center} 
By using LES, Corollary \ref{cor 5.10}(2), and \cite[Corollary 5.6, p.778]{Ka} we have 
\begin{center}
$H^1(Z(s_{3}s_{4}\tau_{3}, (3,4,\underline{j_{3}})), T_{(s_{3}s_{4}\tau_{3},(3,4,\underline{j_{3}}))})=H^1(Z(s_{3}s_{4}w_{2}, (3,4,\underline{i_{2}})), T_{(s_{3}s_{4}w_{2}, (3,4,\underline{i_{2}}))}).$
\end{center} 
By using LES, Lemma \ref{lem 6.4}(2) we have 
\begin{center}
$H^1(Z(s_{3}s_{4}w_{2}, (3,4,\underline{i_{2}})), T_{(s_{3}s_{4}w_{2},(3,4,\underline{i_{2}}))})=H^1(Z(s_{3}s_{4}\tau_{2}, (3,4,\underline{j_{2}})),T_{(s_{3}s_{4}\tau_{2},(3,4,\underline{j_{2}}))}).$
\end{center}
By using LES, Corollary \ref{cor 5.10}(2) and \cite[Corollary 5.6, p.778]{Ka} we have 
\begin{center}
$H^1(Z(s_{3}s_{4}\tau_{2}, (3,4,\underline{j_{2}})), T_{(s_{3}s_{4}\tau_{2},(3,4,\underline{j_{2}}))})=H^1(Z(s_{3}s_{4}w_{1}, (3,4,\underline{i_{1}})), T_{(s_{3}s_{4}w_{1}, (3,4,\underline{i_{1}}))}).$
\end{center} 
By using LES, Lemma \ref{lem 6.4}(3) we have
\begin{center}
$H^1(Z(s_{3}s_{4}w_{1}, (3,4,\underline{i_{1}})), T_{(s_{3}s_{4}w_{1},(3,4,\underline{i_{1}}))})=H^1(Z(s_{3}s_{4}\tau_{1}, (3,4,\underline{j_{1}})),T_{(s_{3}s_{4}\tau_{1},(3,4,\underline{j_{1}}))}).$ 
\end{center}
By using LES, Corollary \ref{cor 5.10}(2), and \cite[Corollary 5.6, p.778]{Ka} we have
\begin{center}
$H^1(Z(s_{3}s_{4}\tau_{1}, (3,4,\underline{j_{1}})), T_{(s_{3}s_{4}\tau_{1},(3,4,\underline{j_{1}}))})=H^1(Z(s_{3}s_{4}s_{1}s_{2}, (3,4,1,2)), T_{(s_{3}s_{4}s_{1}s_{2}, (3,4,1,2))}).$
\end{center}
We see that $H^1(s_{3}s_{4}, \alpha_{4})=0$ (see  \cite[Corollary 5.6, p.778]{Ka}), $H^1(s_{3}s_{4}s_{1}, \alpha_{1})=0,$ $H^1(s_{3}s_{4}s_{1}s_{2}, \alpha_{2})=0.$ Thus by using LES we have 
$$H^1(Z(s_{3}s_{4}s_{1}s_{2}, (3,4,1,2)), T_{(s_{3}s_{4}s_{1}s_{2}, (3,4,1,2))})=0.$$\\ Thus combining all we have $H^1(Z(w_{0},(3,4,\underline{i_{5}})),T_{(w_{0},(3,4,\underline{i_{5}}))})=0.$

{\bf Case 4:} $c=s_{2}s_{3}s_{4}s_{1}.$ Then $w_{0}=s_{2}s_{3}s_{4}\tau_{5}.$ Let $t_{1}=s_{2}s_{3}s_{4}.$

By using LES, Corollary \ref{cor 5.10}(3) and \cite[Corollary 5.6, p.778]{Ka} we have 
\begin{center}
$H^1(Z(w_{0}, (2,3,4,\underline{j_{5}})), T_{(w_{0}, (2,3,4,\underline{j_{5}}))})=H^1(Z(t_{1}w_{4}, (2,3,4, \underline{i_{4}})), T_{(t_{1}w_{4}, (2,3,4,\underline{i_{4}}))}).$
\end{center}
By using LES, Corollary \ref{cor 5.4}, Corollary \ref{cor 5.10}(3), and \cite[Corollary 5.6, p.778]{Ka} we have 
\begin{center}
$H^1(Z(t_{1}w_{4}, (2,3,4,\underline{i_{4}})), T_{(t_{1}w_{4}, (2,3,4,\underline{i_{4}}))})=H^1(Z(t_{1}w_{3}, (2,3,4, \underline{i_{3}})), T_{(t_{1}w_{3}, (2,3,4,\underline{i_{3}}))}).$
\end{center}
By using LES and Lemma \ref{lem 6.5}(1) we have 
\begin{center}
$H^1(Z(t_{1}w_{3}, (2,3,4,\underline{i_{3}})), T_{(t_{1}w_{3},(2,3,4,\underline{i_{3}}))})=H^1(Z(t_{1}\tau_{3}, (2,3,4,\underline{j_{3}})),T_{(t_{1}\tau_{3},(2,3,4,\underline{j_{3}}))}).$
\end{center}
By using LES, Corollary \ref{cor 5.10}(3), and \cite[Corollary 5.6, p.778]{Ka} we have 
\begin{center}
$H^1(Z(t_{1}\tau_{3}, (2,3,4,\underline{j_{3}})), T_{(t_{1}\tau_{3},(2,3,4,\underline{j_{3}}))})=H^1(Z(t_{1}w_{2}, (2,3,4,\underline{i_{2}})), T_{(t_{1}w_{2}, (2,3,4,\underline{i_{2}}))}).$
\end{center} 
By using LES, Lemma \ref{lem 6.5}(2) we have 
\begin{center}
$H^1(Z(t_{1}w_{2}, (2,3,4,\underline{i_{2}})), T_{(t_{1}w_{2},(2,3,4,\underline{i_{2}}))})=H^1(Z(t_{1}\tau_{2}, (2,3,4,\underline{j_{2}})),T_{(t_{1}\tau_{2},(2,3,4,\underline{j_{2}}))}).$
\end{center} 
By using LES, Corollary \ref{cor 5.10}(3), and \cite[Corollary 5.6, p.778]{Ka} we have 
\begin{center}
$H^1(Z(t_{1}\tau_{2}, (2,3,4,\underline{j_{2}})), T_{(t_{1}\tau_{2},(2,3,4,\underline{j_{2}}))})=H^1(Z(t_{1}w_{1}, (2,3,4,\underline{i_{1}})), T_{(t_{1}w_{1}, (2,3,4,\underline{i_{1}}))}).$
\end{center} 
By using LES, Lemma \ref{lem 6.5}(3) we have
\begin{center}
$H^1(Z(t_{1}w_{1}, (2,3,4,\underline{i_{1}})), T_{(t_{1}w_{1},(2,3,4,\underline{i_{1}}))})=H^1(Z(t_{1}\tau_{1},(2,3,4,\underline{j_{1}})),T_{(t_{1}\tau_{1},(2,3,4,\underline{j_{1}}))}).$
\end{center}
By using LES, Corollary \ref{cor 5.10}(3), and \cite[Corollary 5.6, p.778]{Ka} we have
\begin{center}
$H^1(Z(t_{1}\tau_{1}, (2,3,4,\underline{j_{1}})), T_{(t_{1}\tau_{1},(2,3,4,\underline{j_{1}}))})=H^1(Z(t_{1}s_{1}s_{2}, (2,3,4,1,2)), T_{(t_{1}s_{1}s_{2}, (2,3,4,1,2))}).$ 
\end{center}

It is easy to see that $H^1(t_{1}s_{1}, \alpha_{1})=H^1(s_{2}s_{1}, \alpha_{1})=0.$
We see that $H^1(s_{2}s_{3}, \alpha_{3})=0, H^1(t_{1}, \alpha_{4})=0$ by  \cite[Corollary 5.6, p.778]{Ka}. $H^1(t_{1}s_{1}s_{2}, \alpha_{2})=0$ by Corollay \ref{cor 5.4}.

Thus by using LES we have 
\begin{center}
$H^1(Z(t_{1}s_{1}s_{2}, (2,3,4,1,2)), T_{(t_{1}s_{1}s_{2}, (2,3,4,1,2))})=0.$
\end{center}
Thus combining all we have $H^1(Z(w_{0},(2,3,4,\underline{j_{5}})),T_{(w_{0},(2,3,4,\underline{j_{5}}))})=0.$

{\bf Case 5:} $c=s_{4}s_{3}s_{1}s_{2}.$ In this case we have $w_{0}=s_{4}s_{3}s_{4}w_{4}s_{3}s_{1}s_{2}.$ Let $t_{2}=s_{4}s_{3}s_{4}.$
Since $s_{3}$ commutes with $s_{1},$ we have $H^i(t_{2}w_{4}s_{3}s_{1}, \alpha_{1})=H^i(t_{2}w_{3}s_{3}s_{4}s_{2}s_{1}s_{2}, \alpha_{1})=0$ for $i\ge 0$ (see Lemma \ref{lemma1.3}(4)).

Thus by using LES, Corollary \ref{cor 5.10}(4) and \cite[Corollary 5.6, p.778]{Ka} we have 
\begin{center}
$H^1(Z(w_{0}, (4,3,4,\underline{i_{4}},3,1,2)), T_{(w_{0}, (4,3,4,\underline{i_{4}},3,1,2))})=H^1(Z(t_{2}w_{4},(4,3,4,\underline{i_{4}})),T_{(t_{2}w_{4},(4,3,4,\underline{i_{4}}))}).$
\end{center}

By using LES, Corollary \ref{cor 5.5}, Corollary \ref{cor 5.10}(4), and \cite[Corollary 5.6, p.778]{Ka} we have 
\begin{center}
$H^1(Z(t_{2}w_{4},(4,3,4,\underline{i_{4}})),T_{(t_{2}w_{4},(4,3,4,\underline{i_{4}}))})=H^1(Z(t_{2}w_{3}, (4,3,4,\underline{i_{3}})), T_{(t_{2}w_{3},(4,3,4,\underline{i_{3}}))}).$
\end{center}

By using LES, Corollary \ref{cor 5.5}, Corollary \ref{cor 5.10}(4), and \cite[Corollary 5.6, p.778]{Ka} we have 
\begin{center}
$H^1(Z(t_{2}w_{3}, (4,3,4,\underline{i_{3}})), T_{(t_{2}w_{3},(4,3,4,\underline{i_{3}}))})=H^1(Z(t_{2}w_{2}, (4,3,4, \underline{i_{2}})), T_{(t_{2}w_{2}, (4,3,4,\underline{i_{2}}))}).$
\end{center}
By using LES and Lemma \ref{lem 6.6}(1) we have 
\begin{center}
$H^1(Z(t_{2}w_{2}, (4,3,4,\underline{i_{2}})), T_{(t_{2}w_{2},(4,3,4,\underline{i_{2}}))})=H^1(Z(t_{2}\tau_{2}, (4,3,4,\underline{j_{2}})),T_{(t_{2}\tau_{2},(4,3,4,\underline{j_{2}}))}).$ 
\end{center}
By using LES, Corollary \ref{cor 5.10}(4), and \cite[Corollary 5.6, p.778]{Ka} we have 
\begin{center}
$H^1(Z(t_{2}\tau_{2}, (4,3,4,\underline{j_{2}})), T_{(t_{1}\tau_{2},(4,3,4,\underline{j_{2}}))})=H^1(Z(t_{2}w_{1}, (4,3,4,\underline{i_{1}})), T_{(t_{2}w_{1}, (4,3,4,\underline{i_{1}}))}).$
\end{center}
By using LES, Lemma \ref{lem 6.6}(2) we have 
\begin{center}
$H^1(Z(t_{2}w_{1}, (4,3,4,\underline{i_{1}})), T_{(t_{2}w_{1},(4,3,4,\underline{i_{1}}))})=H^1(Z(t_{2}\tau_{1},(4,3,4,\underline{j_{1}})),T_{(t_{2}\tau_{1},(4,3,4,\underline{j_{1}}))}).$ 
\end{center}
By using LES, Corollary \ref{cor 5.10}(4) and \cite[Corollary 5.6, p.778]{Ka} we have
\begin{center}
$H^1(Z(t_{2}\tau_{1}, (4,3,4,\underline{j_{1}})), T_{(t_{2}\tau_{1},(4,3,4,\underline{j_{1}}))})=H^1(Z(t_{2}s_{1}s_{2}, (4,3,4,1,2)), T_{(t_{2}s_{1}s_{2}, (4,3,4,1,2))}).$ 
\end{center}
We see that $H^1(s_{4}s_{3}, \alpha_{3})=0,H^1(t_{2}, \alpha_{4})=0$ by  \cite[Corollary 5.6, p.778]{Ka}.
Since $s_{3}, s_{4}$ commutes with $s_{1}$ we have $H^1(t_{2}s_{1}, \alpha_{1})=H^1(s_{1}, \alpha_{1})=0.$ By Corollary \ref{cor 5.5} we have $H^1(t_{2}s_{1}s_{2}, \alpha_{2})=0.$ 

Thus by using LES we have $H^1(Z(t_{2}s_{1}s_{2},(4,3,4,1,2)), T_{(t_{2}s_{1}s_{2}, (4,3,4,1,2))})=0.$ Thus combining all we have $H^1(Z(w_{0},(4,3,4,\underline{i_{4}},3,1,2)),T_{(w_{0},(4,3,4,\underline{i_{4}},3,1,2))})=0.$

{\bf Case 6:} $c=s_{4}s_{2}s_{3}s_{1}.$ In this case we have $w_{0}=s_{4}s_{2}s_{3}s_{4}w_{4}s_{3}s_{1}.$ Let $t_{3}=s_{4}s_{2}s_{3}s_{4}.$ By using LES and \cite[Corollary 5.6, p.778]{Ka} we have 

$$H^1(Z(w_{0}, (4,2,3,4,\underline{i_{4}},3,1)), T_{(w_{0}, (4,2,3,4,\underline{i_{4}},3,1))})=H^1(Z(t_{3}w_{4},(4,2,3,4,\underline{i_{4}})),T_{(t_{3}w_{4},(4,2,3,4,\underline{i_{4}}))}).$$

By using LES, Corollary \ref{cor 5.6}, Corollary \ref{cor 5.10}(5), and \cite[Corollary 5.6, p.778]{Ka} we have 
\begin{center}
$H^1(Z(t_{3}w_{4},(4,2,3,4,\underline{i_{4}})),T_{(t_{3}w_{4},(4,2,3,4,\underline{i_{4}}))})=H^1(Z(t_{3}w_{3},(4,2,3,4,\underline{i_{3}})),T_{(t_{3}w_{3},(4,2,3,4,\underline{i_{3}}))}).$
\end{center}
By using LES, Corollary \ref{cor 5.6}, Corollary \ref{cor 5.10}(5), and \cite[Corollary 5.6, p.778]{Ka} we have 
\begin{center}
$H^1(Z(t_{3}w_{3},(4,2,3,4,\underline{i_{3}})),T_{(t_{3}w_{3},(4,2,3,4,\underline{i_{3}}))}=H^1(Z(t_{3}w_{2},(4,2,3,4,\underline{i_{2}})),T_{(t_{3}w_{2},(4,2,3,4,\underline{i_{2}}))}).$
\end{center}
By using LES and Lemma \ref{lem 6.7}(1) we have 
\begin{center}
$H^1(Z(t_{3}w_{2}, (4,2,3,4,\underline{i_{2}})),T_{(t_{3}w_{2},(4,2,3,4,\underline{i_{2}}))})=H^1(Z(t_{3}\tau_{2},(4,2,3,4,\underline{j_{2}})),T_{(t_{3}\tau_{2},(4,2,3,4,\underline{j_{2}}))}).$ 
\end{center}
By using Corollary \ref{cor 5.10}(5) and \cite[Corollary 5.6, p.778]{Ka} we have 
\begin{center}
$H^1(Z(t_{3}\tau_{2},(4,2,3,4,\underline{j_{2}})),T_{(t_{3}\tau_{2},(4,2,3,4,\underline{j_{2}}))})=H^1(Z(t_{3}w_{1},(4,2,3,4,\underline{i_{1}})),T_{(t_{3}w_{1},(4,2,3,4,\underline{i_{1}}))}).$ 
\end{center}

By using LES, Lemma \ref{lem 6.7}(2) we have 
\begin{center}
$H^1(Z(t_{3}w_{1},(4,2,3,4,\underline{i_{1}})),T_{(t_{3}w_{1},(4,2,3,4,\underline{i_{1}}))})=H^1(Z(t_{3}\tau_{1}, (4,2,3,4,\underline{j_{1}})),T_{(t_{3}\tau_{1},(4,2,3,4,\underline{j_{1}}))}).$ 
\end{center}
By using LES, Corollary \ref{cor 5.10}(5), and \cite[Corollary 5.6, p.778]{Ka} we have

$$H^1(Z(t_{3}\tau_{1},(4,2,3,4,\underline{j_{1}})),T_{(t_{3}\tau_{1},(4,2,3,4,\underline{j_{1}}))})=H^1(Z(t_{3}s_{1}s_{2},(4,2,3,4,1,2)),T_{(t_{3}s_{1}s_{2},(4,2,3,4,1,2))}).$$ 

We see that $H^1(s_{4}s_{2}, \alpha_{2})=0,$ $H^1(t_{3}s_{1}, \alpha_{1})=0.$ Further, by using \cite[Corollary 5.6, p.778]{Ka} we have $H^1(s_{4}s_{2}s_{3}, \alpha_{3})=0,H^1(t_{3}, \alpha_{4})=0.$ By  Corollary \ref{cor 5.6} we have $H^1(t_{3}s_{1}s_{2}, \alpha_{2})=0.$ 

Therefore by using LES we have $H^1(Z(t_{3}s_{1}s_{2}, (4,2,3,4,1,2)), T_{(t_{3}s_{1}s_{2}, (4,2,3,4,1,2))})=0.$ Thus combining all we have $H^1(Z(w_{0},(4,2,3,4,\underline{i_{4}},3,1)),T_{(w_{0},(4,2,3,4,\underline{i_{4}},3,1))})=0.$

{\bf Case 7:} $c=s_{4}s_{3}s_{2}s_{1}.$ In this case we have 
$w_{0}=s_{4}s_{3}s_{4}s_{2}s_{3}s_{4}w_{3}s_{3}s_{1}s_{2}s_{1}.$
Let $t_{4}=s_{4}s_{3}s_{4}s_{2}s_{3}s_{4}.$ Let $\underline{i'}=(4,3,4,2,3,4,\underline{l_{3}},1,2,1).$ Recall that $l_{r}=(i_{r}, 3).$
Let $\underline{i'_{r}}=(4,3,4,2,3,4,\underline{i_{r}})$ be the reduced expressions of $t_{4}w_{r}$ for $r=1,2,3.$
Let $\underline{j'_{r}}=(4,3,4,2,3,4,\underline{j_{r}})$ be the reduced expressions of $t_{4}\tau_{r}$ for $r=1,2,3.$
Let $\underline{j'}=(4,3,4,2,3,4,1)$ be the reduced expression of $t_{4}s_{1}.$

By Lemma \ref{lem 4.1}(2) and Corollary \ref{cor 5.2}(2) we have $H^i(s_{4}w_{4}, \alpha_{2})=0$ for $i\ge0.$ 
Since $s_{4}$ commutes with $s_{1},s_{2},$ we have $H^i(s_{4}w_{4}, \alpha_{4})$=$H^i(s_{4}w_{3}s_{3}s_{4}s_{1}s_{2}, \alpha_{2})$=$H^i(s_{4}w_{3}s_{3}s_{1}s_{2}, \alpha_{2})$ for $i\ge0.$ Thus we have $H^i(s_{4}w_{3}s_{3}s_{1}s_{2}, \alpha_{2})=0$ for $i\ge0.$ Therefore by using SES we have $H^i(t_{4}w_{3}s_{3}s_{1}s_{2}, \alpha_{2})=0$ for $i\ge 0.$ Since $s_{3}$ commutes with $s_{1}$ we have $H^i(t_{4}w_{3}s_{3}s_{1}, \alpha_{1})=H^i(t_{4}w_{3}s_{1}, \alpha_{1})$ for $i\ge 0.$
$H^i(t_{4}w_{3}s_{1}, \alpha_{1})=H^i(t_{4}[1, 4]^3s_{2}s_{1}s_{2}, \alpha_{1})=0$ for $i\ge 0$ (see Lemma \ref{lemma1.3}(4)). Thus we have $H^i(t_{4}w_{3}s_{3}s_{1}, \alpha_{1})=0$ for $i\ge 0.$
Thus by using LES, above discussion, and \cite[Corollary 5.6, p.778]{Ka} we have 
\begin{center}
$H^1(Z(w_{0},\underline{i'}),T_{(w_{0},\underline{i'})})=H^1(Z(t_{4}w_{3},\underline{i'_{3}}),T_{(t_{4}w_{3},\underline{i'_{3}})}).$
\end{center}

By using LES, Lemma \ref{cor 5.8}(4), Corollary \ref{cor 5.10}(6), and \cite[Corollary 5.6, p.778]{Ka} we have 
\begin{center}
$H^1(Z(t_{4}w_{3},\underline{i'_{3}}), T_{(t_{4}w_{3},\underline{i'_{3}})})=H^1(Z(t_{4}w_{2},\underline{i'_{2}}),T_{(t_{4}w_{2},\underline{i'_{2}})}).$
\end{center}

By using LES, Lemma \ref{cor 5.8}(4), Corollary \ref{cor 5.10}(6), and \cite[Corollary 5.6, p.778]{Ka} we have 
\begin{center}
$H^1(Z(t_{4}w_{2},\underline{i'_{2}}),T_{(t_{4}w_{2},\underline{i'_{2}})})=H^1(Z(t_{4}w_{1},\underline{i'_{1}}), T_{(t_{4}w_{1},\underline{i'_{1}})}).$
\end{center}

By using LES and Lemma \ref{lem case 4}(1) we have 
\begin{center}
$H^1(Z(t_{4}w_{1},\underline{i'_{1}}),T_{(t_{1}w_{1},\underline{i'_{1}})})=H^1(Z(t_{4}\tau_{1},\underline{j'_{1}}),T_{(t_{4}\tau_{1},\underline{j'_{1}})}).$ 
\end{center}

By using LES, Corollary \ref{cor 5.10}(6), and \cite[Corollary 5.6, p.778]{Ka} we have
\begin{center}
$H^1(Z(t_{4}\tau_{1},\underline{j'_{1}}),T_{(t_{4}\tau_{1},\underline{j'_{1}})})=H^1(Z(t_{4}s_{1}s_{2}, (\underline{j'},2)),T_{(t_{4}s_{1}s_{2}, (\underline{j'},2))}).$
\end{center} 

By using LES and Lemma \ref{lem case 4}(2) we have 
\begin{center}
$H^1(Z(t_{4}s_{1}s_{2}, (\underline{j'},2)), T_{(t_{4}s_{1}s_{2},(\underline{j'},2))})=H^1(Z(t_{4}s_{1}, \underline{j'}), T_{(t_{4}s_{1}, \underline{j'})}).$  
\end{center}

By \cite[Corollary 5.6, p.778]{Ka} we see that $H^1(s_{4}s_{3}, \alpha_{3})=0,$ $H^1(s_{4}s_{3}s_{4}, \alpha_{4})=0,$ 

$H^1(s_{4}s_{3}s_{4}s_{2}s_{3}, \alpha_{3})=0,$ and $H^1(t_{4}, \alpha_{4})=0.$ By Lemma \ref{cor 5.8}(1) we have $H^1(s_{4}s_{3}s_{4}s_{2}, \alpha_{2})=0.$ 

Since $s_{3},s_{4}$ commute with $s_{1},$ we have  $H^1(t_{4}s_{1}, \alpha_{1})=H^1(s_{4}s_{3}s_{2}s_{1}, \alpha_{1}).$ It is easy to see by using SES that $H^1(s_{4}s_{3}s_{2}s_{1}, \alpha_{1})=0.$ Thus we have $H^1(t_{4}s_{1}, \alpha_{1})=0.$
Therefore by using LES we have $H^1(Z(t_{4}s_{1}, \underline{j'} ), T_{(t_{4}s_{1}, \underline{j'})})=0.$ Thus combining all we have $H^1(Z(w_{0},\underline{i'}),T_{(w_{0},\underline{i'})})=0.$
\end{proof}
\begin{cor}
Let $c$ be a Coxeter element such that $c$ is of the form $[a_{1}, 4][a_{2}, a_{1}-1]\cdots[a_{k}, a_{k-1}-1]$ with $a_{1}\neq 3$ or $a_{2}\neq2$ and $a_{k}=1.$
Let $(w_{0}, \underline{i})$ be a reduced expression of $w_{0}$ in terms of $c$ as in Theorem \ref {theorem 8.1}. Then, $Z(w_{0}, \underline{i})$ has no deformations.
\end{cor}
\begin{proof}
By Theorem \ref {theorem 8.1} and by \cite[Proposition 3.1, p.673]{CKP}, we have $H^i(Z(w_{0}, \underline{i}), T_{(w_{0}, \underline{i})})=0$ for all $i>0.$ Hence, by \cite[Proposition 6.2.10, p.272]{Huy}, we see that $Z(w_{0}, \underline{i})$ has no deformations.
\end{proof}
	
\section{non rigidity for $G_{2}$}
Now onwards we will assume that $G$ is of type $G_{2}.$ 
Note that the longest element $w_{0}$ of the Weyl group $W$ of $G$ is equal to $-identity.$ We recall the following proposition from \cite[Proposition 1.3, p.858]{YZ}. We use Proposition \ref{Z} and the notation as in \cite{YZ} to deduce the following:
\begin{lem}\label{lem 8.1}
Let $c \in W$ be a Coxeter element. Then, we have
\begin{enumerate}
\item [(1)] $w_{0}=c^{3}.$
			
\item[(2)] For any sequence $\underline{i}=(\underline{i}^{1}, \underline{i}^{2},\underline{i}^{3})$ of reduced expressions of $c;$ the sequence $\underline{i}=(\underline{i}^{1}, \underline{i}^{2},\underline{i}^{3})$ is a reduced expression of $w_{0}.$ 
\end{enumerate}
\end{lem}
	
\begin{proof}
Proof of (1): Let $\eta : S\longrightarrow S$ be the involution of $S$ defined by $i\rightarrow i^*,$ where $i^*$ is given by $\omega_{i^*} = -w_{0}(\omega_{i}).$ Since $G$ is of type $G_{2},$  $w_{0} = -identity.$ Therefore, we have $i=i^*$ for every $i.$ Let $h$ be the Coxeter number. By \cite[Proposition 1.7]{YZ}, we have $h(i,c)+ h(i^*,c)= h.$ Since $h= 2|R^+|/2$ (see \cite[Proposition 3.18]{HUM}) and $i=i^*,$ we have $h(i,c)= h/2=3,$ as $|R^+|=6.$ 
By Proposition \ref{Z}, we have $c^6(\omega_{i})= -\omega_{i}$ for all $i=1,2.$ Since $\{\omega_{i} : i=1,2\}$ forms an  $\mathbb{R}$-basis of $X(T)\otimes \mathbb{R},$ it follows that $c^3 = -identity.$ Hence, we have $w_{0}=c^3.$
The assertion (2) follows from the fact that $l(c)= 2$ and $l(w_{0})=|R^+|=6.$ (see \cite[p.66, Table 1]{Hum1}).
\end{proof}
		
Let $c$ be a coxeter element of $W.$ Then $c=s_{1}s_{2}$ or $c=s_{2}s_{1}.$ Then from Lemma \ref{lem 8.1} we have $w_{0}=s_{1}s_{2}s_{1}s_{2}s_{1}s_{2},$ or $w_{0}=s_{2}s_{1}s_{2}s_{1}s_{2}s_{1}$ according as $c=s_{1}s_{2}$ or $c=s_{2}s_{1}.$

Let $\underline{i_{1}}$ (repectively, $\underline{i_{2}}$) be the the reduced expression of
$w_{0}=s_{1}s_{2}s_{1}s_{2}s_{1}s_{2}$ (respectively, $w_{0}=s_{2}s_{1}s_{2}s_{1}s_{2}s_{1}$). Then we have 

\begin{thm}\label{th 8.2}
$H^1(Z(w_{0},\underline{i_{r}}), T_{(w_{0}, \underline{i_{r}})})\neq0$ for $r=1,2.$ 
\end{thm}	
\begin{proof}
Let $c=s_{1}s_{2}.$
Let $\underline{i}=(1,2)$ be the sequence corresponding to $c.$ Then using LES, we have: 
	
$$0\longrightarrow H^0(c,\alpha_{2})  \longrightarrow H^0(Z(c,\underline{i}), T_{(c,\underline{i})})\longrightarrow H^0(s_{1}, \alpha_{1})\longrightarrow$$
$$ H^1(c,\alpha_{2})\xrightarrow g H^1(Z(c,\underline{i}), T_{(c,\underline{i})})\longrightarrow 0.$$

By using SES, we see that $H^1(s_{1}s_{2}, \alpha_{2})=\mathbb{C}_{\alpha_{2}+\alpha_{1}}\oplus \mathbb{C}_{\alpha_{2}+ 2\alpha_{1}}.$ Now $H^0(s_{1}, \alpha_{1})_{\alpha_{2} + \alpha_{1}}=0.$ Hence $g$ is a non zero homomorphism.
Hence $H^1(Z(c, \underline{i}), T_{(c,\underline{i})}))\neq 0.$ 
By Lemma \ref {lemma 6.1}, the natural homomorphism 
	
$$H^1(Z(w_{0}, \underline{i_{1}}))\longrightarrow H^1(Z(c,\underline{i}), T_{(c, \underline{i})})$$ is surjective.

Hence we have $H^1(Z(w_{0}, \underline{i_{1}}), T_{(w_{0}, \underline{i_{1}})})\neq 0.$

Let $c=s_{2}s_{1},$ $u=s_{2}s_{1}s_{2}$. 
Let $\underline{j}=(2,1,2)$ be the sequence corresponding to $u.$ Then using LES, we have: 
	
$$0\longrightarrow H^0(u,\alpha_{2})  \longrightarrow H^0(Z(u,\underline{j}), T_{(u,\underline{j})})\longrightarrow H^0(Z(s_{2}s_{1}, (2,1)), T_{(s_{2}s_{1},(2,1))})\longrightarrow$$
$$ H^1(u,\alpha_{2})\xrightarrow h H^1(Z(u,\underline{j}), T_{(u,\underline{j})})\longrightarrow H^1(Z(s_{2}s_{1}, (2,1)), T_{(s_{2}s_{1},(2,1))})\rightarrow0.$$
	
We see that $H^1(u,\alpha_{2})=\mathbb{C}_{\alpha_{1}} \oplus  \mathbb{C}_{\alpha_{2} + \alpha_{1}} \oplus \mathbb{C}_{\alpha_{2} + 2\alpha_{1}},$  $H^0(s_{1}, \alpha_{1})_{\alpha_{2} + \alpha_{1}}=0,$ and $H^0(s_{2}s_{1}, \alpha_{1})_{\alpha_{2} + \alpha_{1}}=0.$
	
Therefore by LES we have $H^0(Z(s_{2}s_{1}, (2,1)), T_{(s_{2}s_{1},(2,1))})_{\alpha_{2} + \alpha_{1}}=0.$ Hence $h$ is a non zero homomorphism.
Hence $H^1(Z(u, \underline{j}), T_{(u,\underline{j})}))\neq 0.$ 
By Lemma \ref {lemma 6.1}, the natural homomorphism 
$$H^1(Z(w_{0}, \underline{i_{2}}))\longrightarrow H^1(Z(u,\underline{j}), T_{(u, \underline{j})})$$ is surjective.

Hence we have $H^1(Z(w_{0}, \underline{i_{2}}), T_{(w_{0}, \underline{i_{2}})})\neq 0.$
\end{proof}

{\bf Acknowledgements} The authors would like to thank to the Infosys Foundation for the partial financial support.

\end{document}